\documentclass[11pt,reqno]{amsart}
\usepackage{amsmath, amssymb, amsthm, esint, verbatim, hyperref}
\usepackage{times}
\usepackage{dsfont}

\numberwithin{equation}{section}

\hypersetup{
  pdftitle={Quantitative Reifenberg Theorem for Measures},
  pdfauthor={Nicholas Edelen and Aaron Naber and Daniele Valtorta},
  pdfsubject={},
  pdfkeywords={},
  pdfpagelayout=SinglePage,
  pdfpagemode=UseOutlines,
  colorlinks,
  bookmarksopen,
  linkcolor=[rgb]{0,0,0.7},
  urlcolor=[rgb]{0,0,0.4},
  citecolor=[rgb]{0.4,0.1,0},
}

\newtheorem{theorem}{Theorem}[section]
\newtheorem{ctheorem}[theorem]{Conjectural Theorem}
\newtheorem{proposition}[theorem]{Proposition}
\newtheorem{observation}[theorem]{Observation}
\newtheorem{prop}[theorem]{Proposition}
\newtheorem{lemma}[theorem]{Lemma}
\newtheorem{corollary}[theorem]{Corollary}
\theoremstyle{definition}
\newtheorem{definition}[theorem]{Definition}
\theoremstyle{remark}
\newtheorem{remark}[theorem]{Remark}
\theoremstyle{remark}
\newtheorem{example}[theorem]{Example}
\theoremstyle{remark}
\newtheorem{note}[theorem]{Note}
\theoremstyle{remark}
\newtheorem{question}[theorem]{Question}
\theoremstyle{remark}
\newtheorem{conjecture}[theorem]{Conjecture}

\newcommand{\dist}{\mathrm{dist}}
\newcommand{\graph}{\mathrm{graph}}
\newcommand{\dom}{\mathrm{dom}}
\newcommand{\spt}{\mathrm{spt}}
\newcommand{\haus}{\mathcal{H}}
\newcommand{\leb}{\mathcal{L}}
\newcommand{\cC}{\mathcal{C}}
\newcommand{\cS}{\mathcal{S}}
\newcommand{\cG}{\mathcal{G}}
\newcommand{\cB}{\mathcal{B}}
\newcommand{\cK}{\mathcal{K}}
\newcommand{\cQ}{\mathcal{Q}}
\newcommand{\cT}{\mathcal{T}}
\newcommand{\cU}{\mathcal{U}}
\newcommand{\cF}{\mathcal{F}}

\newcommand{\proj}{\mathrm{proj}}
\newcommand{\setm}{\setminus}

\newcommand{\eps}{\epsilon}

\newcommand{\R}{\mathds{R}}

\newcommand{\dR}{\mathds{R}}

\newcommand{\dZ}{\mathds{Z}}

\newcommand{\norm}[1]{\left\|#1\right\|}

\newcommand{\ton}[1]{\left(#1\right)}

\newcommand{\cur}[1]{\left\{#1\right\}}
\newcommand{\abs}[1]{\left|#1\right|}

\newcommand{\B}[2]{B_{#1}\ton{#2}}

\title{Quantitative Reifenberg Theorem for Measures}
\author{Nick Edelen}  \address{Nick Edelen - University of Notre Dame (USA)}\email{nedelen@nd.edu}
\author{Aaron Naber}  \address{Aaron Naber - Institute for Advanced Study - Princeton (USA)} \email{anaber@ias.edu}
\author{Daniele Valtorta}  \address{Daniele Valtorta - University of Milano-Bicocca (EU)} \email{daniele.valtorta@unimib.it}

\thanks{The first author was supported by NSF grants DMS-1606492 and DMS-220430, the second author has been supported by NSF grant DMS-1406259, the third author has been supported by SNSF grant 200021\_159403/1}
\date{\today}

\begin{document}
\begin{abstract}

We study generalizations of Reifenberg's Theorem for measures in $\R^n$ under assumptions on the Jones' $\beta$-numbers, which appropriately measure how close the support is to being contained in a subspace.  Our main results, which hold for general measures without density assumptions
, give effective measure bounds on $\mu$ away from a closed $k$-rectifiable set with bounded Hausdorff measure.  We show examples to see the sharpness of our results.  Under further density assumptions one can translate this into a global measure bound and $k$-rectifiable structure for $\mu$.  Applications include quantitative Reifenberg theorems on sets and discrete measures, as well as upper Ahlfor's regularity estimates on measures which satisfy $\beta$-number estimates on all scales.

\end{abstract}
\maketitle

\maketitle
\tableofcontents

\section{Introduction}

In his famous work on the solution of Plateau's problem \cite{reif_orig}, Reifenberg proved that if a closed set $S \subset \R^n$ is sufficiently well approximated by $k$-dimensional planes at all scales and points in the set, then it is $C^{0,\alpha}$-bi-H\"older to a disk.  The exponent $\alpha$ can be made arbitrarily close to $1$, so long as the approximation is sufficiently close.  Specifically, Reifenberg considered a two-sided closeness condition on $S$:
\begin{equation}\label{eqn:two-sided-reif}
\inf_{V^k} r^{-1} d_H(S \cap B_r(x), V \cap B_r(x)) \leq \delta(n,\alpha), \quad \forall x \in S \cap B_1, r \leq 8,
\end{equation}
where the infimum is taken over all affine $k$-planes.  Given \eqref{eqn:two-sided-reif} Reifenberg concludes that $S \cap B^n_1$ is bi-H\"older to $B^k_1$.  Sets satisfying \eqref{eqn:two-sided-reif} are often called Reifenberg flat.

We remark that \eqref{eqn:two-sided-reif} has a strong uniform connotation. First of all, it rules out holes in $S$.  Indeed, under \eqref{eqn:two-sided-reif} $S$ is topologically a disk, not just a closed subset of a disk. As a consequence, \eqref{eqn:two-sided-reif} automatically endows $S$ with lower mass bounds, that is, a set $S$ satisfying \eqref{eqn:two-sided-reif} is ``lower $k$-Ahlfors regular''.  However \eqref{eqn:two-sided-reif} cannot guarantee an upper mass bound (``upper Ahlfors regularity''), since we cannot take $\alpha = 1$ in Reifenberg's theorem, except in the trivial case $\delta=0$.  See Example \ref{example:koch} for more on this. 

A second consequence of \eqref{eqn:two-sided-reif} is that the set $S$ has no ``excess set'', meaning that at all points and at all scales $S$ is uniformly contained in a small tubular neighborhood around a $k$-dimensional plane. In other words, $S$ cannot have parts of small measure (or even of zero measure) far away from its approximating plane at all points and at all scales.

This paper is motivated by the question: how can we control the mass and structure of $S$ if we allow for holes and excess sets?  This question has begun playing a role in various settings, in particular in the analysis of singular sets of nonlinear PDE's, see \cite{naber-valtorta:harmonic}.

Especially in recent years, various papers dealing with this question have appeared in literature. In these papers, a need arose for \emph{one-sided} notions of closeness with an integral flavor, which in particular allows for holes and excess set.  Natural quantities that satisfy these requirements were introduced by Jones in \cite{jones}. They are the so-called $\beta$-numbers for a nonnegative Borel measure $\mu$:
\begin{align}
\beta^k_{\mu, p}(x, r) &:= \left( \inf_{V^k} r^{-k} \int_{B_r(x)} r^{-p} d(z, V)^p d\mu(z) \right)^{1/p} \quad \text{for} \ 1\leq p <\infty \notag\\
\beta^k_{\mu, \infty}(x, r) &:= \inf_{V^k} \left( \inf \{ \delta : \spt\mu \cap B_r(x) \subset B_{\delta r}(V) \} \right)
\end{align}
where the infima are taken over all affine $k$-planes $V^k$.  See Definition \ref{d:beta_numbers} for a slight generalization, which will be used in this paper.

Toro \cite{toro:reifenberg} and David-Toro \cite{davidtoro} considered several generalizations of Reifenberg's Theorem.  They showed, for example, that if a closed set $S$ satisfies the uniform Reifenberg \eqref{eqn:two-sided-reif} and
\begin{gather}
\int_0^1 \beta^k_{\haus^k\llcorner S, p}(x, r)^2 dr/r \leq M \quad \forall x \in S,
\end{gather}
with $p=1$ or $\infty$, then the parameterization of Reifenberg's Theorem is bi-Lipschitz.  The reason it suffices to assume the $\beta$-numbers are square-summable is captured in the Koch snowflake of Example \ref{example:koch}. 

David-Toro \cite{davidtoro} also proved that if a collection of points is in some discrete sense Reifenberg flat, then one can find a bi-H\"{o}lder mapping containing all the points.  Assuming further a kind of ``discrete $\beta$-number'' summability condition, one can take the mapping to be bi-Lipschitz.  Because \cite{davidtoro} parameterize the entire net of points by a disk, a topological assumption in the spirit of small-tilting is necessary for their results (see \cite[counterexample 12.4]{davidtoro}; think thin twisted bands attached to an annulus).

If one is not interested in topological information, but only in control over the mass and rectifiability, it is reasonable to think that bounds on the $\beta_p$ ($p < \infty$) will be sufficient.  The intuition is that we should only need to parameterize the region of $S$ near the $L^p$-best planes; the ``excess set'' away from the best-planes already has controlled mass.

The works of Tolsa \cite{tolsa:jones-rect} and Azzam-Tolsa \cite{azzam-tolsa} verify part of this intuition.  They prove that a set $S$ is $k$ rectifiable if and only if
\begin{gather}
\int_0^1 \beta^k_{\haus^k \llcorner S, 2}(x, r)^2 dr/r < \infty \quad \text{$(\haus^k \llcorner S)$-a.e. $x$.}
\end{gather}
In fact their Theorem holds for general Radon measures, assuming $\mu$-a.e. positivity and finiteness of the upper density.

In this paper we shall prove effective bounds for the total measure and mass under a Dini-type assumption on the $\beta_2$, without any Ahlfors-regularity or Reifenberg-flat assumptions.  In fact, we shall consider general nonnegative Borel-regular measures, without any ($\sigma$-)finiteness or density assumptions, and prove a combination of measure estimates on $\mu$ and Minkowski estimates on its support.  As a corollary we recover the rectifiability theorem of \cite{azzam-tolsa}.

Our complete list of theorems for the paper is rather long and at times technical, since it is important for the applications to consider the estimates in a good deal of generality.  We therefore wait until Section \ref{sec:main_thm} to give a complete list of main theorems.  In preparation we list here a few corollaries of these main results which illustrate what type of behavior one might expect.  We begin with the following, which holds for a general Borel-regular measure without density assumptions:

\begin{theorem}\label{thm:teaser1}
Let $\mu$ be a nonnegative Borel-regular measure supported in $B^n_1(0)$, and $\eps \geq 0$.  Suppose
\begin{equation}\label{eqn:teaser1-hyp}
\mu \left\{ z \in B_1(0) : \int_0^2 \beta^k_{\mu, 2}(z, r)^2 \frac{dr}{r} > M \right\} \leq \Gamma\, .
\end{equation}
Then there is a closed, $k$-rectifiable set $\cK \subset B_1(0)$ with $\haus^k(\cK)<c(n)$, so that we have
\begin{gather}
|B_r(\cK)| \leq c(n) r^{n-k} \quad \forall r \leq 1, \quad \text{and} \quad \mu(B_r(x)) \geq \eps r^k \quad \forall x \in \cK, r \leq 1 ,
\end{gather}
and
\begin{gather}
\mu(B_1(0) \setminus \cK) \leq c(n)(\eps + M) + \Gamma.
\end{gather}
\end{theorem}

Here $|A|$ denotes the $n$-dimensional Lebesgue measure of $A$, and $B_r(A)$ the $r$-tubular neighborhood of $A$.  In this general case when one does not assume upper or lower density bounds, effective estimates must be split into two pieces: a $\mu$-noncollapsed set $\cK$ for which there exists effective Minkowski estimates (in particular $\cK$ has finite $k$-dimensional Hausdorff measure), and a set $B_1\setminus \cK$ which has uniformly bounded $\mu$-measure.  See Example \ref{ex:sharp} to see that this setup is the sharp one.\\  

Assuming additionally a bound on the \emph{lower-density} allows us to prove apriori upper bounds on the measure of $\mu$.
\begin{corollary}\label{cor:teaser1}
If $\mu$ is a non-negative Borel-regular measure satisfying the hypothesis \eqref{eqn:teaser1-hyp}, and with
\begin{gather}
\Theta_*^k(\mu, x) \leq b \quad \text{for $\mu$-a.e. $x \in B_1$},
\end{gather}
then
\begin{gather}
\mu(B_1) \leq c(n)(b + M) + \Gamma.
\end{gather}
\end{corollary}

A key consequence of Corollary \ref{cor:teaser1} is proving upper-Ahlfors-regularity from some kind of $\beta$-number condition.  In certain Theorems (e.g. \cite{girela}) this allows us to remove the assumption of upper-Ahlfors-regularity.  We also bring the readers attention to \cite{miskiewicz}, who has recently and independently obtained a similar result to Theorem \ref{thm:teaser2}, assuming upper and lower bounds on the upper-density.
\begin{theorem}\label{thm:teaser2}
Suppose $\mu$ is a nonnegative Borel-regular measure supported in $B_1$ with the property
\begin{gather}
\Theta^{k}_*(\mu, x) \leq b\quad \text{for $\mu$-a.e. $x \in B_1$},
\end{gather}
and satisfying \emph{one} of the following conditions:
\begin{enumerate}
\item[A)] For $\mu$-a.e. $x \in B_1$, 
\begin{gather}
\int_0^2 \beta_{\mu,2}(x, r)^2 \frac{dr}{r} \leq M.
\end{gather}

\item[B)] For $\mu$-a.e. $x \in B_1$, and every $0 < r \leq 1$, 
\begin{gather}
\int_{B_r(x)} \int_0^{2r} \beta_{\mu,2}(z, s)^2 \frac{ds}{s} d\mu(z) \leq M^2 r^k
\end{gather}

\item[C)] $\mu(B_1) < \infty$, and for $\mu$-a.e. $x \in B_1$, and every $0 < r \leq 1$,
\begin{gather}
\int_{B_r(x)} \int_0^{2r} \beta_{\mu, 2}(z, s)^2 \frac{ds}{s} d\mu(z) \leq M \mu(B_r(x)).
\end{gather}
\end{enumerate}

Then for every $x \in B_1$, $0 < r \leq 1$, we have
\begin{gather}
\mu(B_r(x)) \leq c(n) (b + M) r^k.
\end{gather}

Moreover, if $\Theta^{*, k}(\mu, x) > 0$ $\mu$-a.e., then $\mu$ is $k$-rectifiable.
\end{theorem}

When $\mu = \haus^k \llcorner S$, where $\haus^k$ is the $k$-dimensional Hausdorff measure, any one of the conditions A)-C) of Theorem \ref{thm:teaser2} implies $\mu$ is $\sigma$-finite, and therefore we obtain as an immediate Corollary:
\begin{corollary}\label{cor:teaser2}
If $\mu = \haus^k \llcorner S$ satisfies any of the conditions A)-C) of Theorem \ref{thm:teaser2}, then $S$ is $k$-rectifiable, and for every $x \in B_1$, $0 < r \leq 1$, we have
\begin{gather}
\haus^k(S \cap B_r(x)) \leq c(n)(1 + M) r^k.
\end{gather}
\end{corollary}

The corollary is similar in spirit to \cite[section 3]{AS}, where the authors use a different notion of $\beta$-numbers based on a Choqet integral in terms of Hausdorff content and require lower regularity of the set $S$.

Summability conditions like \eqref{eqn:teaser1-hyp} have arisen in several interesting scenarios.  The $\beta$-numbers have long been a bridge between rectifiability and boundedness of Calderon-Zygmund operators.  Indeed for Ahlfors-David-regular sets $S$, David-Semmes \cite{david-semmes} show that uniform rectifiability is equivalent to an $L^1$-bound of the form
\begin{gather}
\int_{B_r(x) \cap S} \int_{0}^r \beta^k_{\haus^k\llcorner S, 2}(z, s)^2 \frac{ds}{s} d\haus^k(z) \leq M \mu(B_r(x)) \quad \forall x, r.
\end{gather}

Also, in the past 20 years there has been significant research on boundedness of Calderon-Zygmund operators on sets or measures satisfying only an upper-Ahlfors-regularity condition (e.g. \cite{david:vanishing}, \cite{tolsa:l2-boundedness}, \cite{nazarov-tolsa-volberg:cauchy}).  Very recently \cite{girela} (see also \cite{azzam-tolsa} Theorem 1.4, and \cite{jaye-nazarov-tolsa}) has shown that if a Radon measure $\mu$ is upper-Ahlfors-regular, and satisfies
\begin{gather}\label{eq:girela-hyp}
\int_{B_r(x)} \int_0^r \beta^k_{\mu, 2}(z, s)^2 \frac{\mu(B_s(z))}{s^k} \frac{ds}{s} d\mu(z) \leq M \mu(B_r(x)) \quad \forall x, r, 
\end{gather}
then Calderon-Zygmund operators are bounded in $L^2(\mu)$.

Combining this result with Theorem \ref{thm:teaser2}, we obtain

\begin{corollary}[Theorem \ref{thm:teaser2} $+$ \cite{girela}]\label{cor:teaser3}
Let $\mu$ be a nonnegative Borel-regular measure with the property that for $\mu$-a.e. $x \in B_1$, and every $0 < r \leq 1$, we have
\begin{gather}
\Theta^{k}_*(\mu, x) \leq b, \quad \int_{B_r(x)} \int_0^{2r} \beta^k_{\mu,2}(z, r)^2 \frac{dr}{r} d\mu(z) \leq M \mu(B_r(x))  .
\end{gather}
Then Calderon-Zygmund operators are bounded in $L^2(\mu)$.
\end{corollary}

Another important application can be found in the analysis of singular sets.  In \cite{naber-valtorta:harmonic}, \cite{naber-valtorta:varifold} the authors use a similar (though less general) technique to prove uniform mass bounds and rectifiability for the singular set of nonlinear harmonic maps and integral currents with bounded mean curvature.

In fact our Theorems generalize the discrete and rectifiable Reifenberg Theorems of \cite{naber-valtorta:harmonic}, which prove mass bounds for discrete measures assuming small $\beta$-bounds at all scales.  Our Main Theorem \ref{thm:main} works for arbitrary Borel-regular measures, relaxes the smallness assumption, allows the user to specify the scale of packing, and most importantly requires only weak $\beta$-bounds at a \emph{single scale}.

Note that there is a growing literature on $\beta$-numbers and problems related to rectifiability. For example, \cite{Lerman,BS1,BS2} proves necessary and sufficient condition for rectifiability of one dimensional curves, and \cite{AS} focuses on the traveling salesman problem.
\vspace{5mm}

Our proofs are based on a generalized corona decomposition for the support of $\mu$. 
Starting from $\B 1 0$, we cover $\spt \mu$ with smaller and smaller balls, where at each step we decrease the radia of our covering by a factor of $\rho$. For each ball in the covering, we consider the $k$-dimensional plane achieving the minimum in the definition of $\beta$, i.e., the plane minimizing the $L^2(\mu)$ norm of the distance. The part of $\spt \mu$ away from this best approximating plane $V$ will be counted as ``excess set'', and the estimates on $\beta$ imply that this part has uniformly bounded measure. The part close to $V$ will be covered by balls satisfying a Vitali condition. Each of these balls is classified into ``good'' or ``bad'' according to how spread the support of $\mu$ is inside them. In particular, good balls are the ones where the support of $\mu$ effectively spans $V$, while bad balls are those where $\mu$ is concentrated near some lower dimensional subspace. 

On good balls, we will be able to exploit the bounds on the $\beta$ numbers to obtain nice estimates on the tilting of the best planes from one scale to the next. This will allow us to use a Reifenberg-type construction to create a sequence of smooth manifolds $T_i$ approximating $\spt \mu$ over good balls. These manifolds will have uniform bi-Lipschitz bounds. 

On bad balls, we will lose the ability to refine this Reifenberg-type construction. However, on these balls the support of $\mu$ is almost contained in a $k-1$ dimensional space, making it easy to obtain uniform $k$-dimensional estimates on it.

Thus we will be able to decompose $\mu$ into a piece with controlled measure, and a piece supported on a set with uniform $k$-dimensional packing estimates. By assuming an upper bound over the (lower) density of $\mu$, we will be able to turn the $k$-dimensional estimates on the support into mass estimates on $\mu$ itself. Lower bounds on the (upper) density of $\mu$ will allow us to conclude the $k$-rectifiability of this measure.


\textbf{Acknowledgments}
The authors thank Max Engelstein and Benjamin Jaye for leading us to reference \cite{girela}, and warmly thank Tatiana Toro for several useful conversations.  The first author expresses his deepest gratitude to Leon Simon for his insight and inspiration.

\vspace{.5cm}

\section{Main Theorems}\label{sec:main_thm}

In this section, we introduce the relevant definitions in order to state in a clean fashion our main theorems.  
We begin by slightly generalizing the standard definition of $\beta_2$, by allowing $\beta_2$ to only be zero on balls with small $\mu$-measure (as in the results this is all one needs).

\begin{definition}\label{d:beta_numbers}
Let $\mu$ be a Borel measure in $\R^n$, and $\bar\eps_\beta \geq 0$.  The \emph{$k$-dimensional truncated $\beta$-number of $\mu$} in $B_r(x)$ is defined to be
\begin{gather}
\beta^k_{\mu,2,\bar\epsilon_\beta}(x, r)^2 = \left\{ \begin{array}{l l} \inf_{V^k} r^{-k-2} \int_{B_r(x)} d(z, V)^2 d\mu(z) & \mu(B_r(x)) > \bar\eps_\beta r^k \\
0 & \mu(B_r(x)) \leq \bar\eps_\beta r^k . \end{array} \right. ,
\end{gather}
where we take the infimum over affine $k$-planes $V^k$.  We write $V^k_\mu(x, r)$ for any of the affine $k$-planes achieving $\beta^k_{2,\mu}(x, r)$.  
\end{definition}

\begin{remark}
 In all the theorems in this paper, we will assume that $\bar \eps_\beta$ is smaller than a constant depending on the parameters of the specific theorem. This means that we do not require any control over the measure $\mu$ on balls with small mass. This can be helpful in some applications. 
 
 Since this assumption is not common, we remark for the sake of clarity that with $\bar \eps_\beta=0$ all the statements of this paper are evidently still valid.
\end{remark}

Often we will simply drop subscripts when no confusion arises, for instance we may write $\beta_2$ in place of $\beta^k_{2,\mu,\bar\epsilon_\beta}$ and $V$ in place of $V^k_\mu$.  
We make some elementary remarks on distortion:

\begin{remark}[Monotonicity of $\beta_2$]\label{rem:pointwise-integral}
Note the following monotonicity type properties:
\begin{enumerate}
\item If $B_r(y) \subset B_R(x)$, and we know $\mu B_R(x) > \bar\epsilon_\beta R^k$, then $\beta_2(y, r)^2 \leq (R/r)^{k+2}\beta_2(x, R)^2$.

\item If $|x - y| < r$, and $\mu \ton{B_{2r}(y)} > 2^k \bar\epsilon_\beta r^k$, then $\beta_2(x, r)^2 \leq 2^{k+2} \beta_2(y, 2r)^2$.

\item Lastly, if $\spt \mu \subset B_1$, and $\mu B_1 > \bar\eps_\beta 2^k$, then for any $x \in B_1$ we have $\beta_2(x, 16)^2 \leq c(k) \beta_2(x, 2)^2$.  
\end{enumerate}
\end{remark}

\begin{remark}[Scale-invariance of $\beta_2$]
A trivial but important property to note is the following: if $\mu_{x,r}(A) := r^{-k} \mu(x + rA)$, then $\beta^k_{\mu_{x, r}, 2}(y, s) = \beta^k_{\mu, 2}(x + ry, rs)$.
\end{remark}

Now our main theorem will be quite general in nature, something which naturally involves more technical details than would be liked in an introduction, thus let us take a moment to preface it.  Besides allowing $\mu$ to be a general Borel-regular measure, it may be in practice that one only has $\beta$-number estimates on part of $\mu$, or only has estimates until a given fixed scale, which may vary from point to point in $\mu$, for instance in the case of discrete measures.  Therefore, associated to our estimates we must consider a covering of the measure $\mu$ which tells us where the estimates hold.  The following notation gives us a consistent way of dealing with this:

\begin{definition}\label{def:mu_C}
Let $\cC$ be a closed subset of $\dR^n$, and $r_x : \cC \to \R_+$ a radius function, we say $(\cC, r_x)$ is a \emph{covering pair} and we decompose $\cC = \cC_+ \cup \cC_0$, with $r_x > 0$ on $\cC_+$ and $r_x = 0$ on $\cC_0$.  Given a measure $\mu$ on $\dR^n$, a covering pair $(\cC, r_x)$ is called a \emph{covering pair for $\mu$} if $\cC \supseteq \spt\mu$.

If $\cC' \subset \cC$, we write $\cC'_+ = \cC' \cap \cC_+$ and $\cC'_0 = \cC' \cap \cC_0$.  For shorthand we will sometimes write
\begin{gather}
B_{r_x}(\cC') := \cC'_0 \cup \bigcup_{x \in \cC'_+} B_{r_x}(x).
\end{gather}

Associated to any $\cC'$ we have the packing measure
\begin{gather}
\mu_{\cC'} := \haus^k \llcorner \cC'_0 + \sum_{x \in \cC'_+} r_x^k \delta_x.
\end{gather}
By $\delta_x$ we mean the Dirac delta centered at $x$.
\end{definition}

\vspace{.25cm}
We are now ready to state our main theorems. Aside from the principal results, which are stated in a very general fashion, we prove many corollaries under different additional assumptions. For the reader's convenience, we split these results into different subsections. 
\subsection{Main theorems}\label{ss:main_thm}

Our main technical Theorem is the following, which is stated with a good deal of generality.  In particular, besides only assuming our measure $\mu$ is Borel-regular, without additional density assumptions, we only assume weak summability estimates on the $\beta$ numbers, as opposed to the more standard assumption of a.e. control or $L^1$ control.
In subsequent subsections we will add additional hypotheses to obtain improved results, however almost all will follow in rather short order from the following: 

\begin{theorem}[Main Theorem]\label{thm:main}
There is a constant $c(n)$ so that the following holds: Let $\eps \geq 0$, and $\mu$ be a non-negative Borel-regular measure in $\R^n$ supported in $B_1(0)$, and $(\cC, r_x)$ a covering pair for $\mu$ with $r_x \leq 1$.  Suppose 
\begin{equation}\label{eq_main_integral_bounds}
\mu \cur{ z \in \B 1 0: \int_{r_z}^2 \beta^k_{\mu,2,\bar\eps_\beta}(z, r)^2\frac{dr}{r} > M } \leq \Gamma ,
\end{equation}
and that $\bar\eps_\beta \leq c(k)\max\{ M, \eps\}$.  Then there is a closed subset $\cC' \subset \cC$ which admits the following properties:
\begin{enumerate}
\item[A)] Packing bounds:
\begin{gather}\label{eq_C_+'_est}
\haus^k(\cC'_0) + \sum_{x \in \cC'_+} r_x^k \leq c(n).
\end{gather}
In fact, we have the Minkowski type estimate
\begin{gather}\label{eq_mink_est}
r^{k-n} |B_r(\cC')| \leq c(n) \quad  \forall \,0 < r \leq 1\, .
\end{gather}

\item[B)] Measure control away from $\cC'$:
\begin{gather}\label{eq_main_measure_away}
\mu \ton{ B_1 \setminus B_{r_x}(\cC') } \leq c(n)(\eps + M) + \Gamma\, .
\end{gather}

\item[C)] A noncollapsing condition on $\cC'$:
\begin{gather}\label{eq_main_noncoll}
\mu(B_r(x)) \geq \max\{M, \eps\} r^k \quad \forall x \in \cC'\, , \quad \forall r_x \leq r \leq 1\, .
\end{gather}

\item[D)] Fine-scale packing structure: $\cC'_0$ is closed, rectifiable, and admits the upper bounds
\begin{gather}
\haus^k(\cC'_0 \cap B_r(x)) \leq c(n) r^k \quad \forall x \in B_1, 0 < r \leq 1.
\end{gather}


\end{enumerate}
\end{theorem}

\begin{remark}
Theorem \ref{thm:main} has a scale-breaking quality to it: the hypothesis and measure estimates B), C) scale with the measure, but the packing estimate does not. Note also that in this theorem we make no assumption on the density of $\mu$.
\end{remark}

\begin{remark}
The $\cC'_0$ is actually a kind of ``uniformly rectifiable,'' in the sense that whenever $\haus^k(\cC'_0 \cap B_r(x)) > 2^{-k-1}\omega_k r^k$, there is some bi-Lipschitz $k$-manifold $T$ so that
\begin{gather}
\haus^k(\cC'_0 \cap T) \geq \frac{1}{2} \haus^k(\cC'_0 \cap T \cap B_r(x)).
\end{gather}
\end{remark}

\begin{remark}
If \eqref{eq_main_integral_bounds} holds for some $\cC$ that is not closed, then conclusions A)-D) hold for the closure $\overline{\cC'} = \cC'_+ \cup \overline{\cC'_0}$ (this equality is due to Lemma \ref{lem:original-packing}).
\end{remark}

\begin{remark}
In fact one can prove upper-Ahlfors-regularity of the packing measure $\mu_{\cC'}$ (on the scales $r\in (r_x,1]$), either by using Theorem \ref{thm:main_upperahlfors} or by analyzing the tree structures more closely.
\end{remark}

Let us also remark in the above that though we estimate all of $\mu$ on the ball, it is necessary to break $\mu$ into two complimentary pieces and give apriori fundamentally different types of estimates on each piece.  Indeed, in $(A)$ we look at the part of $\mu$ covered by $\cC'$ and prove packing estimates on the support of $\mu$, which give little control over $\mu$ itself.  On the other hand, we show in $(B)$ that on the remaining part of $\mu$ a uniform bound on the measure itself.  Packing estimates and measure estimates are not related to one another for a general $\mu$, though we will see in the next subsection how to convert between the two under additional lower or upper density estimates.

\subsection{Mass Bounds with Density}\label{ss:main_thm_density}
The previous theorem works under essentially no assumptions on $\mu$ other than the weak $\beta_2$ bound. As expected, we obtain better control over $\mu$ by assuming also some control over its $k$-dimensional density. In particular, let us now see that assuming either upper or lower density bounds on $\mu$ allows us to apply Theorem \ref{thm:main} in order to conclude mass bounds.\\

 For the applications, in particular for a generalization of the discrete Reifenberg of \cite{naber-valtorta:harmonic}, we will again want a version which is with respect to a general covering.  In this case it is unnecessary, and indeed unnatural given the purpose of such a covering, to make density assumptions on $\mu$ below the scale being considered.  Thus let us consider the lower and upper density of $\mu$ with respect to a covering pair $(\cC,r_x)$ of $\mu$ given by
\begin{align}\label{eq_deph_density}
&\Theta^{*,k}(\mu,\cC,x)\equiv \limsup_{r\to r_x} \frac{\mu(B_{ r/50}(x))}{(r/50)^k}\, ,\notag\\
&\Theta^{k}_*(\mu,\cC,x)\equiv \liminf_{r\to r_x} \frac{\mu(B_{r/50}(x))}{(r/50)^k}\, .
\end{align}
Of course, if $\cC=\cC_0$ then this is the usual upper and lower density, while if $\cC=\cC_+$ and $\mu$ is finite then $\Theta^{*,k}(\mu,\cC,x)=\Theta^{k}_*(\mu,\cC,x) = (r_x/50)^{-k}\mu(\overline B_{r_x/50}(x))$. 

Recall that if $\mu=\haus^k\llcorner S$, and $\haus^k(S) < \infty$, then for $\haus^k$-a.e. $s\in S$:
\begin{gather}\label{eq_haus_dens}
 2^{-k}\leq \Theta^{*,k}(\haus^k\llcorner S,x)\leq 1\, .
\end{gather}

Let us now see how to derive global mass bounds under density assumptions on our measure:

\begin{theorem}[Mass Bounds under Density]\label{thm:main_density}
Let $\mu$ be a non-negative Borel-regular measure in $B_1$ and $(\cC,r_x)$ a covering pair for $\mu$ such that $r_x \leq 1$.  Suppose
\begin{equation}\label{eq_massbound_integralest}
\mu \cur{ z \in B_1 : \int_{r_z}^2 \beta^k_{\mu,2,\bar\eps_\beta}(z, r)^2\frac{dr}{r}> M } \leq \Gamma ,
\end{equation}

Then we have:
\begin{enumerate}
\item[L)] If $\Theta^{*,k}(\mu,\cC, x) \geq a > 0$ for $\mu$-a.e. $x$, then there exists $\cC' \subseteq \cC$ with $\mu(B_1 \setminus B_{r_x}(\cC')) = 0$ so that
\begin{gather}
\haus^k(\cC_0')+\sum_{\cC'_+} r_x^k \leq  \frac{c(n)(\bar\eps_\beta + M + \Gamma)}{a} + c(n)\, .
\end{gather}
\item[U)] If $\Theta^{k}_*(\mu,\cC, x) \leq b$ for $\mu$-a.e. $x\in \cC$, then 
\begin{gather}
\mu(B_1) \leq c(n)(\bar\eps_\beta + b + M) + \Gamma.
\end{gather}
\end{enumerate}
\end{theorem}

Note that, since \eqref{eq_massbound_integralest} is only an assumption at one scale, one cannot hope to obtain upper Ahlfors regularity for $\mu$.  However, in Theorem \ref{thm:main_upperahlfors} we will see that if one assumes \eqref{eq_massbound_integralest} on all points and scales, then one may obtain uniform upper Ahlfors regularity on the measure.  We refer to Examples \ref{ex_density} and \ref{ex_packing} to see that the results of Theorem \ref{thm:main_density} are sharp, in that the improvement from Theorem \ref{thm:main} to Theorem \ref{thm:main_density} may fail without the corresponding density bounds.\\

As in the main theorem, our primary applications of the above are either in the case when $\cC_+=\emptyset$, so we are controlling $\mu$ on all points and scales, or in the discrete case $\cC_0=\emptyset$.  Let us record our statements in these cases: 

\begin{corollary}[Mass Bounds under Density]\label{cor:main_density_1}
Let $\mu$ be a non-negative Borel-regular measure in $B_1$, and $(\cC, r_x)$ a covering pair for $\mu$ with $\cC = \cC_0$.  Suppose
\begin{equation}
\mu \cur{ z \in B_1 : \int_{0}^2 \beta^k_{\mu,2,\bar\eps_\beta}(z, r)^2\frac{dr}{r}> M } \leq \Gamma \, .
\end{equation}
Then we have
\begin{enumerate}
\item[L)] If $\Theta^{*,k}(\mu, x) \geq a > 0$ for $\mu$-a.e. $x$, then $\exists$ $\cC' = \cC'_0 \subseteq \spt\mu$ with $\mu(B_1 \setminus \cC') = 0$ s.t. 
 \begin{gather}
  \haus^k(\cC'_0) \leq  \frac{c(n)(\bar\eps_\beta + M + \Gamma)}{a} + c(n)\, . 
 \end{gather}
\item[U)] If $\Theta^{k}_*(\mu, x) \leq b$ for $\mu$-a.e. $x$, then $\mu(B_1) \leq c(n)(\bar\eps_\beta + b + M) + \Gamma.$
\end{enumerate}
\end{corollary}

Additionally, in the case of a discrete cover (i.e. $\cC_0=\emptyset$) we have the following:

\begin{corollary}[Packing Estimates under Density]\label{cor:main_density_2}
Let $\mu$ be a non-negative Borel-regular measure supported in $B_1$, and $(\cC, r_x)$ a covering pair for $\mu$ with $\cC = \cC_+$.  Suppose
\begin{equation}
\mu \cur{ z \in B_1 : \int_{r_z}^2 \beta^k_{\mu,2,\bar\eps_\beta}(z, r)^2\frac{dr}{r}> M } \leq \Gamma .
\end{equation}
Then we have
\begin{enumerate}
\item[L)] If $\mu(B_{r_x}(x)) \geq a r_x^k$, then $\exists$ $\cC' = \cC'_+ \subseteq \spt\mu$ with $\mu(B_1 \setminus B_{r_x}(\cC')) = 0$ s.t.
\begin{gather}
\sum_{x \in \cC'_+} r_x^k \leq  \frac{c(n)(\bar\eps_\beta + M + \Gamma)}{a} + c(n)\, .
\end{gather}
\item[U)] If $\mu(B_{r_x}(x)) \leq b r_x^k$, then $\mu(B_1) \leq c(n)(\bar\eps_\beta + b + M) + \Gamma .$
\end{enumerate}
\end{corollary}
\vspace{.5cm}

\subsection{\texorpdfstring{$\beta_2$}{beta2}-Estimates on all Scales}\label{ss:main_thm_allscales}

In this subsection we study what improvements to the main results are obtained when the $\beta_2$-estimates are assumed to exist on all scales.  Under a density bound, either lower or upper, one is able to conclude upper Ahlfors regularity estimates.  Our most general result is the following:

\begin{theorem}[Upper Ahlfor's Regularity]\label{thm:main_upperahlfors}
Take $\mu$ a non-negative Borel-regular measure in $B_1$ and $(\cC,r_x)$ a covering pair for $\mu$  such that $r_x \leq 1$, and for all $x\in \cC$ and all $r > 0$ we have
\begin{equation}\label{eqn:upperahlfors-hyp}
\mu \cur{ z \in B_r(x) : \int_{r_z}^{2r} \beta^k_{\mu,2,\bar\eps_\beta}(z, s)^2\frac{ds}{s}> M } \leq M r^k\, ,
\end{equation}
Then we have
\begin{enumerate}
\item[L)] If $\Theta^{*,k}(\mu,\cC, x) \geq a > 0$ for $\mu$-a.e. $x\in \cC$, then there is a $\cC' \subset \cC$ with $\mu(B_1 \setminus B_{r_x}(\cC')) = 0$, satisfying:
\begin{gather}
\haus^k(\cC'_0\cap B_r(x))+\sum_{z \in \cC'_+\cap B_r(x)} r_z^k \leq  \ton{\frac{c(n)(\bar\eps_\beta + M)}{a} + c(n)}r^k \quad \forall x\in \cC', r \geq r_x.
\end{gather}
\item[U)] If $\Theta^{k}_*(\mu,\cC, x) \leq b$ for $\mu$-a.e. $x\in \cC$, then there is a $\cC' \subset \cC$, with $B_{r_x}(\cC') \supset \cC$, and satisfying:
\begin{gather}\label{eq_upper_Ahl}
\mu(B_r(x)) \leq c(n)(\bar\eps_\beta + b + M)r^k \quad \forall x \in \cC', r \geq r_x.
\end{gather}
\end{enumerate}
\end{theorem}
\vspace{.25cm}

Despite appearances this theorem is not a straightforward corollary of Theorem \ref{thm:main_density}.  The issue is $\cC_+$: in L), we want a \emph{single} choice of balls which admit packing at all scales.  
When $\cC_+ = \emptyset$ then the proof becomes much easier.

In the spirit of the previous results let us again study the above in the cases when $\cC_+=\emptyset$ or in the discrete case when $\cC_0=\emptyset$.  In the case $\cC_+=\emptyset$ we have produced a generalization of the rectifiable Reifenberg from \cite{naber-valtorta:harmonic}, while in the discrete case $\cC_0=\emptyset$ we have produced a generalization of the discrete Reifenberg from \cite{naber-valtorta:harmonic}.  We will analyze this a bit more in the next subsection.  Let us begin with the $\cC_+=\emptyset$ case:

\begin{corollary}[Upper Ahlfor's Regularity]\label{cor:main_upperahlfors_1}
Let $\mu$ be a non-negative Borel-regular measure in $B_1$, and $(\cC, r_x)$ a covering pair for $\mu$ with $\cC = \cC_0$ and $r_x \leq 1$.  Suppose that for all $x$ and $r > 0$:
\begin{equation}
\mu \cur{ z \in B_r(x) : \int_{0}^{2r} \beta^k_{\mu,2,\bar\eps_\beta}(z, s)^2\frac{ds}{s}> M } \leq M r^k\, ,
\end{equation}
Then:
\begin{enumerate}
\item[L)] If $\Theta^{*,k}(\mu, x) \geq a$ for $\mu$-a.e. $x$, then there is a $\cC' = \cC'_0 \subset \cC$ with $\mu(B_1\setminus \cC') = 0$, so that
 \begin{gather}
  \haus^k(\cC'\cap B_r(x))\leq  \ton{\frac{c(n)(\bar\eps_\beta + M)}{a} + c(n)}r^k \quad \forall x \in B_1, r > 0 \, .
 \end{gather}
\item[U)] If $\Theta^{k}_*(\mu, x) \leq b$ for $\mu$-a.e. $x$, then
 \begin{gather}
  \mu(B_r(x)) \leq c(n)(\bar\eps_\beta + b + M)r^k \quad \forall x \in B_1, r > 0 \, .
 \end{gather}
\end{enumerate}
\end{corollary}
\vspace{.25cm}

Finally let us end with the following discrete case when $\cC_0=\emptyset$:

\begin{corollary}[Upper Ahlfor's Regularity]\label{cor:main_upperahlfors_2}
Let $\mu$ be a non-negative Borel-regular measure supported in $B_1$, and $(\cC, r_x)$ a covering pair for $\mu$ with $\cC = \cC_+$ and $r_x \leq 1$.  Suppose that for all $x$ and $r > 0$ we have
\begin{equation}
\mu \cur{ z \in B_r(x) : \int_{r_z}^{2r} \beta^k_{\mu,2,\bar\eps_\beta}(z, s)^2\frac{ds}{s}> M } \leq M r^k\, ,
\end{equation}
Then:
\begin{enumerate}
\item[L)] If $\mu(B_{r_x}(x)) \geq a r_x^k$, then there is a $\cC' \subset \cC$ with $\mu(B_1 \setminus B_{r_x}(\cC')) = 0$, and satisfying
\begin{gather}
\sum_{z \in \cC'_+\cap B_r(x)} r_z^k \leq  \ton{\frac{c(n)(\bar\eps_\beta + M)}{a} + c(n)}r^k \quad \forall x \in \cC', r \geq r_x.
\end{gather}
\item[U)] If $\mu(B_{r_x}(x)) \leq b\, r_x^k$, then there is a $\cC' \subset \cC$ with $B_{r_x}(\cC') \supset \cC$, and satisfying
\begin{gather}
\mu(B_r(x)) \leq c(n)(\bar\eps_\beta + b + M)r^k \quad \forall x \in \cC', r \geq r_x.
\end{gather}
\end{enumerate}
\end{corollary}
\vspace{.25cm}

\subsection{Applications.}\label{section:applications}

In this subsection we record a variety of applications of the previous results.  In some of these applications the distinction between rectifiability of a set and of a measure is important.  Let us be explicit: we say a Borel measure $\mu$ on $\dR^n$ is $k$-rectifiable if $\mu \ll \haus^k$, and $\mu(\dR^n \setminus \cK) = 0$ for some $k$-rectifiable set $\cK$.


Our first application is to the structure of measures which satisfy an $L^1$ $\beta$-estimate at some scale.  Under the additional hypothesis of lower and upper finiteness on the upper density it was shown in \cite{azzam-tolsa} that such measures are rectifiable. A decomposition similar in spirit to the one obtained here has been proven in \cite{BS2} for the case $k=1$. Here we prove a sharp generalization of these results which does not require any density assumptions.  Let us begin with an effective version:

\begin{theorem}[Structure under $\beta_2$-control]\label{thm:main-thm-L1}
Let $\mu$ be a non-negative Borel-regular measure supported in $B_1$, and $\eps > 0$.  Suppose that
\begin{gather}
\int_{B_1} \int_0^2 \beta^k_{\mu, 2, 0}(z, r)^2 \frac{dr}{r} d\mu(z) \leq M^2
\end{gather}
Then we can decompose $\mu = \mu_h + \mu_\ell + \mu_0$, so that:
\begin{enumerate}
\item[A)] $\mu_h(B_1 \setminus \cK_h) = 0$ for some $k$-rectifiable set $\cK_h$, with volume bounds
\begin{gather}\label{eqn:structure-ahlfors-reg}
\haus^k(\cK_h \cap B_r(x)) \leq c(n) r^k \quad \forall x \in B_1, 0 < r \leq 1,
\end{gather}
and such that $\haus^k \llcorner \cK_h \leq c(k)\eps^{-1} \mu$ on \emph{open} sets.

\item[B)] $\mu_\ell$ is $k$-rectifiable, with mass bounds $\mu_\ell(B_1) \leq c(n)(\eps + M)$.

\item[C)] $\mu_0(B_1) \leq c(n)(\eps + M)$, and $\Theta^{*, k}(\mu, x) = 0$ for $\mu_0$-a.e. $x$.
\end{enumerate}
\end{theorem}

The $\mu_0$ piece of the decomposition is important in the above, see Example \ref{ex_density} for an illustration of this.  If we assume $\beta$-control at all scales, we obtain a similar decomposition, which satisfies an upper-Ahlfors-regularity condition on $\mu_\ell, \mu_0$.  Notice again we assume nothing on the density.

\begin{corollary}[$\beta_2$-control at all scales]\label{cor:main-thm-L1-scales}
Let $\mu$ be a non-negative Borel-regular measure supported in $B_1$, and $\eps > 0$.  Suppose that
\begin{gather}
\int_{B_r(x)} \int_0^{2r} \beta^k_{\mu,2,0}(z, r)^2 \frac{dr}{r} d\mu(z) \leq M^2 r^k \quad \forall x \in B_1, 0 < r \leq 1.
\end{gather}
Then we can decompose $\mu = \mu_h + \mu_\ell + \mu_0$ in the same manner as Theorem \ref{thm:main-thm-L1}, but which instead of \eqref{eqn:structure-ahlfors-reg} admits the volume measure bounds:
\begin{gather}
\haus^k(\cK_h \cap B_r(x)) \leq c(n)(1 + M/\eps) r^k\quad \text{and}\quad (\mu_\ell + \mu_0)(B_r(x)) \leq c(n)(\eps + M) r^k
\end{gather}
for all $x \in B_1$, and $0 < r \leq 1$.
\end{corollary}

Assuming lower and upper bounds on the upper density, and finiteness of $\int_0^1 \beta_2(x,r)^2 \frac {dr}{r}$ for $\mu$-almost all $x$, it was shown in \cite{azzam-tolsa} that such measures are rectifiable.  Here, as a Corollary to Theorem \ref{thm:main-thm-L1} we prove a generalization of this result where we only require lower bounds on the upper density to prove rectifiability of the support of $\mu$, and upper bounds on the lower density to prove rectifiability for the whole measure $\mu$. 

\begin{theorem}[compare to \cite{azzam-tolsa}]\label{thm:main_rect}
Let $\mu$ be a non-negative Borel-regular measure in $B_1$ such that for $\mu$-a.e. $x$
\begin{equation}
\int_{0}^{2} \beta^k_{\mu,2,0}(x, s)^2\frac{ds}{s} <\infty, \quad \text{ and } \quad \Theta^{*, k}(\mu, x) > 0\, .
\end{equation}
Then there is a $k$-rectifiable set $\cK_0$ so that $\mu(B_1 \setminus \cK_0) = 0$.  Moreover, if $\Theta_*^k(\mu, x)<\infty$ for $\mu$-a.e. $x$, then $\mu$ is $k$-rectifiable.
\end{theorem}
\vspace{.25cm}

In another direction, applications of these techniques arise in \cite{naber-valtorta:harmonic}, where it was proved that the singular sets of some nonlinear equations are rectifiable with uniform measure bounds.  The key connection between singular sets and the type of theory presented here are the rectifiable Reifenberg and discrete Reifenberg theorems of \cite{naber-valtorta:harmonic}.  The techniques of the current paper not only reproduce these, but allow us to prove generalizations which are substantially better in several regards.  To discuss this let us state our generalized discrete Reifenberg. 

\begin{theorem}[Discrete Reifenberg]\label{thm:main_discrete}
Let $\{B_{r_i}(x_i)\}\subseteq B_1$ be a disjoint collection of balls, and consider a generalized packing measure $\mu \equiv \sum \mu_i r_i^{k}\, \delta_{x_i}$. Let us suppose we have the estimate
\begin{equation}\label{eq_discrete_integral_bounds}
\mu \cur{ x_i \in B_1 : \int_{r_i}^2 \beta^k_{\mu,2,\bar\eps_\beta}(x_i, r)^2\frac{dr}{r}> M } \leq \Gamma .
\end{equation}
Then the following hold:
\begin{enumerate}
\item[L)] If $\mu_i \geq a > 0$, then we have $\sum r_i^k \leq \frac{c(n)(\bar\eps_\beta + M + \Gamma)}{a} + c(n)$.
\item[U)] If $\mu_i \leq b$, then we have $\mu(B_1)\leq c(n)(b+\bar\eps_\beta + M) + \Gamma$.
\end{enumerate}
\end{theorem}

The above generalizes the discrete Reifenberg of \cite{naber-valtorta:harmonic} in several ways, for instance it applies to a general class of discrete measures and only requires a bound on the $\beta$-estimates, not smallness.  The most fundamental improvement however, is that one only needs to assume the $\beta$-estimates of \eqref{eq_discrete_integral_bounds} on the first scale in order to conclude the packing estimates.  In \cite{naber-valtorta:harmonic}, the discrete Reifenberg requires $\beta$-estimates analogous to \eqref{eq_discrete_integral_bounds} on {\it every} point and scale in order to conclude the desired mass bound.

We want to point out that in a recent and interesting work \cite{miskiewicz} has independently obtained a generalization of discrete Reifenberg, which also only requires boundedness of the $\beta$-estimates (as opposed to smallness), but still requires $L^1$-estimates at all points and scales.

In a similar spirit, let us consider a generalization of the rectifiable Reifenberg of \cite{naber-valtorta:harmonic}.  In \cite{naber-valtorta:harmonic}, it is shown that if a set $S$ satisfies the correct $\beta$-estimates on {\it every} point and scale, then $S$ is rectifiable with uniform mass bounds.  In the same vein as our generalized discrete Reifenberg, to obtain the mass bounds we may weaken this assumption substantially by only requiring $\beta$-estimates on the first scale:

\begin{theorem}[Mass Bounds for Sets]\label{thm:main_hausdorff_massbounds}
Take $S\subseteq B_1$, and assume we have the estimate
\begin{equation}\label{eq_hausdorff_integral_bounds}
\haus^k\cur{ x \in S : \int_{0}^2 \beta^k_{\haus^k \llcorner S,2,\bar\eps_\beta}(x, r)^2\frac{dr}{r}> M } \leq \Gamma .
\end{equation}
Then we have that $\haus^k(S)\leq c(n)(1 + \bar\eps_\beta + M) + \Gamma$ .
\end{theorem}
\begin{remark}
It is worth noting that the under the above assumption $S$ need not be rectifiable.  It follows from Theorem \ref{thm:main} that all but some definite amount must be, but to prove all of $S$ is rectifiable requires assuming a true weak-$L^1$ assumption.
\end{remark}

\subsection{Outline of proof}\label{ss:proof_outline}

Let us establish some basic intuition about what information $\beta$-numbers can and cannot give us.

First, it is instructive to point out that the $n$-Lebesgue measure $\leb^n$ has summably finite $\beta^k$-number at every point:
\begin{gather}
\int_0^1 \beta^k_{\leb^n, 2}(x, r)^2 \frac{dr}{r} \leq c(n).
\end{gather}
This is calculated in Example \ref{ex_density}.  This says that finiteness or even smallness of the summed $\beta$-number does not by itself give any kind of $k$-dimensional structure.  In particular, given a condition like \eqref{eq_main_integral_bounds}, without any extra information, we cannot hope to have a $k$-dimensional packing or Hausdorff estimate away from arbitrarily small $\mu$-measure.  In other words, the $M$ in the measure estimate \eqref{eq_main_measure_away} is necessary.

Moreover, as illustrated in Example \eqref{ex_packing}, by approximating the $n$-Lebesgue measure by packed $k$-spheres, even if we are dealing with a nice $k$-dimensional manifold we cannot hope to improve the $k$-dimensional Hausdorff estimates we obtain to stronger Minkowski or packing estimates.  The problem in this case is that we have no information about the scale at which things start to look $k$-dimensional.  Thus we see there is some subtlety as to what estimates exist, and it is important to split our general estimates into two types: Hausdorff estimates on one piece and measure estimates on the other.  

\vspace{2.5mm}

However, let us come to terms with this, and attempt prove a packing and measure estimate as in Theorem \ref{thm:main}.  A basic motivating observation is the that $\beta_2$-numbers control mass at any fixed distance from the $L^2$-best planes.  For example, if we write $V^k$ for the $L^2$-best plane of $B_1(0)$ which realizes $\beta^k_{\mu, 2}(0, 1)$, then
\begin{gather}\label{eqn:example-excess}
\mu (B_1(0) \setminus B_\rho(V)) \leq \rho^{-2} \beta(0, 1)^2 \, .
\end{gather}
In other words, the region of $B_1(0)$ with $\mu$-mass much larger than $\beta(0, 1)^2$ is concentrated in the neighborhood $B_\rho(V^k)$.

\eqref{eqn:example-excess} suggests the following naive strategy.  Take a Vitali cover of $B_\rho(V^k)$ by balls of size $\rho$, and thereby obtain a $k$-dimensional packing estimate on high-mass region in $B_1(0)$.  In each ball $B_{\rho}(x)$, we can repeat the same process with balls of size $\rho^2$, obtaining a $k$-packing estimate on the high-mass region of $B_{\rho}(x)$.  Then repeat in each $B_{\rho^2}$, and then in each $B_{\rho^3}$, etc...

Unfortunately, no matter how small we fix $\rho$, attempting to iterate this will result in an exponential error, due to double-counting on the ball-overlaps.  If one attempts to use disjoint cubes instead of balls, an exponential error arises for a different reason: the $k$-volume of a $k$-plane in the cube varies with orientation.

In order to obtain a finite $k$-dimensional packing estimate, we require either some kind of \emph{global memory} as we progress in scale, or the knowledge that big mass regions look like a \emph{very small portion} of a $k$-plane.  It turns out these two options give a dichotomy: in any ball, the region with big mass either looks ``very $k$-dimensional,'' giving us enough tilting control to construct a global memory; or it looks ``at most $(k-1)$-dimensional,'' giving us very good packing estimates to compensate for double-counting.  This dichotomy is captured in Lemma \ref{lem:ball-dichotomy}.

We call balls satisfying the first condition of being ``very $k$-dimensional'' \emph{good}, and those satisfying the second condition of being ``at most $(k-1)$-dimensional''  \emph{bad}.  See Definition \ref{def:good-bad-ball} for a precise formulation.  It is worth emphasizing that our notion of a {\it good} ball, as compared to similar constructions in previous results (e.g. \cite{naber-valtorta:harmonic}), is much more general.  Our ability to technically deal with such {\it good} balls without additional measure assumptions is a key observation for the paper.

\vspace{2.5mm}

Now let us assume we have a condition like
\begin{gather}
\int_0^1 \beta^k_{\mu, 2}(z, r)^2 \frac{dr}{r} \leq \delta(n)^2 M \quad \forall z \in B_1(0),
\end{gather}
where $\delta(n)$ is some small constant.  We say a ball $B_r(x)$ is a \emph{stop} ball if $\mu(B_r(x)) \leq M r^k$, so that non-stop balls have big mass relative to $\delta(n)^2M$.

We implement two separate stopping-time arguments, one for good balls (Section \ref{section:good-tree}) and one for bad balls (Section \ref{section:bad-tree}).  In each case we build a \emph{tree} of balls, which is a successive refinement of coverings by good, bad, and stop balls.  We then chain these trees together (Section \ref{section:finishing}) to obtain our estimate.  Let us describe the good and bad tree constructions in a little more detail.


\vspace{2.5mm}

A \emph{good tree} is built in the following manner.  Start at some initial good ball $B_{r_0}(g_0)$, with $L^2$-best plane $L_0$.  We define the good/bad/stop balls at the next scale $r_1$ to be a Vitali cover of $B_{r_1/40}(L_0)$.  Let us call these $\{B_{r_1}(g)\}_{g \in \cG_1}$, $\{B_{r_1}(b_1)\}_{b \in \cB_1}$, and $\{B_{r_1}(s)\}_{s \in \cS_1}$, respectively.

Now $B_{r_0}(g_0)$ is covered by $B_{r_0}(g_0) \setminus B_{r_1/50}(L_0)$, and the good/bad/stop balls at scale $r_1$.  The set $B_{r_0}(g_0) \setminus B_{r_1/50}(L_0)$ is called the \emph{excess set} for $B_{r_0}(g_0)$.

Write $L_{1, g}$ for the $L^2$-best plane associated to good ball $B_{r_1}(g)$.  We define the good/bad/stop balls at scale $r_2$ to be a Vitali cover of
\begin{gather}
\left( \bigcup_{g \in \cG_1} B_{r_1}(g) \cap B_{r_2/40}(L_{1,g}) \right) \setminus \left( \bigcup_{b \in \cB_1} B_{r_1}(b) \cup \bigcup_{s \in \cS_1} B_{r_1}(s) \right).
\end{gather}
In other words, we cover the neighborhoods of $L^2$-best planes, while avoiding all previous bad/stop balls.  Now $B_{r_0}(g_0)$ is covered by the various excess sets, and the good $r_2$-balls, and the bad/stop balls at scales $r_1$ and $r_2$.

We proceed in this fashion, inducting into the $r_i$-good balls, and avoiding the bad/stop balls at scales $r_i, r_{i-1}, \ldots, r_1$.  We end up with a cover of $B_{r_0}(g_0)$ by excess sets, bad/stop balls at all scales, and a region inside every good $r_i$-ball collection.  This covering is called the \emph{good tree}.

In the good tree, we have good tilting control.  We can perform a construction analogous to \cite{reif_orig}, \cite{toro:reifenberg}, \cite{davidtoro}, \cite{naber-valtorta:harmonic}, to build at each scale a bi-Lipschitz manifold $T_i$ that cuts through all $r_i$-good balls in a controlled fashion, and all bad/stop balls at scales $r_i, r_{i-1}, \ldots, r_1$.  Since the corresponding balls of radius $r_i/10$ are pairwise disjoint, and their intersection with $T_i$ has $k$-measure comparable to $r_i^k$, we obtain the $k$-packing estimate in a good tree, and we can bound the measure of stop balls and excess sets in terms of the packing estimate.

Let us elaborate a little on the manifold construction.  As in \cite{reif_orig}, each $T_i$ is essentially a glorified interpolation of planes, constructed inductively via maps $\sigma_i$.  We set $T_0 = L_0$, and then $T_{i+1} = \sigma_{i+1}(T_i)$.  More precisely, if $\{B_{r_i}(g)\}_{g \in \cG_i}$ is the collection of good $r_i$-balls, and $\{L_{i,g}\}$ are associated $L^2$-best planes, then $\sigma_i$ is the following interpolation of projection mappings
\begin{gather}
\sigma_i(x) = \sum_{g \in\cG_i} \phi_{i, g}(x) (\proj_{L_{i,g}}(x - X_{i,g}) + X_{i,g}) \, .
\end{gather}
Here $X_{i,g}$ is the (generalized) $\mu$-center of mass for $B_{r_i}(g)$ (see Definition \ref{def:gen-COM}), $\proj_{L_{i,g}}$ the linear projection to $L_{i,g}$, and $\phi_{i,g}$ is a partition of unity subordinate to $\{B_{3r_i}(g)\}$.

Estimates of Section \ref{subsection:good-tilting} control the tilting of subsequential $L^2$-best planes in terms of $M^{-1/2} \beta$, and consequently in Section \ref{subsection:good-tilting-map} we obtain scale-invariant $C^1$ bounds on $\sigma_i$ in terms of $M^{-1/2} \beta$ also.  These estimates are essentially standard, but we emphasize that our construction requires \emph{no upper mass control} in the good and bad balls.  We may be packing regions with infinite $\mu$-mass.

These various estimates imply that the composition $\sigma_i \circ \cdots \circ \sigma_1$ has a bi-Lipschitz bound at $x$ like
\begin{gather}
\prod_{\ell=0}^i (1 + c(n) M^{-1} \beta(x, r_\ell)^2) \leq e^{c(n)\delta^2}\, .
\end{gather}
This gives a uniformly bi-Lipschitz bound on each $T_i$, and shows that $T_i \to T_\infty$ for some bi-Lipschitz $T_\infty$.

Ultimately, the main Theorem \ref{thm:good-tree} of Section \ref{section:good-tree} says: the good tree decomposes $B_{r_0}(g_0)$ into a region with bounded mass, a collection of bad balls with packing, and a subset of $T_\infty$ (the region inside every good $r_i$-ball collection).

\vspace{2.5mm}

The \emph{bad tree} is constructed in essentially the same way as the good tree, except with the roles of good and bad balls swapped, and with a ``best'' $(k-1)$-plane instead of a best $L^2$ $k$-plane.  In bad trees we are covering a neighborhood of a $(k-1)$-plane, and can therefore naively estimate the $k$-packing of all bad/good/stop balls at all scales.  In fact we can make our $k$-packing estimate \emph{very small}, by choosing the scale-drop very small, which also allows us to not worry about issues like double counting, which is a serious concern for the good ball coverings.  The measure of stop balls and excess sets is then controlled by the packing.

The main Theorem \ref{thm:bad-tree} of Section \ref{section:bad-tree} says: the bad tree decomposes a bad ball $B_{r_0}(b_0)$ into a region with bounded mass, a collection of good balls with \emph{small} packing, and a region with $\haus^k$-measure $0$ (the region inside every bad $r_i$-ball collection).

\vspace{2.5mm}

We outline how to chain these trees together.  In a good tree, away from the bad balls we have packing and measure estimates down to a small scale.  We sometimes call the bad balls of a good tree the \emph{leaves}.  Similarly, in a bad tree we have small-scale packing and measure estimates away from the good balls (the \emph{bad tree leaves}).

To obtain global estimates down to a small scale we first construct an initial good or bad tree at $B_1(0)$, which of course we can assume is not a stop ball (otherwise we are done anyway).  At any leaf of this good/bad tree, we build bad/good tree, and thereby obtain global estimates away from smaller leaves.  Then in leaves of the secondary trees, we build tertiary trees,  and continue this construction inductively.

Each time we build a new family of trees, the trees switch type, and we get estimates on a smaller scale.  The type-switching is very important: it means we can always cancel double-counting error with the small bad-tree packing, choosing an appropriate scale drop (equation \eqref{eq_rho}).  In fact we end up with global packing on \emph{all} leaves of \emph{all} trees (Theorem \ref{thm:tree-packing}), which allows us to concatenate packing and measure estimates from each individual tree.

\section{Examples}

In this section we gather some examples regarding measures and $\beta$ number estimates. We start with a basic example on the finiteness of the $\beta$ number for graphs in order to help build an intuition.

\subsection{Graphs}
\begin{example}\label{ex_C2manifold}
 Let $S$ be the graph of a $C^2$ function $f:\R^k\to \R^{n-k}$, so that $S\subset \R^n$. Evidently, one expects $\haus^k\llcorner S$ to have nice $\beta^2$ bounds. In order to verify this, we compute the Taylor expansion of $f$ around $(x_0,f(x_0))\in S$:
 \begin{gather}
  \abs{f(x)-f(x_0)-\nabla f(x_0)\cdot (x-x_0)}\leq \norm{\nabla^2 f}_\infty \abs{x-x_0}^2\, .
 \end{gather}
In particular, using simple geometric considerations, this proves that
\begin{gather}
\beta^k_{\haus^k\llcorner S}((x_0,f(x_0)),r)^2\leq c(k,\norm{\nabla f}_\infty)\norm{\nabla^2 f}_\infty r\, .
\end{gather}
In turn, this implies the uniform pointwise bound
\begin{gather}\label{eq_C2D}
 \int_0^1 \beta^k_{\haus^k\llcorner S}((x_0,f(x_0)),r)^2 \frac{dr}{r}\leq c(k,\norm{\nabla f}_\infty)\cdot \norm{\nabla^2 f}_\infty \, .
\end{gather}
\end{example}
\vspace{.25cm}

It is easy to see that a similar computation holds for all graphs of $C^{1,\alpha}$ functions.\\

As opposed to the previous example, we show an easy case where the integral in \eqref{eq_C2D} does not converge.
\begin{example}\label{ex_lip_graph}
 Let $S$ be the graph of the function $\alpha \abs x$ in $\R^2$, where $\alpha \neq 0$. Then clearly $\beta^k_{\haus^k\llcorner S}((0,0),r)$ is constant in $r$, and thus $\int_0^1 \beta^k_{\haus^k\llcorner S,2}(r)^2 \frac{dr}{r}=\infty$. This proves that the estimate in \eqref{eq_C2D} cannot hold in a pointwise sense for Lipschitz graphs.
\end{example}

It is worth noticing however that although the pointwise bound of \eqref{eq_C2D} does not hold for Lipschitz graphs, an integral estimate follows from \cite[theorem 6]{Dorronsoro} (see also \cite[theorem 1.2]{AS}  and \cite[Theorem 1.42]{david-semmes}). For the sake of completeness, here we report the result (without proof):
\begin{proposition}\label{lemma_1}
 Let $f:\R^k \to \R^{n-k}$ be a Lipschitz map with Lipschitz constant bounded by $L$, and let $S$ be the graph of $f$. Then for all $x\in S$ and $r\geq 0$ we have
 \begin{gather}
  \int_{B_r(x)} \ton{\int_0^r \beta^k_{\haus^k\llcorner S}(x, s)^2 \frac {ds}{s} }d\haus^k(x)\leq C(k) (1+L^2)^{k/2} L^2 r^k\, .
 \end{gather}
\end{proposition}
\vspace{5mm}

\subsection{GeneralizedS Koch Snowflake}

\begin{example}\label{example:koch}

In this example, we recall some properties of the snowflake curve, which will allow us to illustrate why it is reasonable to use summability properties of the \textit{square} of the beta numbers in order to obtain measure bounds and rectifiability. We start by briefly recalling the construction of the Koch curve.

The construction of a Koch curve (snowflake) of height $\kappa>0$ is well known (as a reference, see \cite[section 4.13]{mattila}).  The basic building block is the following:  Given $a,b\in \dR^2$ and the induced segment $\ell=[a,b]\subseteq \dR^2$ of length $|\ell|$, split the segment into three equal parts $[a,x_1]\cup[x_1,x_2]\cup[x_2,b]$, and replace the middle segment $[x_1,x_2]$ with a pair of segments $[x_1,z]\cup [z,x_2]$ having the same lenght and such that $d(z,\ell)=\kappa \abs{\ell}$.  To build the Koch snowflake let us consider a sequence of piecewise segments $\gamma_i$ defined as follows.  Let $\gamma_0$ be the unit segment $[0,1]\times\{0\}\subseteq \R^2$.  To construct $\gamma_{i+1}$ from $\gamma_i$ let us apply the building block construction above with parameter $\kappa$ to each segment of $\gamma_i$.  Limiting gives the Koch curve $\gamma$.

It is clear that the length of the curve $\gamma_i$ is given by $|\gamma_i|=\ton{1+(\sqrt{1+c\kappa^2}-1)/3}|\gamma_{i+1}|$, so the length of the snowflake $\gamma$ will be infinity for any $\kappa>0$. This is a simple application of the Pythagorean theorem, and the extra square power on $\kappa$ comes from the fact that loosely speaking at each step we are adding some length $\kappa$ to the curve, but in a direction perpendicular to it.\\

To build the generalized Koch snowflake let us instead consider a sequence of parameters $\kappa_i$.  Let $\gamma_0$ be the unit segment as before, and now to construct $\gamma_{i+1}$ from $\gamma_i$ let us apply the building block construction above with parameter $\kappa_{i}$ to each segment of $\gamma_i$.  Limiting gives the generalized snowflake $\gamma$.  It is clear that the length of the curve $\gamma_{i+1}$ is given by $|\gamma_{i+1}|\sim (1+c\kappa_i^2) |\gamma_{i}|$, so the length of the snowflake $\gamma$ is finite if and only if $\sum \kappa_i^2<\infty$.

This suggests that the finiteness of the Hausdorff measure of the set $S$ is related to the summability properties of \textit{squares} of its $\beta$ numbers over scales. 

\end{example}

\subsection{\texorpdfstring{$\beta$}{beta}-Control without Density Bounds}
The following examples illustrate that some of the assumptions in the main theorems (and its corollaries) are necessary.  In particular, we will see how densities play a role in our results.

\begin{example}\label{ex_density}
This example illustrates how in Theorem \ref{thm:main-thm-L1}, we cannot get rid of the $\mu_0$ part in the decomposition without assuming lower bounds on $\Theta^{*,k}$.  
Consider the $n$-dimensional Lebesgue measure $\mu = \leb^n\llcorner \B 1 0$.  Then for every $x, r$, we have
\begin{align}
 \beta^k (x, r)^2 \leq r^{-k-2} \int_{B_r} 2r^2 d\leb^n  = 2\omega_n r^{n-k}\, . 
\end{align}

Hence for all $x\in \B 1 0$ and $R\leq 2$:
\begin{gather}
 \int_0^R \beta^k (x, r)^2 (x, r)\frac{dr}{r} \leq c(n) \int_0^R r^{n-k-1}dr \leq \frac{c(n)}{n-k}R^{n-k}\, .
\end{gather}

However $\leb^n$ is trivially not supported on any set of finite $\haus^k$-measure.  This tells us that finiteness (or even smallness) of $\int \beta^2 dr/r$ by itself does not give any kind of $k$-dimensional structure.
\end{example}

In the next example, we show that in the main Theorem \ref{thm:main}, we cannot obtain uniform packing estimates on the whole support of the measure $\mu$. In other words, we show that in general point $(B)$ in theorem \ref{thm:main} is necessary.

\begin{example}\label{ex_packing}
Fix $\rho \leq 1$.  Let $\{x_i\}$ be a maximal $\rho$-net in $B_2$, and we can assume the $x_i$ lie on a lattice parallel to coordinate axes.  For each $x_i$, let $S_i$ be a $k$-sphere centered at $x_i$, of radius $\rho^{n/k}$.  Set $S = B_1 \cap \ton{\cup_i S_i}$, and let $\mu = \haus^k \llcorner S$.

We have by construction that
\begin{gather}
\mu(B_1) \leq c(n) , \quad \Theta^{k}(x, \mu) = 1 \quad \forall x \in \spt \mu .
\end{gather}
We claim that for all $x\in \B 1 0$:
\begin{gather}
 \int_0^2 \beta^k (x,r)^2 \frac{dr}{r} \leq c(n)\, .
\end{gather}

Let us prove the Claim.  Morally what is happening, is that on scales $< \rho^{n/k}$ we look like a $k$-sphere, while on scales $> \rho^{n/k}$ we look like the Lebesgue measure as in the last example.  Fix $r > 0$.  If $r \leq \rho^{n/k}/10$, then $S \cap B_r(x)$ is the graph of a $C^2$ function over some $k$-dimensional plane. Thus, according to example \ref{ex_C2manifold}, we deduce
\begin{gather}
 \int_0^{10^{-1} \rho^{n/k}} \beta^k (x,r)^2 \frac{dr}{r} \leq c(n)\, 
\end{gather}
for every $x$.

Suppose $r \geq \rho^{n/k}$.  Then as with the Lebesgue measure we can calculate
\begin{gather}
\beta^k(x, r)^2\leq c(n) r^{-k-2} \cdot (r/\rho)^k \sum_{\ell=0}^{r/\rho} (\ell \rho + \rho^{n/k})^2 \cdot \rho^n (r/\rho)^{n-k-1} \leq 
c(n) r^{n-k} \, .
\end{gather}
Therefore
\begin{gather}
 \int_{10^{-1} \rho^{n/k}}^2 \beta^k (x,r)^2 \frac{dr}{r} \leq c(n)\, \quad \text{ and thus } \quad \int_0^{2} \beta^k (x,r)^2 \frac{dr}{r} \leq c(n)
\end{gather}
also.  This proves the claim.

So $\mu$ satisfies all the hypothesis of Theorem \ref{thm:main} uniformly in $\rho$, and has uniformly bounded upper density.  However we trivially have that $S$ fails any kind of packing estimate for $r \geq \rho$:
\begin{gather}
|B_\rho(S)| \geq |B_1| \geq c(n) .
\end{gather}

Let's go further, and even assume that we have cut out a set $S'$ of measure $\haus^k(S') \leq \epsilon$.  Then the number of entire $k$-spheres $S_i$ which can be covered by $S'$ is at most $c(n) \epsilon \rho^{-n}$.  Therefore
\begin{gather}
|B_\rho(S \setminus S')| \geq |B_1 \setminus B_{c(n) \epsilon}| \geq 1 - c(n) \epsilon^n .
\end{gather}

This shows that in general we cannot hope to have reasonable $k$-dimensional packing estimates on the support of $\mu$, not even if we consider the support of $\mu$ away from a portion with $\epsilon$ mass. 
\end{example}

\subsection{Sharpness without Density Assumptions}

In this subsection we give a brief example both illustrating Theorem \ref{thm:main} and showing that the estimates given are sharp.

\begin{example}\label{ex:sharp}
	For $L^k\subseteq \dR^n$ a $k$-dimensional subspace with $\{x_i\}\subseteq L^k\cap B_1$ a countable dense subset and $\Lambda_i>0$ any collection of positive numbers.  Let us consider the measure
	
\begin{align}
\mu = \leb^n\llcorner B_1 + \sum_i \Lambda_i\,\delta_{x_i}\, ,
\end{align}
where $\leb^n$ is the standard Lebesgue measure and $\delta_{x_i}$ is a Dirac delta centered at $x_i\in L^k\cap B_1$.  As in Example \ref{ex_density} we have the pointwise $\beta$-control
\begin{align}
	\int_0^2 \beta^k_{\mu}(x, s)^2 \frac {ds}{s} <C(n) \text{ for all }x\in B_1\, ,
\end{align}
where notice the bound on the right hand side is independent of the $\Lambda_i$.  Clearly, this measure has no upper or lower $k$-density bounds at {\it any} point.  Applying Theorem \ref{thm:main}, or more appropriately Corollary \ref{cor:main-thm-L1-scales}, we obtain the $k$-rectifiable subset $\cK=L^k\cap B_1$ with uniform $k$-Minkowski and $k$-Hausdorff estimates such that $\mu(B_1\setminus \cK) < C(n)$, as claimed.

We see in this extreme example where no point has either an upper or lower density bound that we cannot strengthen the estimates of the main theorems.
\end{example}

\vspace{.5cm}

\section{Preliminaries}

In this section, we gather the necessary preliminary results on $\beta$-estimates and tilting of the best approximating planes. Our ambient space will always be $\R^n$, equipped with the usual Euclidean metric.

Most of our Theorems are stated in a $k$-dimensional scale-invariant form.  When we say ``apply Theorem X to the measure $\mu$ at scale $B_r(x)$,'' we mean ``apply Theorem X to the measure $\mu_{x,r}$,'' where
\begin{gather}
\mu_{x,r}(A) := r^{-k} \mu(x + rA).
\end{gather}

\subsection{Definitions and notation}
Let us outline some basic notation and definitions used throughout the paper.  Given a set $A \subset \R^n$, write
\begin{gather}
B_r(A) = \{ x : \dist(x, A) < r \}
\end{gather}
for the (open) $r$-tubular neighborhood of $A$.  We write $|A|=\leb^n(A)$ for the $n$-dimensional volume, and $\overline{A}$ for the closure of the set $A$.  For ease of notation we will sometimes write $B_r(0)$ as $B_r$.  We write $\haus^k(A)$ for the $k$-dimensional Hausdorff measure of $A$, and $\haus^k_\delta$ for the usual $\delta$-approximation of Hausdorff measure.

Given a measure $\mu$, the upper-/lower-$k$-densities are defined as
\begin{gather}
\Theta^{*,k}(\mu, x) = \limsup_{r \to 0} \frac{\mu(B_r(x))}{r^k}, \quad \Theta_*^k(\mu, x) = \liminf_{r \to 0} \frac{\mu(B_r(x))}{r^k}.
\end{gather}

We write $d_H$ for the usual Hausdorff distance between sets
\begin{gather}
d_H(A_1, A_2) = \inf\{ r : A_1 \subset B_r(A_2) \quad\text{and}\quad A_2 \subset B_r(A_1) \}.
\end{gather}
Given two subspaces $L_1, L_2$, we define a Grassmanian type distance
\begin{gather}
d_G(L_1, L_2) = d_H(B_1(0) \cap L_1, B_1(0) \cap L_2).
\end{gather}
If $L_1, L_2$ are affine spaces, we define $d_G$ in terms of the associated subspaces.

Given an affine or linear space $L$, we will write $p_L$ for the \emph{linear} projection mapping, onto the associated linear subspace.  It's easy to show that
\begin{gather}\label{eqn:tilting-equivalence}
d_G(L_1, L_2) \leq |p_{L_1} - p_{L_2}| \leq 2 d_G(L_1, L_2) ,
\end{gather}
where $|p| = \sup_{|v| \leq 1} |p(v)|$ denotes the operator norm.  In this paper we will deal exclusively with the operator norm for linear mappings.

Given a function $f$, we define the $C^k$-norm at scale $\rho$ to be
\begin{gather}
|f|_{C^k_\rho} = \sum_{\ell=0}^k \sup_{\dom(f)} \rho^{\ell-1} |D^\ell f| \, , \quad \text{where} \quad \abs{D^\ell f} = \sup_{\abs {v_1}=\cdots=\abs{v_\ell} = 1 } \abs{D^\ell(f)[v_1,\cdots,v_\ell]}\, ,
\end{gather}
and $|f|_{C^k_\rho(\Omega)} = |f|_{\Omega}|_{C^1_\rho}$.  This definition is scale-invariant for graph-dilation: if $f_\rho(y)=\rho^{-1} f(\rho y)$, then $|f|_{C^k_\rho}=|f_\rho|_{C^k_1}$.  Given an affine $k$-plane $L$, and a function $f : \Omega \subset L \to L^\perp$,
\begin{gather}
\graph_\Omega(f) = \{ x + f(x) : x \in \Omega \}.
\end{gather}

We require the following effective notion of linear-independence for a collection of points.
\begin{definition}\label{def:COM}
We say a collection of points $p_0, \ldots, p_k$ are in \emph{$\rho$-general position} if, for each $i$, we have $p_i \not\in B_\rho(<p_0, \ldots, p_{i-1}>)$.  Here by $<p_0, \ldots, p_{i-1}>$ we mean the affine $(i-1)$-space containing these points.
\end{definition}

The $\mu$-center of mass of $B_r(x)$ is defined to be
\begin{gather}
\frac{1}{\mu(B_r(x))} \int_{B_r(x)} z d\mu(z).
\end{gather}
Let us define here a generalized $\mu$-center of mass, which captures a crucial property of the center of mass even if $\mu(B_r(x)) = \infty$.  We beg the reader's patience in justifying the well-definedness of Definition \ref{def:gen-COM} until Proposition \ref{prop:gen-COM}.
\begin{definition}\label{def:gen-COM}
We define the \emph{generalized $\mu$-center of mass} $X_{x, r}$ of a ball $B_r(x)$ as follows.  If $\mu(B_r(x)) < \infty$, let
\begin{gather}
X_{x, r} = \frac{1}{\mu(B_r(x))} \int_{B_r(x)} z d\mu(z)
\end{gather}
be the usual center of mass.  If $\mu(B_r(x)) = \infty$, we let $X_{x, r}$ be any point in the intersection
\begin{gather}
\overline{B_r(x)} \cap \bigcap \left\{\text{affine $V^k$} : \int_{B_r(x)} d(z, V)^2 d\mu(z) < \infty \right\}.
\end{gather}
\end{definition}

We now define our fundamental notion of good and bad ball. This definition depends on a choice of $\rho$, $m$, and is of course is specific to some measure $\mu$.  Usually the $\rho$, $m$, $\mu$ will be unambiguous from the context, but if not we will be explicit.

\begin{definition}\label{def:good-bad-ball}
We say a ball $B_r(x)$ is \emph{good} if there are points $y_0, \ldots, y_k \in B_r(x)$, such that for each $i$
\begin{gather}
\mu(B_{\rho r}(y_i)) \geq m (\rho r)^k,
\end{gather}
and if $Y_i$ is the generalized $\mu$-center of mass of $B_{\rho r}(y_i)$, then the $\{Y_i\}$ are in $\rho r$-general position.  We say $B_r(x)$ is \emph{bad} if it is not good.
\end{definition}
In simple terms, a good ball has large mass which is spread on some affine $k$-plane.  The reason we care is because $\beta$-numbers control the distance between centers of mass, and $L^2$-best planes (Proposition \ref{prop:COM-control}).  If we have lots of mass nicely spread out over a $k$-plane, we can effectively control tilting between different $L^2$-best planes.\\

\subsection{Basic lemmas}

We collect a few basic results used throughout the paper.  First, we require the following standard measure-theoretic result.
\begin{lemma}\label{lem:density-to-ineq}
Let $\mu$ be a non-negative Borel-regular measure.
\begin{enumerate}
\item[A)] Let $A_1 \subset A_2$, and suppose
\begin{gather}
\Theta^{*, k}(\mu \llcorner A_2, x) \geq t \quad \text{for $\mu$-a.e. $x \in A_1$}.
\end{gather}
Then
\begin{gather}
t\haus^k(A_1) \leq \mu(A_2).
\end{gather}

\item[B)] Suppose $\mu$ is finite, and for some $A$ we know instead that
\begin{gather}
\Theta^{*, k}(\mu, x) \geq t \quad \text{for $\mu$-a.e. $x \in A$},
\end{gather}
then
\begin{gather}
t\haus^k(A) \leq \mu(A).
\end{gather}
\end{enumerate}
\end{lemma}

\begin{proof}
Let us prove A).  The proof is given in \cite{simon:gmt}, but for the convenience of the reader we reproduce it here.  We can assume $\mu(A_2) < \infty$ and $t > 0$.  Take $\delta > 0$, and $0 < \tau < t$.

By assumption, the collection of balls
\begin{gather}
\{ B_s(x) : x \in A_1, \quad s < \delta, \quad \text{and}\quad\mu(B_s(x) \cap A_2) > \tau \omega_k s^k \}
\end{gather}
covers $A_1$ finely.  Therefore we can choose a disjoint Vitali subcollection $\{B_{s_i}(x_i)\}_i$, so that
\begin{gather}
A_1  \subset \bigcup_{i=1}^N B_{s_i}(x_i) \cup \bigcup_{i=N+1}^\infty B_{5s_i}(x_i) \quad \forall N.
\end{gather}
We have
\begin{gather}
\tau \sum_i \omega_k s_i^k \leq \mu(\cup_i B_{s_i}(x_i) \cap A_2) \leq \mu(A_2) < \infty.
\end{gather}

We then calculate, for any $N$, 
\begin{align}
\tau \haus^k_{5\delta} (A_1) 
&\leq \tau \sum_{i=1}^N \omega_k s_i^k + 5^k \tau \sum_{i=N+1}^\infty \omega_k s_i^k \\
&\leq \mu(A_2) + o(1) \quad \text{ as $N \to \infty$}.
\end{align}
Taking $N \to \infty$, $\delta \to 0$, then $\tau \to t$ proves part A).

Let us prove part B).  Since $\mu$ is finite, we can choose an open $U \supset A$ so that $\mu(U) \leq \mu(A) + \eps$.  Then taking $A_1 = A$, and $A_2 = U$, we can apply the first part of the Lemma to deduce
\begin{gather}
t\haus^k(A) \leq \mu(A) + \eps.
\end{gather}
Now take $\eps \to 0$.
\end{proof}

\vspace{2.5mm}

Next it is worth pointing out that $\beta^k_\mu(x, r)$ is lower-semicontinuous in $(x, r)$.  This implies in particular that pointwise bounds on the $\beta$-numbers or summed-$\beta$-numbers in a set can always be extended to the closure.

\begin{lemma}\label{thm:lsc}
Suppose $\beta^k_{2,\mu,\bar \epsilon_\beta} (x, r)^2 < \infty$.  If $x_i \to x$ and $r_i \to r > 0$, then
\begin{gather}
\beta^k_{\mu} (x, r)^2  \leq \liminf_i \beta^k_{\mu} (x_i, r_i)^2  .
\end{gather}
\end{lemma}

\begin{proof}
We can assume $\beta^k_\mu (x, r) > 0$, otherwise the statement is vacuous, so in particular $\mu \ton{B_r(x)} > \bar\epsilon_\beta r^k$.  Suppose, towards a contradiction, we have a sequence $r_i \to r$, and $x_i \to x$, so that
\begin{gather}
\liminf_i \beta^k_\mu (x_i, r_i) < \beta^k_\mu (x, r)\, .
\end{gather}
First of all, since $1_{\B {r}{x}}\leq \liminf_{i\to \infty} 1_{\B {r_i}{x_i}}$, by Fatou's lemma we have
\begin{gather}
\liminf_i \mu \ton{B_{r_i}(x_i)} > \bar\epsilon_\beta r_i^k\, .
\end{gather}

Moreover, let $V_i = V(x, r_i)$ and $V = V(x, r)$.  We can assume, by passing to a subsequence, that $V_i \to W$ for some $W$.  Since
\begin{gather}
d(z, V_i)^2 1_{B_{r_i}(x_i)} \to d(z, W)^2 1_{A} \quad \forall z,
\end{gather}
for some set $A \supset B_r(x)$, we have again by Fatou that
\begin{gather}
r^{-k-2} \int_{B_r(x)} d(z, W)^2 d\mu(z) \leq \liminf_i \beta^k_\mu(x_i, r_i) < \beta^k_\mu(x, r),
\end{gather}
which contradicts our definition of $V$.
\end{proof}

\begin{remark}\label{rem:sum-vs-int}
Another observation which will prove useful is that we can replace the Dini-type condition $\int \beta \frac{dr}{r}$ with a sum over scales.  By monotonicity of $\beta$, we have
\begin{gather}
\sum_{\alpha \in \dZ \,\, : \,\, 2^{\alpha} \leq r/2} \beta^k_{\mu, 2, 2^k\bar\eps_\beta}(x, 2^\alpha)^2 \leq 2^{k+3} \int_0^r \beta^k_{\mu,2,\bar\eps_\beta}(x, s)^2 ds/s,
\end{gather}
and conversely
\begin{gather}
\int_0^{r} \beta^k_{\mu, 2,2^k\bar\eps_\beta}(x, s)^2 ds/s \leq 2^{k+3} \sum_{\alpha \in \dZ \,\, : \,\, 2^\alpha \leq r} \beta^k_{\mu,2,\bar\eps_\beta}(x, 2^\alpha)^2.
\end{gather}
\end{remark}

\vspace{2.5mm}

As discussed in Definition \ref{def:gen-COM} we will be using the (generalized) center of mass extensively to control tilting between $L^2$-best planes.  We need two preparatory Lemmas, which allow us to control $k$-plane Grassmanian distances by considering only a choice of points in general position.

\begin{lemma}\label{lem:general-pos-bounds}
Suppose $x_0, \ldots, x_k$ in $B_1$ are in $\rho$-general position, and there is an affine space $W$ so that $d(x_i, W) \leq \delta$ for each $i$.  Then $<x_0, \ldots, x_k> \cap B_1 \subset B_{c(n, \rho) \delta}(W)$.
\end{lemma}

\begin{proof}
By Gram-Schmidt orthogonalization, we can write any $y \in <x_0, \ldots, x_k> \cap B_1$ as
\begin{gather}
y = x_0 + \sum_{i=1}^k a_i(x_i - x_0),
\end{gather}
with $|a_i| \leq c(n,\rho)$.  Let $w_i\in W$ with $|x_i-w_i|\leq \delta$, then we can estimate
\begin{align}
  |y_i-w_0-\sum a_i(w_i-w_0)|\leq |x_0-w_0|+\sum |a_i|\,\big(|x_i-w_i|+|x_0-w_0|\big)\leq c(n,\rho)\delta\, ,
\end{align}
as claimed.
\end{proof}

In the following Lemma the fact that both planes have the same dimension is crucial.
\begin{lemma}\label{lem:contain-to-bound}
Let $V, W$ be affine $k$-spaces, such that $V \cap B_2 \subset B_\delta(W)$, and $V \cap B_1 \neq\emptyset$.  Then $d_H(V \cap B_2, W \cap B_2) \leq c(n) \delta$.
\end{lemma}

\begin{proof}
We can assume $\delta \leq \delta_0(n)$.  The assumptions imply $V \cap B_2 \subset B_{c(n) \delta}(W \cap B_2)$.  Choose $x_0 \in V$ to minimize $|x_0|$.  Now let $x_i \in V$ be chosen so that
\begin{gather}
|x_i - x_0| = 1, \quad <x_i - x_0, x_j - x_0> = 0.
\end{gather}

Let $y_i$ be a point in $W$ with $d(x_i, y_i) \leq \delta$.  Then
\begin{gather}
|y_i - y_0| \geq 1 - 2\delta, \quad |<y_i - y_0, y_j - y_0>| \leq 8\delta .
\end{gather}
Hence, for $\delta \leq \delta_0(n)$, the $y_i$ are in $1/2$-general position in $W$.  One can see this using the Gram-Schmidt orthogonalization process.

For any vector $w \in W \cap B_2$, since $y_i \in B_{2+c(n)\delta} \cap W$, we have
\begin{gather}
w = y_0 + \sum_i a_i (y_i - y_0), \quad |a_i| \leq c(n) ,
\end{gather}
and therefore $d(w, V) \leq c(n) \delta$.  Since $V \cap B_1 \neq \emptyset$, we have $W \subset B_{c(n) \delta}(V \cap B_2)$.
\end{proof}

\vspace{2.5mm}

Our approximating maps $\sigma_i$ in the good tree are locally perturbations of affine projections.  We need a standard lemma, due in this formulation to L. Simon, which tracks how graphs change under almost-projections.

\begin{lemma}[``Squash Lemma'' \cite{simon_reif}]\label{lem:squash}
Let $L$, $\tilde L$ be affine $k$-planes, with $L \cap B_{3/2} \neq \emptyset$.  Let $G \subset \dR^n$ be a graph over $L$:
\begin{gather}
G = \graph_{\Omega \subset L} f, \quad |f| + |Df| \leq 1, \quad B_{5/2} \cap L \subset \Omega \subset L.
\end{gather}
Consider the mapping $\Phi : B^n_3 \to \dR^n$ defined by
\begin{gather}
\Phi(x) = m + p_{\tilde L}(x - m) + e(x),
\end{gather}
where $p_{\tilde L}$ is the linear projection operator onto $\tilde L$.

Suppose
\begin{gather}
d(m, \tilde L) \leq \eps, \quad d_H(L \cap B_3, \tilde L \cap B_3) \leq \eps, \quad \text{and}\quad |e| + |De| \leq \eps \text{ on } B_{5/2}.
\end{gather}
Then provided $\eps \leq \eps_0(n)$, we have
\begin{gather}
\Phi(G \cap B_{5/2} ) = \graph_{U \subset \tilde L} \tilde f, \quad |\tilde f| + |D \tilde f| \leq 8\eps, \quad B_2 \cap \tilde L \subset U \subset \tilde L.
\end{gather}
\end{lemma}

\begin{proof}
For ease of notation let us write $\tilde p$, $\tilde p^\perp$, $p$ for the linear projections onto $\tilde L$, $\tilde L^\perp$, $L$, respectively.

Write $L = y + L_0$, and $\tilde L = \tilde y + \tilde L_0$, where $L_0$ and $\tilde L_0$ are linear subspaces.  Choose ON bases $\{e_i\}_i$ of $L_0$, and $\{\tilde e_i\}_i$ of $\tilde L_0$, so that $\tilde p(e_i) = \lambda_i \tilde e_i$ (see for example \cite[Lemma 7.1]{dlms} for the details).  By assumption,
\begin{gather}
|\lambda_i - 1| \leq 2d_G(L, \tilde L) \leq 6 \eps.
\end{gather}

Define $h : \Omega \to \tilde L$ by setting
\begin{align}
h(x) &= \tilde y + \tilde p(x - \tilde y) + \tilde p(f(x)) + \tilde p(e(x + f(x))) \\
&\equiv \tilde p^\perp(\tilde y) + \tilde p(x) + \tilde p(f(x)) + \tilde p(e(x + f(x))).
\end{align}
so that
\begin{gather}
\Phi(x + f(x)) = h(x) + \tilde p^\perp(e(x + f(x))) + \tilde p^\perp(m - \tilde y).
\end{gather}

For $x \in L \cap B_{5/2}$, we have,
\begin{align}
|h(x) - x| 
&\leq |\tilde y + \tilde p(x - \tilde y)| + |\tilde p - p| |f(x)| + |e(x + f(x))|\\
&\leq d_H(L \cap B_3, \tilde L \cap B_3) + 2d_G(L, \tilde L) + \eps\\
&\leq 8\eps, \label{eqn:squash-h-small}\, .
\end{align}
Since $p(f(x))=0$, we have $\tilde p (f(x))=[\tilde p - p ] (f(x))$ and so:
\begin{align}
|Dh(x) - \tilde p| 
&\leq |\tilde p - p| |Df| + |\tilde p| |De| |I + Df| \\
&\leq 8\eps. \label{eqn:squash-Dh-small}
\end{align}

Let us identify $L$, $\tilde L$ (and hence $L_0$, $\tilde L_0$) with $\dR^k$ via the bases $\{e_i\}_i$, $\{\tilde e_i\}_i$.  Under this identification, we have
\begin{gather}
|Dh - I| \leq |Dh - \tilde p| + |\tilde p - I| \leq 20\eps \quad \text{on $B^k_{5/2}$}.
\end{gather}
Therefore, provided $\eps \leq \eps_0(n)$, $h$ is a diffeomorphism from $B_{5/2} \cap L$ onto its image $U$ in $\tilde L$, with $|Dh^{-1}| \leq 2$.  From \eqref{eqn:squash-h-small} we can say $U \supset B_2 \cap \tilde L$.

Define
\begin{gather}
\tilde f(z) = \tilde p^\perp(m - \tilde y) + \tilde p^\perp(e(h^{-1}(z) + f(h^{-1}(z)))).
\end{gather}
By assumption and \eqref{eqn:squash-Dh-small}, we have for $z \in U$, 
\begin{gather}
|\tilde f| \leq d(m, \tilde L) + |e| \leq 2 \eps, \quad |D \tilde f| \leq |De| |I + Df| |Dh^{-1}| \leq 4\eps.
\end{gather}

Since we can write
\begin{gather}
\Phi(x + f(x)) = h(x) + \tilde f(h(x)),
\end{gather}
this completes the proof of Lemma \ref{lem:squash}.
\end{proof}

\vspace{2.5mm}

\subsection{Generalized center of mass}
We prove well-definedness of the generalized $\mu$-center of mass, and the relation it has with $L^2$-best planes.

\begin{proposition}\label{prop:gen-COM}
 Suppose that $\mu(\B r x)=\infty$. Then there exists $X\in \overline{\B r x}$ such that $X\in V^k$ for all affine $k$-planes $V^k$ such that $\int_{\B r x} d(x,V)^2 d\mu(x)<\infty$. 

In particular, the notion of generalized $\mu$-center of mass (Definition \ref{def:gen-COM}) is well-defined.
\end{proposition}
\begin{proof}
 Consider the collection $\cur{V_\alpha}_{\alpha \in A}$ of all $k$-affine spaces $V$ such that
 \begin{gather}
  \int_{\B r x} d(x,V)^2 d\mu(x)<\infty\, .
 \end{gather}
It's easy to see that there must be some finite subcollection $V_0, \ldots, V_{k+1}$ so that
\begin{gather}
W := \bigcap_\alpha V_\alpha = \bigcap_{i=0}^{k+1} V_i.
\end{gather}

Suppose, towards a contradiction, that $\overline{B_r(x)} \cap W = \emptyset$.  Then it must follow
\begin{gather}
 \overline{B_r(x)} \cap \bigcap_{i=0}^{k+1} B_\eps(V_i)=\emptyset
\end{gather}
for some $\eps > 0$, because $\overline{B_r(x)}$ and each $V_i$ are closed.

But then, using the definition of $V_i$, we have
\begin{align}
\mu(B_r(x))
&= \mu \left(B_r(x) \cap (\dR^n \setminus \cup_{i=0}^{k+1} B_\eps(V_i)) \right)  \\
&\leq \sum_{i=0}^{k+1} \mu(B_r(x) \cap (\dR^n \setminus B_\eps(V_i)) \\
&\leq \sum_{i=0}^{k+1} \eps^{-2} \int_{B_r(x)} d(z, V_i)^2 d\mu(z) \\
& < \infty,
\end{align}
which is a contradiction.  In the penultimate inequality we used that $k < n$.
\end{proof}
\vspace{.2cm}

Recall that we defined in Definition \ref{def:gen-COM} the generalized $\mu$-center of mass to be any $X$ as in Proposition \ref{prop:gen-COM} whenever a ball has infinite $\mu$-mass.  The essential characteristics of the generalized center of mass is that its distance from best approximating planes is controlled by the $\beta$-numbers. More specifically:
\begin{prop}\label{prop:COM-control}
Suppose $B_r(y) \subset B_R(x)$, and $\mu\ton{ B_R(x)} > \bar\eps_\beta R^k$.  Let $Y$ be the generalized center of mass for $B_r(y)$.  Then
\begin{gather}
d(Y, V(x, R))^2 \leq \frac{R^{k+2}}{\mu(B_r(y))} \beta^k_{\mu,2,\bar\eps_\beta}(x, R)^2\, .
\end{gather}
\end{prop}

\begin{proof}
We can assume $\beta^k_\mu(x, R) < \infty$, otherwise there is nothing to show.  By the previous proposition, if $\mu \ton{\B r y }=\infty$, then $Y\in V(x,R)$. Otherwise, by Jensen's inequality, we calculate
\begin{align}
d(Y, V(x, R))^2
&\leq \frac{1}{\mu(B_r(y)) } \int_{B_r(y) } d(z, V(x, R))^2 d\mu(z) \\
&\leq \frac{R^{k+2}}{\mu(B_r(y)) } R^{-k-2} \int_{B_R(x)} d(z, V(x, R))^2 d\mu(z) . \qedhere
\end{align}

\end{proof}

\vspace{.25cm}

\subsection{Tilting and packing control}\label{subsection:good-tilting}

We are ready to prove the original dichotomy described in the Proof Outline: in good balls the $\beta$-numbers control tilting, and in bad balls we have good packing estimates.

Let us first give some intuition about why it is necessary to have mass effectively spread out to say anything about the $L^2$-best plane.  As a guiding example, consider the measure $\mu_1=\haus^{k-1}\llcorner L$, where $L$ is a linear $(k-1)$-dimensional space.  In this case, it is clear that $\beta^k_{\mu_1}(x,r)=0$ for all $x$ and $r$. However, the notion of best approximating $k$-dimensional plane is not even uniquely defined.

We can make this example more precise by slightly modifying the measure $\mu$. Consider
\begin{gather}
 \mu_2 = \mu_1 + \delta_{x_1} + \chi \delta_{x_2}\, ,
\end{gather}
where $\abs{x_1}=d(x_1,L)=1/2$ and $\abs{x_2}=d(x_2,<L,x_1>)=1/10$. In other words, $x_1$ is orthogonal to $L$ and $x_2$ is orthogonal to the subspace spanned by $x_1$ and $L$.

It is clear that $V(0,1/4)=<L,x_2>$, and $\beta^k_{\mu_2}(0,1/4)=0$, while we have the simple estimate
\begin{gather}
 \beta^k_{\mu_2}(0,1)\leq \chi/10^2\, .
\end{gather}
Moreover, it is easy to see that $V(0,1)$ can be made arbitrarily close to $<L,x_1>$ by choosing $\chi$ small enough. This implies that the distance between $V(0,1)$ and $V(0,1/4)$ cannot be uniformly bounded by $\beta(0,1)+\beta(0,1/4)$.

The issue in both these examples is that mass is not sufficiently $k$-dimensional at scale $\approx 1$ for $\beta(0, 1)$ and $\beta(0, 1/4)$ to sense the right plane.  This is why we can only hope to get tilting control in good balls.  In bad balls, mass will be sufficiently $(k-1)$-dimensional that we have good packing.

We start by showing this straightforward dichotomy:
\begin{lemma}\label{lem:ball-dichotomy}
Given a non-negative Borel measure $\mu$, and a choice of $m, \rho$, then at least one of the following occurs:
\begin{enumerate}
\item[A)] There are points $x_0, \ldots, x_k \in B_1$ such that for each $i$ we have
\begin{equation}\label{eqn:good-ball-points}
\mu \ton{B_\rho(x_i)} > m \rho^k , \quad\text{and}\quad X_i \not\in B_{\rho}(<X_0, \ldots, X_{i-1}>),
\end{equation}
where $X_i$ is the generalized center of mass of $B_\rho(x_i)$.

\item[B)] There is an affine $W^{k-1}$ such that
\begin{gather}\label{eqn:bad-ball-dichotomy}
\mu (B_1 \setminus B_{2\rho}(W^{k-1})) \leq 2^n \rho^{k-n} m.
\end{gather}
\end{enumerate}
In particular, if $B_1$ is bad then \eqref{eqn:bad-ball-dichotomy} holds for some $W^{k-1}$.
\end{lemma}

\begin{proof}
Let us construct such a collection $x_0, \ldots, x_k$ by induction, and see what happens when it fails.  Take $i \geq 0$, and suppose we've obtained $x_0, \ldots, x_{i-1}$ satisfying \eqref{eqn:good-ball-points}.  Write $W = <X_0, \ldots, X_{i-1}>$.

Suppose we can find an $x_i \in B_1 \setm B_{2\rho}(W)$ with $\mu \ton{B_\rho(x_i)} > m \rho^k$.  Then we have
\begin{gather}
d(X_i, V) \geq d(x_i, V) - \rho \geq \rho \, .
\end{gather}
And therefore we have constructed an $x_i$.

Otherwise, necessarily
\begin{gather}
y \in B_1 \setminus B_{2\rho}(W) \implies \mu \ton{B_\rho(y)} \leq m \rho^k .
\end{gather}
Take a maximal $\rho$-net $\{y_j\}_j$ of $B_1 \setm B_{2\rho}(W)$.  Then the balls $\{B_\rho(y_j)\}_j$ cover $B_1 \setm B_{2\rho}(W)$, and 
\begin{gather}
\#\{y_j\}_j \leq 2^n \rho^{-n}
\end{gather}
and $\mu \ton{B_\rho(y_j)} \leq m \rho^k$ for each $j$.  Since $\dim V = i \leq k-1$, the lemma is established.
\end{proof}

An easy consequence of Lemma \ref{lem:ball-dichotomy} is the bad-ball packing.
\begin{theorem}\label{thm:packing-control}
Let $B_1$ be a bad ball, and let $W^{k-1}$ be the $(k-1)$-plane of Lemma \ref{lem:ball-dichotomy}.  Then
\begin{enumerate}
\item[A)] If $\{y_i\}_i$ is a maximal $2\rho/5$-net in $B_{5\rho}(W^{k-1})$, then
\begin{gather}
\#\{y_i\}_i \rho^k \leq c_1(n) \rho.
\end{gather}

\item[B)] There is an $m_0(n, \rho) > 0$ so that if $m = m_0(n,\rho)\eps$, then 
\begin{gather}
\mu(B_1 \setminus B_{2\rho}(W^{k-1})) \leq \eps/2.
\end{gather}

\item[C)] If we know
\begin{gather}
\mu(B_\rho(y)) \leq \Gamma \rho^k \quad \forall y \in B_1 \,.
\end{gather}
Then provided $\rho \leq \rho_0(n, \eps, \Gamma)$, and $m = m_0(n, \rho)\eps$ as in part B), we have
\begin{gather}
\mu(B_1) \leq \eps.
\end{gather}
\end{enumerate}
\end{theorem}

\begin{proof}
Part A) follows because
\begin{gather}
\#\{y_i\}_i \leq \frac{\omega_{n-k+1} \omega_{k-1} 3 \cdot 30^n}{\omega_n} \rho^{-k+1}.
\end{gather}
Part B) follows by choosing $m_0 = \rho^{n-k} 2^{-n}/2$.  In part C), the assumption and parts A), B) imply
\begin{gather}
\mu(B_1) \leq 2^n \rho^{k-n} m + c_1(n) \Gamma \rho.
\end{gather}
\end{proof}

Let us now prove tilting control in good balls.
\begin{theorem}[Good-ball tilting]\label{thm:tilting-control}
Suppose $x \in B_7$, and $B_\rho(x)$ is a good ball, and that $\mu \ton{B_8} > \bar\epsilon_\beta 8^k$, $\mu \ton{B_{8\rho}(x)} > \bar\epsilon_\beta (8\rho)^k$.  Then we have that
\begin{gather}
d_H(V(x, 8\rho) \cap B_{8\rho}(x), V(0, 8) \cap B_{8\rho}(x))^2 \leq \frac{c(n,\rho)}{m} \ton{\beta^k_{\mu}(x, 8\rho)^2 + \beta^k_\mu(0, 8)^2}\, .
\end{gather}

Therefore we get that
\begin{gather}
d_G(V(x, 8\rho), V(0, 8))^2 \leq \frac{c(n, \rho)}{m} \ton{\beta^k_{\mu}(x, 8\rho)^2 + \beta^k_\mu(0, 8)^2}\, .
\end{gather}
\end{theorem}

\begin{proof}
In the below any $c$ will depend only on $n, \rho$.  By virtue of being a good ball, we have $y_0, \ldots, y_k \in B_\rho(x)$ so that
\begin{gather}
\mu \ton{B_{\rho^2}(y_i)} > m \rho^{2k}\, ,
\end{gather}
and the generalized center-of-masses $Y_i$ are in $\rho^2$-general position.

By Proposition \ref{prop:COM-control} we have
\begin{align}
d(Y_i, V(x, 8\rho))^2 &\leq \frac{(8\rho)^{k+2}}{m \rho^{2k}} \beta^k_\mu(x, 8\rho)^2 \leq \frac{c}{m} \beta^k_\mu(x, 8\rho)^2\, ,\notag\\
d(Y_i, V(0, 8))^2 &\leq \frac{8^{k+2}}{m\rho^{2k}} \beta^k_\mu(0, 8)^2  \leq \frac{c}{m} \beta^k_\mu(0, 8)^2\, .
\end{align}

Therefore, letting $L = <Y_0, \ldots, Y_k>$ we have by Lemma \ref{lem:general-pos-bounds} (at scale $8\rho$) that
\begin{align}
&L \cap B_{8\rho}(x) \subset B_{c m^{-1/2} \beta^k_\mu(x, 8\rho) }(V(x, 8\rho))\, ,\notag \\
&L \cap B_{8\rho}(x) \subset B_{c m^{-1/2} \beta^k_\mu(0, 8)}(V(0, 8))\, .
\end{align}

Since $Y_i \in B_{2\rho}(x)$ we have $L \cap B_{2\rho}(x) \neq\emptyset$.  Therefore, we have by Lemma \ref{lem:contain-to-bound} (at scale $8\rho$)
\begin{align}
&d_H(V(x, 8\rho) \cap B_{8\rho}(x), V(0, 8) \cap B_{8\rho}(x))^2 \notag\\
&\leq 2 d_H(L \cap B_{8\rho}(x), V(x, 8\rho) \cap B_{8\rho}(x))^2  + 2d_H(L \cap B_{8\rho}(x), V(0, 8) \cap B_{8\rho}(x))^2 \\
&\leq \frac{c}{m}(\beta^k_\mu(x, 8\rho)^2 + \beta^k_\mu(0, 8)^2) \, , \notag
\end{align}
as claimed.
\end{proof}
\vspace{1cm}

\section{Main construction}\label{section:construction}

We will use the dichotomy of good/bad balls of Definition \ref{def:good-bad-ball} to obtain good estimates as we drop down in scale.  In very simple terms, either we have good tilting control, and can obtain volume bounds by a Reifenberg type process; or, in bad balls, we have very good packing control and can compensate for the naive overlaps by choosing our scale sufficiently small.

In this construction, we fix $m=m_0(n,\rho) M$ as in Theorem \ref{thm:packing-control}.  The construction has a running dependence on $\rho$, which for Theorem \ref{thm:core-estimate} will be fixed as $\rho = \rho(n)$ in \eqref{eq_rho}, but in general may be chosen differently (as in Proposition \ref{prop:rect_with_bounds}).  However we can always assume $\rho \leq 1/20$.  For convenience, we can assume $\rho=2^{\alpha_0}$ for some $\alpha_0\in \mathbb Z$, and we use the following notation.

\begin{definition}
We write $r_i = \rho^i$.
\end{definition}

Recall that we are given a covering pair $(\cC, r_x)$.  We may refer to the $\{B_{r_y}(y)\}_{y \in \cC_+}$ as the \emph{original} balls. As explained in the outline, our construction is based on refining inductively on scales approximations of the support of $\mu$. Our refining needs to stop when we hit some of the ``original balls'' in the covering $\cC_+$. Here we define this stopping condition in detail.  

\begin{definition}\label{def:balls}
A ball $B_r(x)$ is a \emph{stop} ball if either $\mu \ton{B_r(x)} \leq M r^k$, or if there is some original ball $B_{r_y}(y)$ satisfying $x \in B_{r_y + 2r}(y)$ and $r < r_y \leq r/\rho$.
\end{definition}

Our key technical result is the following.  From this Theorem \ref{thm:main} follows more or less directly.

\begin{theorem}\label{thm:core-estimate}
There are constants $c_{key}(n)$, $\delta_0(n)$ so that the following holds: Let $M > 0$, $\mu$ be a non-negative Borel-regular measure in $\dR^n$ supported in $B_1$, and $(\cC, r_x)$ a covering pair for $\mu$ with $r_x \leq 1$.
Suppose we know
\begin{equation}\label{eq_sum_beta_bounds}
\sum_{\alpha \in \dZ \,\, : \,\, r_x < 2^\alpha \leq 2} \beta^k_{\mu, 2, \bar\eps_\beta}(x, 2^\alpha)^2 \leq \delta_0(n)^2 M \quad \text{for} \ \ \mu-a.e. \ x \in \cC.
\end{equation}

Then provided $\bar\eps_\beta \leq 2^{-k} \omega_k M/c_{key}(n)$, there is a closed subset $\cC' \subset \cC$ which admits the following properties:
\begin{enumerate}
\item[A)] Packing bound:
\begin{gather}\label{eq_C_+'_est_II}
\haus^k(\cC'_0) + \sum_{x \in\cC'_+} r_x^k \leq c_{key}(n),
\end{gather}
and Minkowski estimates
\begin{gather}\label{eq_mink_est_II}
r^{k-n} |B_r(\cC')| \leq c_{key}(n) \quad \forall 0 < r \leq 1,
\end{gather}

\item[B)] Upper measure estimates:
\begin{gather}\label{eq_main_measure_away_II}
\mu \left( B_1 \setminus B_{4r_x}(\cC') \right) \leq c_{key}(n) M ,
\end{gather}

\item[C)] Lower measure estimates:
\begin{gather}\label{eq_main_noncoll_II}
\mu(B_r(x)) \geq \frac{M}{c_{key}(n)} r^k \quad \forall 4 r_x < r \leq 1, \quad \forall x \in \cC' ,
\end{gather}


\item[D)] Fine-scale packing structure: $\cC'_0$ is closed, rectifiable, and admits bounds
\begin{gather}
\haus^k(\cC'_0 \cap B_r(x)) \leq c_{key}(n) r^k \quad \forall x \in B_1, 0 < r \leq 1.
\end{gather}

\end{enumerate}
\end{theorem}

\begin{remark}
This Theorem only requires the weaker hypothesis on $\cC$ that $\mu(\dR^n \setminus B_{r_x}(\cC)) = 0$, except in this case we would require \eqref{eq_sum_beta_bounds} to hold at \emph{every} $x \in \cC$. 

In fact the only real reason we define a covering pair by $\cC \supset \spt\mu$, instead of $\mu(\dR^n \setminus B_{r_x}(\cC)) = 0$, is to make sense of (weak) $L^1$ conditions as in Theorem \ref{thm:main}.  However, in the presence of $L^\infty$ bounds, like in Theorem \ref{thm:core-estimate}, the two notions are essentially the same: using monotonicity of $\beta$ one can always extend $\cC$ to cover $\spt\mu$, and still satisfy \eqref{eq_sum_beta_bounds} up to a factor of $c(k)$.
\end{remark}

Let us demonstrate how Theorem \ref{thm:main} follows from Theorem \ref{thm:core-estimate}.
\begin{proof}[Proof of Theorem  \ref{thm:main} given \ref{thm:core-estimate}]
We can of course assume $c_{key} = \delta_0^{-2}$, by enlarging $c_{key}$ or shrinking $\delta_0$ as necessary.

Define the set
\begin{gather}
A = \left\{ z \in B_1 : \int_{r_z}^2 \beta^k_{\mu, 2, \bar\eps_\beta}(z, r)^2 dr/r > M \right\}.
\end{gather}
Since $\mu$ is regular, and by \eqref{eq_main_integral_bounds}, there is a Borel set $U \supset A$ with $\mu(U) = \mu(A) \leq \Gamma$.

Define the measure
\begin{gather}
\mu' = \mu \llcorner (B_1 \setminus U).
\end{gather}
Then $\cC$ is trivially still a covering pair for $\mu'$.

By monotonicity of $\beta$ in remark \ref{rem:pointwise-integral}, and our choice of $\mu'$ (supported in $B_1$), we have 
\begin{gather}\label{eqn:cropped-mu-beta}
\int_{r_z/4}^4 \beta_{\mu', 2, \bar\eps_\beta}^k(z, r)^2 dr/r \leq c(k)M \quad \text{for $\mu'$-a.e. $z \in \cC$},
\end{gather}
and therefore, again by monotonicity of $\beta$, 
\begin{gather}\label{eqn:main-proof-int-to-sum}
\sum_{r_z/4 \leq 2^\alpha \leq 2} \beta_{\mu', 2, c(k)\bar\eps_\beta}^k(z, 2^\alpha)^2 \leq c(k)M \quad \text{for $\mu'$-a.e. $z \in \cC$}.
\end{gather}
Here of course $c(k) \geq 1$.  Note we have used the scale bounds of Remark \ref{rem:pointwise-integral} to drop the factor of $4$ above, which will be convenient.

We can suppose the $\eps$ of Theorem \ref{thm:main} is equal to $M$, as enlarging either only weakens the assumptions, and leaves the conclusions unchanged.

First suppose $\eps = M = 0$, so that necessarily $\mu'$ is supported in some $k$-plane $V$.  In this case Theorem \ref{thm:main} becomes trivial.  Set $\cC'_0 = \cC_0 \cap V \cap B_2$, and let $\cC'_+ \subset V\cap B_2$ be (the centers of) a Vitali subcover of
\begin{gather}
\{ B_{r_x/5}(x) : x \in \cC_+ \cap V\cap B_2 \}\, ,
\end{gather}
so that that balls $\{B_{r_x/5}(x)\}_{x \in \cC'_+}$ are disjoint, and
\begin{gather}
B_{r_x}(\cC'_+) \supset B_{r_x/5}(\cC_+ \cap V\cap B_2) \supset \cC_+ \cap V \cap B_2\, . 
\end{gather}

Necessarily $B_{r_x}(\cC') \supset \spt\mu'$, and the packing bounds of conclusion A) follow simply because $\cC' \subset V$, and $\{B_{r_x/5}(x) \cap V : x \in \cC'_+ \}$ are disjoint.  The measure estimate B) is immediate, the non-collapsing C) becomes vacuous, and conclusion D) follows since $\cC'_0 \subset V$.

If $\eps = M > 0$, then we can apply Theorem \ref{thm:core-estimate} to $\mu'$, with $c(k)c_{key}M$ in place of $M$, and covering pair $(\cC, r_x/4)$. 
\end{proof}
\vspace{1cm}

\section{Good tree}\label{section:good-tree}

We define precisely the good tree at $B_1(0)$.  As explained in the proof outline of Section \ref{ss:proof_outline} , the good tree is a succession of coverings at finer and finer scales, which decompose $B_1(0)$ into a part of controlled measure, a family of bad balls with packing estimates, and a subset of a Lipschitz manifold.

Suppose $B_1(0)$ is a good ball as in Definition \ref{def:good-bad-ball} for a non-negative Borel-regular measure $\mu$, with parameters $\rho$ and $m = m_0(n, \rho)M$.  We take $\cC$ a covering pair for $\mu$, with the property that
\begin{gather}\label{eqn:covering-good-hyp}
B_2(0) \cap B_{r_y}(y) \neq\emptyset \implies r_y \leq 1\quad \forall y \in \cC_+.
\end{gather}
In principle we are assuming $r_y \leq 1$, but for technical reasons it is better for tree-chaining (Section \ref{section:finishing}) to allow for big balls in $\cC$ that are far away.

We define inductively families of good balls $\cur{B_{r_i}(g)}_{g\in \cG_i}$, bad balls $\cur{B_{r_i}(b)}_{b\in \cB_i}$, and stop balls $\cur{B_{r_i}(s)}_{s\in \cS_i}$.  At our initial scale $r_0 = 1$, $B_1(0)$ is the only ball, and it is good by assumption.  So $\cG_0 = \{0\}$, and $\cB_0 = \cS_0 = \emptyset$.

Let us now discuss the remainder and excess sets, which are useful in the technical constructions.  In each good ball $B_{r_i}(g)$ we have an associated $L^2$-best $k$-plane $L^k_{ig} := V^k_\mu(g, 8r_i)$.  For technical reasons we take the plane associated with scale $8r_i$ rather than $r_i$.  We define the good ball excess set to be
\begin{gather}
E_{ig} = B_{r_i}(g) \setminus B_{r_{i+1}/50}(L_{ig}),
\end{gather}
and the total excess
\begin{gather}\label{eq:excess}
E_i = \bigcup_{g \in \cG_i} E_{ig}.
\end{gather}
The excess sets will have controlled measure.

Define inductively the remainder set to be the bad and stop balls at all previous scales:
\begin{gather}\label{eq:remainder}
R_i = \bigcup_{\ell=0}^i \ton{ B_{r_\ell}(\cS_\ell) \cup B_{r_\ell}(\cB_\ell) }
\end{gather}

We now move on to the inductive construction of our coverings.  Suppose we have defined the good/bad/stop balls down through scale $r_{i-1}$.  We define the $r_i$-balls as follows.  Let $J_i$ form a maximal $2r_i/5$-net in
\begin{gather}\label{eqn:good-to-cover2}
B_1 \cap \left[ \bigcup_{g\in \cG_{i-1}}  \ton{(B_{r_{i-1}}(g) \cap B_{r_i/40}(L_{i-1,g})} \right] \setminus R_{i-1} ,
\end{gather}
so the balls $\{B_{r_i}(z)\}_{z \in J_i}$ cover \eqref{eqn:good-to-cover2} and the balls $\{B_{r_i/5}(z)\}_{z \in J_i}$ are disjoint. Recalling the definition of stop balls in \ref{def:balls}, we define
\begin{align}
\cS_i &= \{z \in J_i : \text{$B_{r_i}(z)$ is a stop ball} \}, \\
\cG_i &= \{z \in J_i : \text{$B_{r_i}(z)$ is not stop, but is good} \}, \\
\cB_i &= \{z \in J_i : \text{$B_{r_i}(z)$ is not stop, but is bad}\}.
\end{align}
Evidently $J_i = \cS_i \cup \cG_i \cup \cB_i$.

Let $p_{ig}$ be the linear projection to $L_{ig}$, and $p^\perp_{ig}$ the linear projection to $L^\perp_{ig}$.  For each $i$ define the map
\begin{gather}
\sigma_i(x) = x - \sum_{g\in \cG_i} \phi_{ig}(x) p_{ig}^\perp(x - X_{ig}).
\end{gather}
Here $X_{ig}$ is the generalized $\mu$-center of mass of $B_{r_i}(g)$, and $\{\phi_{ig}\}_{g \in \cG_i}$ is a partition of unity subordinate to $\{B_{3r_i}(g)\}_{g \in \cG_i}$, which satisfies
\begin{gather}\label{eq_part_est}
 \spt\phi_{ig}\subset B_{3r_i}(g), \quad \sum_{g\in \cG_i} \phi_{ig} = 1 \text{ on } B_{5r_i/2}(g), \quad |D\phi_{ig}| \leq c(n)/r_i\, .
\end{gather}

We use maps $\sigma_i$ to build inductively a sequence of manifolds $T_i$.  First set $T_0 = L_{0,1}$, and then define
\begin{gather}
T_i = \sigma_i(T_{i-1}) .
\end{gather}
This completes the good tree inductive construction.

Let us make some elementary remarks concerning the good tree construction.
\begin{remark}\label{rem:vitaliness}
Every good/bad/stop ball is contained in $B_2$, and for any $i$, the balls
\begin{gather}
\{ B_{r_i/5}(g) : g \in \cG_i \} \cup \bigcup_{\ell=0}^i \{ B_{r_\ell/5}(x) : x \in \cB_\ell \cup \cS_\ell \}
\end{gather}
are pairwise disjoint.
\end{remark}

\begin{remark}\label{rem:hole-control}
For each $i$, $\sigma_i$ is the identity outside $B_{3r_i}(\cG_i)$.  In particular, since $\rho \leq 1/20$, we have
\begin{equation}\label{eqn:hole-control}
T_i = T_{i-1} \text{ inside } \cup_{\ell=0}^{i-1} B_{4r_\ell/5}(\cB_\ell \cup \cS_\ell) .
\end{equation}
\end{remark}

\begin{definition}\label{def:good-tree}
The construction defined above is said to be the \emph{good tree rooted at $B_1(0)$}, and may be written as $\cT = \cT(B_1)$.  Given such a tree $\cT$ we refer to the collection of all good/bad balls by
\begin{gather}
\cG(\cT) := \cup_i \cG_i, \quad \cB(\cT) := \cup_i \cB_i.
\end{gather}
In a slight abuse of notation, we let $r_g$ and $r_b$ be the associated radius functions for $\cG(\cT)$ and $\cB(\cT)$.  So for example, if $g \in \cG_i \subset \cG(\cT)$, then $r_g = r_i$.\\

For every stop ball $B_{r_i}(s)$ we have either $\mu \ton{B_{r_i}(s)} \leq M r_i^k$, or $s \in B_{r_y + 2r_i}(y_s)$ for some choice of $y_s \in \cC_+$ with $r_i<r_y\leq r_{i-1}$.  Define
\begin{gather}
\cC_+(\cT) = \bigcup_{i=0}^\infty \bigcup_{\substack{s \in \cS_i \\ \mu B_{r_i}(s) > M r_i^k }} y_s
\end{gather}
to be a choice of $y$'s arising in this way (so, at most one $y$ per stop ball).  Define
\begin{gather}
\cC_0(\cT) = \cap_i \overline{B_{r_i}(\cG_i)}
\end{gather}

In a good tree, we refer to the collection of all bad balls as the \emph{tree leaves}.  
\end{definition}

\vspace{2.5mm}

Let us make some remarks concerning the above definition.

\begin{remark}\label{rem:good-containment}
We have $\cC_+(\cT) \subset \cC_+$.  Since every good ball intersects $\spt \mu \subset \cC$, and from \eqref{eqn:avoids-small-good}, we have the inclusions $\cC_0(\cT) \subset \cC_0 \setminus B_{r_x}(\cC_+)$.  And since every good/bad ball is centered in $B_1$, from \eqref{eqn:inside-big-good} we have $\cC(\cT) \subset B_3$.
\end{remark}

\begin{remark}\label{rem:good-mass}
By construction we have the inclusion
\begin{gather}
\bigcup_{i=0}^\infty \bigcup_{\substack{s \in \cS_i \\ \mu \ton{B_{r_i}(s)} > M r_i^k}} B_{r_i}(s) \subset \bigcup_{y \in \cC_+(\cT)} B_{4r_y}(y),
\end{gather}
and consequently we also have the lower bound
\begin{gather}
\mu \ton{B_{4r_y}(y)} > \frac{M}{\rho^k} r_y^k \quad \forall y \in \cC_+(\cT).
\end{gather}
\end{remark}

The following is our main result about good trees in this section:\\
\begin{theorem}[Good Tree Structure]\label{thm:good-tree}
There is a $\delta_1(n,\rho)$ so that if $M > 0$, and
\begin{gather}
\sum_{\alpha \in \dZ \,\, : \,\, r_y < 2^{\alpha} \leq 16} \beta^k_{\mu, 2, \bar\eps_\beta}(y, 2^{\alpha})^2 \leq \delta^2 \leq \delta_1(n,\rho)^2 M \quad \text{for $\mu$-a.e. $y \in \cC$},
\end{gather}
then provided $\bar\eps_\beta \leq c(k)M$, we have for any $i \geq 0$ the packing estimate
\begin{gather}\label{eq_pack_est}
\#\cur{\cG_i} r_i^k + \sum_{\ell = 0}^i\#\cur{\cB_\ell \cup \cS_\ell}r_\ell^k \leq 50^k\, ,
\end{gather}
and for any $i\geq 0$ the measure estimate
\begin{align}\label{eq_measure_est}
&\mu \left[ B_1 \setminus \left( \bigcup_{g\in\cG_i } B_{r_i}(g) \cup \bigcup_{\substack{b\in \cB_\ell \\ \ell\leq i}} B_{r_\ell}(b) \cup \bigcup_{\substack{x \in \cC(\cT) \\ r_x \geq r_i}} B_{4r_x}(x) \right) \right] \leq c(k)M + c(k, \rho)\delta^2.
\end{align}

Moreover, for each $i$, the manifold $T_i\cap \B 2 0$ is a $e^{cM^{-1}\delta^2}$-biLipschitz to a $k$-dimensional disk, and satisfies
\begin{gather}\label{eqn:T-ball-inclusions}
B_{r_i}(\cG_i) \subset B_{2r_i}(T_i \cap B_1) \subset B_{4r_i}(\cG_i) \cup R_i
\end{gather}

There is a fixed Lipschitz manifold $T_\infty$, also $e^{cM^{-1}\delta^2}$-biLipschitz to a disk, so that $T_i \to T_\infty$ in $C^{0,\alpha}(B_1)$ (for any $\alpha < 1$) as $i \to \infty$.  Here $c = c(n, \rho)$.
\end{theorem}

\begin{remark}
Here, as always in this paper, we mean \emph{intrinsically} bi-Lipschitz, in the sense that the inverse mapping $T_\infty \to T_0$ is Lipschitz with respect to the intrinsic distance in $T_\infty$.
\end{remark}

In the language of Definition \ref{def:good-tree}, we have
\begin{corollary}\label{cor:good-tree}
Suppose the hypotheses of Theorem \ref{thm:good-tree}.  Then for the associated good-tree $\cT$ we have
\begin{itemize}
\item[A)] Tree-leaf/bad-ball packing:
\begin{gather}
\sum_{b \in \cB(\cT)} r_b^k \leq 50^k.
\end{gather}

\item[B)] Original ball packing:
\begin{gather}
\sum_{x \in \cC_+(\cT)} r_x^k \leq c(n, \rho) \quad \text{and} \quad r^{k-n}|B_r(\cC(\cT))| \leq c(n, \rho) \quad \forall 0 < r \leq 1.
\end{gather}

\item[C)] Upper measure estimates:
\begin{gather}
\mu \left[ B_1 \setminus \left( \bigcup_{b \in \cB(\cT)} B_{r_b}(b) \cup B_{4r_x}(\cC(\cT)) \right) \right] \leq c(k)M + c(k,\rho) \delta^2,
\end{gather}

\item[D)] Lower measure estimates:
\begin{gather}
\mu \ton{B_r(y)} > \frac{M}{c(n,\rho)} r^k \quad \forall y \in \cC(\cT), \,\, \forall 4r_y < r \leq 1.
\end{gather}

\item[E)] Fine-scale packing structure: $\cC_0(\cT)$ is a closed subset of $T_\infty$.
\end{itemize}
\end{corollary}

For the duration of this Section we assume the hypothesis of Theorem \ref{thm:good-tree}, and the good-tree construction.  Our main goal of this section will now be to prove Theorem \ref{thm:good-tree}.

\subsection{Inductive hypothesis}

The key properties we wish of our manifolds and coverings in the tree construction are the following.  We shall prove them inductively on the scale $i$.  Recall from \eqref{eq:excess} and \eqref{eq:remainder} the definitions of the excess and remainder sets.  We will want the following inductive control:
\begin{enumerate}
 \item ``covering control":
  \begin{equation}\label{eqn:covering-control}
   B_1 \subset \left( \cup_{g\in \cG_i} B_{r_i}(g) \right) \cup R_i \cup \left( \cup_{\ell \leq i-1} E_\ell \right),
  \end{equation}

\item ``radius control'': 
\begin{equation}\label{eqn:scale-control}
y \in \cC \text{ and } B_{2r_i}(\cG_i \cup \cB_i) \cap B_{r_y}(y) \neq\emptyset \implies r_y \leq r_i.
\end{equation}

\item ``graphicality": 
\begin{equation}\label{eqn:graph-control} \tag{$\star_i$}
\forall i, \ \forall g\in \cG_i \, , \ \ \left\{ \begin{array}{l}
 T_i \cap B_{2r_i}(g) = \graph_{\Omega_{ig} \subset L_{ig}} f_{ig} \\
 \abs{f_{ig}}_{C^1_{r_i}} \leq \Lambda(n, \rho)M^{-1/2} \delta \\
 B_{3r_i/2}(g) \cap L_{ig} \subset \Omega_{ig} \subset L_{ig} .\end{array}\right.
\end{equation}
(this last condition is really just to ensure $\Omega_{ig} \neq\emptyset$) 
\end{enumerate}
\vspace{.25cm}

Let us now concentrate on proving these properties.  \\

\subsection{Covering control}

The first two properties are the easiest.

\begin{lemma}[covering control]\label{lem:covering-control}
Conditions ``covering control'' and ``radius control'' hold for every $i$.
\end{lemma}

\begin{proof}
We prove the first inclusion \eqref{eqn:covering-control}.  Trivially this holds at $i = 0$.  Suppose \eqref{eqn:covering-control} holds at scale $i-1$.  By construction, for each $g\in \cG_{i-1}$ we have
\begin{align}
B_{r_i/40}(L_{i-1, g}) \cap B_{r_{i-1}}(g) 
&\subseteq R_{i-1} \cup B_{r_i}(\cG_i\cup \cB_i \cup\cS_i) \\
&= \cup B_{r_i}(\cG_i) \cup R_i.
\end{align}

And therefore
\begin{gather}
\ton{\cup_{g\in \cG_{i-1}} B_{r_{i-1}}(g)} \subseteq \cup_{g\in \cG_{i}} B_{r_i}(g) \cup R_i \cup E_{i-1} .
\end{gather}

By our inductive hypothesis, we've therefore proven \eqref{eqn:covering-control} at scale $i$.

We prove ``radius control'' \eqref{eqn:scale-control}.  This holds at scale $0$ since we assume $r_y \leq 1$ for any $y \in \cC_+$ with $B_{r_y}(y) \cap B_2$.  Suppose \eqref{eqn:scale-control} holds down to scale $i-1$, and consider a $g \in \cG_i$ with $g \in B_{r_y + 2r_i}(y)$.  We can choose a $g' \in \cG_{i-1}$, so that $g \in B_{r_{i-1}}(g')$.

Since $2r_i < r_{i-1}$, by our inductive hypothesis $r_y \leq r_{i-1}$.  Therefore if $r_y > r_i$, by construction $B_{r_i}(g)$ would be a stop ball, a contradiction.
\end{proof}

Some key relations which follows from ``radius control'' are below:
\begin{corollary}\label{lem:good-containment}
For any $y \in \cC_+(\cT)$, with $r_y \leq r_i$, we have
\begin{gather}\label{eqn:inside-big-good}
B_{r_y}(y) \cap B_{2r_{i-1}}(\cG_{i-1} \cup \cB_{i-1}) \neq \emptyset.
\end{gather}

Conversely, if $y \in \cC_+$ and $r_y > r_i$, then
\begin{gather}\label{eqn:avoids-small-good}
B_{r_y}(y) \cap B_{2r_i}(\cG_i \cup \cB_i) = \emptyset.
\end{gather}

For any $y \in \cC(\cT)$, we have
\begin{gather}\label{eqn:lem-orig-good-disjoint}
B_{\frac{r_y}{2\rho^2}}(y) \cap B_{r_x}(x) = \emptyset \quad \forall x \in \cC \text{ such that } r_x > \rho^{-3} r_y.
\end{gather}
\end{corollary}

\begin{proof}
Relation \eqref{eqn:inside-big-good} is immediate from the construction, and \eqref{eqn:avoids-small-good} follows from ``radius control.''  We prove \eqref{eqn:lem-orig-good-disjoint}.  If $r_y > 0$ choose $r_i$ so that $\rho r_i \leq r_y < r_i$, otherwise take any $r_i > 0$.  By relation \eqref{eqn:inside-big-good}, we deduce
\begin{gather}
B_{\frac{r_y}{2\rho^2}}(y) \subset B_{2r_{i-2}}(\cG_{i-2} \cup \cB_{i-2}).
\end{gather}
Using \eqref{eqn:avoids-small-good}, we have $B_{r_y}(y) \cap B_{r_x}(x) = \emptyset$ for any $x \in \cC$ with $r_x \geq r_{i-2}$.  This proves \eqref{eqn:lem-orig-good-disjoint}.
\end{proof}

We note that Lemma \ref{lem:covering-control} also gives control over the $\beta$ numbers centered at $\cG_i$.
\begin{corollary}\label{cor_good_y}
For any $g\in \cG_i$, there is a $y \in \cC \cap B_{2r_i}(g)$ with $r_y \leq r_i$ and
\begin{gather}
\sum_{\alpha \in \dZ \,\, : \,\, r_y < 2^\alpha \leq 16} \beta^k_{\mu, 2, \bar\eps_\beta}(y, 2^\alpha)^2 \leq \delta^2.
\end{gather}
We have
\begin{gather}\label{eq_beta<delta}
\beta(g, 8 r_i)^2 \leq c(k) \beta(y, 11r_i)^2 \leq c(k) \delta^2 .
\end{gather}
\end{corollary}

\begin{proof}
This is a direct corollary of the construction and Remark \ref{rem:pointwise-integral}. In particular, 
\end{proof}

\subsection{Tilting control in good balls}\label{subsection:good-tilting-map}

We can apply theorem \ref{thm:tilting-control} to obtain Reifenberg-type control between subsequent good balls' best planes.

\begin{lemma}\label{lem:good-ball-tilting}

Take $g' \in B_{2r_{i-1}}(g)$ for some $g'\in \cG_i$ and $g\in \cG_{i-1}$.  Then for $c = c(n,\rho)$ we have
\begin{gather}\label{eq_gt_1}
d_H(L_{ig'} \cap B_{8r_i}(g'), L_{i-1,g}\cap B_{8r_i}(g') ) \leq cM^{-1/2}(\beta(g', 8r_i) + \beta(g, 8r_{i-1})) r_i .
\end{gather}
And hence
\begin{align}
d_G(L_{ig'}, L_{i-1,g}) &\leq c M^{-1/2}(\beta(g', 8r_i) + \beta(g, 8r_{i-1})) \label{eq_gt_2} \\
d(X_{ig'}, L_{i-1,g}) &\leq cM^{-1/2} \beta(g, 8r_{i-1}) r_i \, .\label{eq_gt_3}
\end{align}

Similarly, if $g''\in \cG_i$ satisfies $|g''- g'| < 7 r_i$, then
\begin{align}
d_G(L_{ig'}, L_{ig''}) &\leq c M^{-1/2} (\beta(g', 8r_i) + 2\beta(g, 8r_{i-1}) + \beta(g'', 8r_{i}))\label{eq_gt_4} \\
d(X_{ig''}, L_{ig'}) &\leq c M^{-1/2} \beta(g', 8r_i) r_i\, .\label{eq_gt_5}
\end{align}

More coarsely, we have the following inequalities:
\begin{align}
d_G(L_{ig'}, L_{i-1,g}) \leq c M^{-1/2} \delta,& \quad d(X_{ig'}, L_{i-1,g}) \leq cM^{-1/2} \delta r_i\label{eq_gt_6} \\
d_G(L_{ig''}, L_{ig'}) \leq c M^{-1/2} \delta,& \quad d(X_{ig''}, L_{ig'}) \leq c M^{-1/2} \delta r_i \, ,\label{eq_gt_7}
\end{align}
where $c$ is always a constant $c=c(n,\rho)$.
\end{lemma}

\begin{proof}
Apply Theorem \ref{thm:tilting-control} at scale $r_{i-1}$ to deduce first two equations \eqref{eq_gt_1} and \eqref{eq_gt_2}, recalling our choice of $m = m_0(n)M$. Equation \eqref{eq_gt_3} follows by Proposition \ref{prop:COM-control}. 
Precisely the same reasoning works with $g'$ replaced by $g''$, so \eqref{eq_gt_4} follows from the triangle inequality.  The estimate in \eqref{eq_gt_5} is again deduced from Proposition \ref{prop:COM-control}, since $B_{r_i}(g'') \subset B_{8r_i}(g')$.

Evidently, the coarse inequalities \eqref{eq_gt_6} and \eqref{eq_gt_7} are direct consequences of the previous estimates and the coarse estimate \eqref{eq_beta<delta}.
\end{proof}

With this control over the tilting of best planes we can prove good $C^1$ bounds on our maps $\sigma_i$ and manifolds $T_i$.  We prove two classes of estimates, following \cite{simon_reif}.  The ``coarse estimates'' give bounds on $\sigma_i$ in terms of the graphical bounds of $T_{i-1}$.  From this one can deduce slightly worse graph bounds of $T_i$ at scale $i$, however this technique cannot hope to give \eqref{eqn:graph-control} as the bounds would degenerate as $i \to \infty$.

The ``squash estimates'' say that provided $(\star_{i-1})$ holds, we get good graphical bounds at scale $i$ which are \emph{independent} of $\Lambda$.  This is the key to demonstrating \eqref{eqn:graph-control} inductively.  Then by going \emph{two} scales back, and linking the squash estimates with the coarse estimates, we get bounds on $\sigma_i$ which are independent of scale (the ``refined coarse estimates'').

\begin{lemma}[``coarse estimates'']\label{lem:coarse-estimates}
Take $g' \in \cG_i\cap B_{r_{i-1}}(g)$ for some $g\in \cG_{i-1}$.  Suppose
\begin{gather}
\left\{ \begin{array}{l}
 T_{i-1} \cap B_{2r_{i-1}}(g) = \graph_{\Omega_{i-1, g} \subset L_{i-1, g}} f \\
 |f|_{C^1_{r_{i-1}}} \leq \gamma . \end{array} \right.
\end{gather}
Then
\begin{align}
&\sup_{B_{3r_i}(g') \cap T_{i-1}} \left( r_i^{-1} |\sigma_i - Id| + |D^\perp_{T_{i-1}} (\sigma_i - Id)|  + |D^\top_{T_{i-1}}(\sigma_i - Id)|^{1/2} \right)\notag \\
&\quad \leq c(n, \rho) (\gamma + M^{-1/2} \beta(g, 8r_{i-1}) ) \\
&\quad \leq c_{cs}(n, \rho) (\gamma + M^{-1/2} \delta) .\notag
\end{align}

Here $D_{T_{i-1}}$ is the derivative operator on functions $T_{i-1} \to \R^n$, and $D^\top_{T_{i-1}}$, $D^\perp_{T_{i-1}}$ are the projections of $D_{T_{i-1}}$ to the tangent, normal bundles of $T_{i-1}$.
\end{lemma}

\begin{proof}
Throughout the proof $c$ will denote a constant depending only on $n, \rho$.  Recall that
\begin{gather}
\sigma_i(x) - x = -\sum_{z\in \cG_i} \phi_{iz}(x) p_{iz}^\perp(x - X_{iz}) ,
\end{gather}
and consequently for any vector $V$ we have
\begin{gather}
D_V (\sigma_i(x) - x) = -\sum_{z\in \cG_i} (D_V \phi_{iz}(x)) p_{iz}^\perp(x - X_{iz}) - \sum_{z\in \cG_i} \phi_{iz}(x) p^\perp_{iz}(V) .
\end{gather} 

By assumption we know
\begin{gather}
B_{3 r_i}(g') \cap T_{i-1} = \graph_{\Omega_{i-1, g} \subset L_{i-1, g}} f\quad \text{ with } \quad |f|_{C^1_{r_{i-1}}} \leq \gamma .
\end{gather}

Now fix an $x \in B_{3r_i}(g') \cap T_{i-1}$, and write $x = (\zeta, f(\zeta))$ for $\zeta \in L_{i-1, g}$.  Suppose $\phi_{iz}(x) > 0$, so in particular $|g' - z| < 6r_i$.  We calculate that
\begin{align}
|p^\perp_{iz}(x - X_{iz})| 
&\leq |p^\perp_{iz} - p^\perp_{i-1, g}| |x - X_{iz}| + |p^\perp_{i-1, g}(x - \zeta)| + |p^\perp_{i-1, g}(\zeta - X_{iz})| \notag\\
&\leq c d_G(L_{iz}, L_{i-1, g}) r_i + \gamma r_{i-1} + d(X_{iz}, L_{i-1, g}) \\
&\leq c M^{-1/2} \left( \beta(z, 8r_i) + \beta(g, 8 r_{i-1}) \right) r_i + \gamma r_{i-1}\notag
\end{align}

There are at most $c(n)$ $z\in \cG_i$ such that $\phi_{iz}(x) > 0$.  We therefore obtain
\begin{align}
|\sigma_i(x) - x| &\leq c \gamma r_i + c M^{-1/2} \left( \beta(g, 8r_{i-1}) + \sum_{z \in \cG_i\cap B_{6r_i}(g')} \beta(z, 8r_i) \right) r_i \notag\\
&\leq c \gamma r_i + c M^{-1/2} \beta(g, 8r_{i-1}) r_i
\end{align}

We bound the gradient term.  Fix $V$ a unit vector in $T_x T_{i-1}$.  The first term of the gradient is bounded in precisely the same way, since $|D \phi_{ig}| \leq 10/r_i$.  The second term can be bounded like
\begin{align}
|p^\perp_{iz}(V)|
&\leq |p^\perp_{iz} - p^\perp_{i-1,g}| + |p_{i-1,g}^\perp(V)| \notag\\
&\leq 2 d_G(L_{iz}, L_{i-1,g}) + 2 d_G(L_{i-1,g}, T_x T_{i-1}) \\
&\leq c M^{-1/2} (\beta(z, 8r_i) + \beta(g, 8r_{i-1})) + c \gamma\notag .
\end{align}

The extra tangent power comes about because, given a unit vector $W \in T_x T_{i-1}$, and $Y$ arbitrary, we have
\begin{align}
|W \cdot p^\perp_{iz}(Y)| 
&\leq \left(|p^\perp_{iz} - p^\perp_{i-1, g}| + |p^\perp_{i-1, g}(W)|\right) |p_{iz}^\perp(Y)| \notag\\
&\leq \left( 2d_G(L_{iz}, L_{i-1, g}) + \gamma \right) |p_{iz}^\perp(Y)| \\
&\leq c \left[ M^{-1/2}\beta(g, 8r_{i-1}) + \gamma \right] |p_{iz}^\perp(Y)| \notag. \qedhere
\end{align}

\end{proof}
\vspace{.25cm}

\begin{lemma}[``squash estimates'']\label{lemma_squash_+}
Take $g' \in \cG_i\cap B_{r_{i-1}}(g)$ for some $g\in \cG_{i-1}$, and suppose $(\star_{i-1})$ holds.  Then provided $\delta \leq \delta_2(n, \rho, \Lambda) M^{1/2}$, we have 
\begin{gather}
\left\{ \begin{array}{l} T_i \cap B_{2r_i}(g') = \graph_{\Omega_{ig'} \subset L_{ig'}} f \\
 |f|_{C^1_{r_i}} \leq c(n, \rho)M^{-1/2} \beta(g, 8r_{i-1}) \leq c_{sq}(n,\rho)M^{-1/2} \delta \\
  B_{3/2}(g') \cap L_{ig'} \subset \Omega_{ig'} \subset L_{ig'}. \end{array} \right.
\end{gather}
\end{lemma}

\begin{proof}
As before $c$ will denote a generic constant depending only on $n, \rho$.  For convenience define $p = p_{ig'}$, $L=L_{ig'}$ and $X=X_{ig'}$.  We have
\begin{align}
\sigma_i(x) 
&= X + p(x - X) + p^\perp(x - X) - \sum_{z\in \cG_i} \phi_{iz}(x)p_{iz}^\perp(x - X_{iz}) \notag\\
&= X + p(x - X) + e(x) .
\end{align}
For $x \in B_{5r_i/2}(g')$, since $\sum_{z\in \cG_i} \phi_{iz}(x) = 1$, we have
\begin{align}
e(x) 
&= \sum_{z\in \cG_i} \phi_{iz}(x) \left( p^\perp(x - X) - p_{iz}^\perp(x - X_{iz}) \right) \notag\\
&= \sum_{z\in \cG_i} \phi_{iz}(x) \left( p^\perp (X_{iz} - X) + (p^\perp - p_{iz}^\perp)(x - X_{iz}) \right) .
\end{align}

If $\phi_{iz}(x) > 0$, then $|g' - z| < 6 r_i$, and so by Lemma \ref{lem:good-ball-tilting} we have
\begin{align}
|p^\perp(X_{iz} - X)| \leq cM^{-1/2} \beta(g', 8r_i) r_i\, .
\end{align}
and
\begin{align}
|(p^\perp - p_{iz}^\perp)(x - X_{iz})| 
&\leq 2d_G(L, L_{iz}) |x - X_{iz}|\notag \\
&\leq c M^{-1/2} (\beta(g, 8 r_{i-1}) + \beta(g', 8 r_i) + \beta(z, 8r_i)) r_i .
\end{align}
Since there are at most $c(n)$ $z \in \cG_i$ for which $\phi_{iz}(x) > 0$, we obtain that
\begin{align}
|e(x)| &\leq c M^{-1/2} \left( \beta(g, 8 r_{i-1}) + \sum_{z \in B_{6r_i}(g')} \beta(z, 8r_i) \right) r_i \leq c M^{-1/2} \beta(g, 8r_{i-1}) r_i
\end{align}

For any unit vector $V$ we have
\begin{gather}
D_V e = \sum_{z\in \cG_i} (D_V \phi_{iz}) \left( p^\perp(x - X) - p_{iz}^\perp(x - X_{iz}) \right) - \sum_{z\in \cG_i} \phi_{iz}(x) (p^\perp - p_{iz}^\perp)(V) .
\end{gather}
Use the same estimates on the first term, with $|D\phi_{iz}| \leq 10/r_i$.  The second term we can estimate as
\begin{gather}
|(p^\perp - p_{iz}^\perp)(V)| \leq 2 d_G(L, L_{iz})  .
\end{gather}
We deduce
\begin{gather}
|e|_{C^1_{r_i}(B_{5r_i/2}(g'))} \leq c M^{-1/2} \beta(g, 8r_{i-1})
\end{gather}

By $(\star_{i-1})$ we have
\begin{gather}
B_{3r_i}(g') \cap T_{i-1} = \graph_{\Omega_{i-1,g} \subset L_{i-1, g}} f, \quad |f|_{C^1_{r_{i-1}}} \leq \Lambda \delta_2, \quad \Omega_{i-1, g} \supset B_{5r_{i-1}/2}(g) \cap L_{i-1, g},
\end{gather}
for some $g \in \cG_{i-1}$ with $g' \in B_{r_{i-1}}(g)$.  From Lemma \ref{lem:good-ball-tilting}, we have
\begin{gather}
d(X_{ig'}, L_{i-1, g}) \leq c \delta_2 r_i \quad\text{and}\quad d_H(L \cap B_{8r_i}(g'), L_{i-1, g} \cap B_{8r_i}(g')) \leq c \delta_2 r_i ,
\end{gather}
and because $B_{5r_i/2}(g') \subset B_{3r_{i-1}/2}(g)$ we have $\Omega_{i-1, g} \supset B_{5r_i/2}(g') \cap L_{i-1,g}$.


Apply the Squash Lemma \ref{lem:squash} at scale $r_i$ to deduce
\begin{gather}
\sigma_i(B_{5r_i/2}(g') \cap T_{i-1}) \cap B_{2r_i}(g') = \graph_{U \subset L} f, \quad |f|_{C^1_{r_i}} \leq 4 |e|_{C^1_{r_i}} \leq c M^{-1/2} \beta(g, 8r_{i-1}),
\end{gather}
where $B_{3r_i/2}(g') \cap L \subset U \subset L$.

By the coarse estimates we know $|\sigma_i(x) - x| \leq c(1+\Lambda) \delta_2 r_i$.  Thus if $c (1+\Lambda)\delta_2$ is universally small, we can replace $\sigma_i(B_{5r_i/2}(g') \cap T_{i-1}) \cap B_{2r_i}(g')$ in the above with $\sigma_i(T_{i-1}) \cap B_{2r_i}(g')$.
\end{proof}
\vspace{.25cm}

By plugging the squash estimates at scale $i-1$ into the coarse estimates at scale $i$, and using the relation
\begin{gather*}
|D_{T_{i-1}} \sigma_i| = \sqrt{ |D^\perp_{T_{i-1}}(\sigma_i - Id)|^2 + |D^\top_{T_{i-1}}\sigma_i|^2}
\end{gather*}
we obtain
\begin{lemma}[refined estimates]\label{lem:final-coarse}
Suppose $(\star)$ holds through $i-1$.  If $i \geq 2$, then given $g'\in \cG_i$, $g'' \in \cG_{i-1}$, and $g \in \cG_{i-2}$ so that $g' \in B_{r_{i-1}}(g'')$ and $g'' \in B_{r_{i-2}}(g)$, then we have the estimate
\begin{align}
&\sup_{B_{3r_i}(g') \cap T_{i-1}} \left( r_i^{-1} |\sigma_i - Id| + |D^\perp_{T_{i-1}} (\sigma_i - Id)| + |D^\top_{T_{i-1}}(\sigma_i - Id)|^{1/2} \right) \nonumber \\
&\quad \leq c(n, \rho) M^{-1/2} \beta(g, 8r_{i-2}) \label{eqn:coarse-estimate-D} \\
&\quad \leq c_{rcs}(n, \rho)M^{-1/2} \delta. \label{eqn:coarse-estimate-delta}
\end{align}
As a consequence, we have
\begin{gather}\label{eqn:final-course-concl3}
\sup_{x \in B_{3r_i}(g') \cap T_{i-1}} |D_{T_{i-1}}\sigma_i| \leq 1 + c(n,\rho) M^{-1} \beta(g, 8r_{i-2})^2 .
\end{gather}

If $i = 1$, then for any $g\in \cG_1$ we have
\begin{align}
&\sup_{B_{3r_1}(g) \cap T_{0}} \left( r_1^{-1} |\sigma_1 - Id| + |D^\perp_{T_{0}} (\sigma_1 - Id)|  + |D^\top_{T_0}(\sigma_1 - Id)|^{1/2} \right) \notag \\
&\quad \leq c(n, \rho)M^{-1/2} \beta(0, 8) \\
&\quad \leq c_{rcs}(n, \rho) M^{-1/2} \delta.\notag
\end{align}
\end{lemma}

The above estimates suffice to prove ``graph control'':
\begin{theorem}
We have \eqref{eqn:graph-control} for every $i \geq 0$.
\end{theorem}

\begin{proof}
Take $\Lambda = c_{sq}$, the constant from the squash estimates.  Then ensure $\delta_1 \leq \delta_2(n, \rho, c_{sq})$.
\end{proof}

\begin{corollary}\label{cor:diffeomorphism}
Provided $\delta_1(n, \rho)$ is sufficiently small, $\sigma_i$ is a diffeomorphism onto its image, and $\sigma_i^{-1}$ satisfies the same coarse estimates (Lemma \ref{lem:final-coarse}) as $\sigma_i$.
\end{corollary}

\begin{proof}
Take $g' \in B_{r_{i-1}}(g)$.  Hypothesis \eqref{eqn:graph-control} implies $T_{i-1} \cap B_{5r_i}(g')$ is $C^1$ close to $L_{i-1,g}$.  The coarse estimates of Lemma \ref{lem:final-coarse} show the $\sigma_i$ is locally invertible, and $\sigma_i(T_{i-1}|_{B_{3r_i}(g')}) \subset B_{5r_i}(g')$.
\end{proof}
\vspace{.25cm}

\subsection{Properties of the approximating manifolds}\label{section:prop_T}

We show that the manifolds $T_i$ are uniformly bi-Lipschitz to $T_0$, and approach in the Hausdorff sense a limit $T_\infty$, sharing the same bi-Lipschitz bound.  For $i > j \geq 1$, define the maps
\begin{gather}
\tau_{i,j} = \sigma_i \circ \cdots \circ \sigma_j : T_{j-1} \to T_i .
\end{gather}
By Corollary \ref{cor:diffeomorphism} each $\tau_{i,j}$ is a diffeomorphism between $T_{j-1}$ and $T_i$.

\begin{lemma}\label{lem:lipschitz-bounds}
For every $i \geq j$, we have
\begin{gather}\label{eq_est_tau_unif}
|\tau_{i, j}(x) - x| \leq c_2 M^{-1/2} \delta r_j , 
\end{gather}
and
\begin{gather}\label{eq_grad_est}
|D\tau_{i,j}| \leq e^{c_2 M^{-1} \delta^2}, \quad |D\tau^{-1}_{i,j}| \leq e^{c_2 M^{-1} \delta^2} .
\end{gather}
Here $c_2 = c_2(n, \rho)$.

\end{lemma}

\begin{proof}
The first estimate follows immediately from the coarse estimates: since $|\sigma_k(x) - x| \leq c_{rcs} M^{-1/2} \delta r_k$ for every $x \in T_{k-1}$, we have
\begin{gather}
\abs{\tau_{i,j}(x)-x}\leq c_{rcs} M^{-1/2} \delta \sum_{k=j}^{i-1} r_k\leq \frac{c_{rcs}}{1-\rho} M^{-1/2} \delta r_j \quad \forall x \in T_{i-1}.
\end{gather}

In order to prove the gradient estimates, let $x = x_{j-1} \in T_{j-1}$, and let $x_k=\tau_{k,j}(x)\in T_k$. By definition of $\tau_{ij}$, we have
\begin{gather}
 \abs{D_{T_{j-1}} \tau_{ij}(x)}\leq \prod_{k=j}^i \abs{D_{T_{k-1}} \sigma_k(x_{k-1})}\, .
\end{gather}
(Here, as usual, the gradient of $\sigma_k$ is computed on the tangent space of $T_{k-1}$.)

Let us bound each term in the above product.  If $x_{k-1} \not\in B_{3r_k}(\cG_k)$ then $|D\sigma_k(x_{k-1})| = 1$ and there is nothing to show.  Otherwise, choose $g'_k \in \cG_k$ so that $x_{k-1} \in B_{3r_{k}}(g'_{k})$.  By Lemma \ref{lem:final-coarse}, there is a $g_{k-2} \in \cG_{k-2}$ so that $g'_k \in B_{r_{k-1}+r_{k-2}}(g_{k-2})$, and
\begin{gather}
|D\sigma_k(x_{k-1})| \leq 1 + c M^{-1} \beta(g_{k-2}, 8r_{k-2})^2.
\end{gather}

We can assume that $x_{i-1} \in B_{3r_i}(\cG_i)$.  By assumption and ``radius control'' \eqref{eqn:scale-control}, we can find a $y \in \cC\cap B_{2r_i}(g_i')$ with $r_y < r_i$ and such that \eqref{eq_sum_beta_bounds} holds (see also Corollary \ref{cor_good_y}).  Using the refined squash estimates of Lemma \ref{lem:final-coarse}, we have
\begin{align}
|y - g_{k-2}| 
&\leq |y - x_{i-1}| + |x_{i-1} - x_{k-1}| + |x_{k-1} - g'_k| + |g'_k - g_{k-2}| \\
&< 4r_i + cM^{-1/2}\delta r_k + 3r_k + r_{k-1} + r_{k-2} \\
&\leq 2r_{k-2} .
\end{align}
From the monotonicity of $\beta$ given in Remark \ref{rem:pointwise-integral}, we deduce that
\begin{gather}
|D\sigma_k(x_{k-1})| \leq 1 + cM^{-1} \beta(y, 10 r_{k-2})^2 \quad \forall k = j, \ldots, i.
\end{gather}
Similarly, we have $|D\sigma^{-1}_k(x_k)| \leq 1 + c M^{-1} \beta(y, 10 r_{k-2})^2$ for each $k$.

Therefore
\begin{align}
\log |D_{T_{j-1}} \tau_{ij}(x)|
&\leq \sum_{k=j}^i \log \left(1 + cM^{-1} \beta(y, 10 r_{k-2})^2 \right) \\
&\leq c M^{-1} \sum_{k=j-2}^i \beta(y, 10 r_k)^2 \\
&\leq c M^{-1} \delta^2.
\end{align}
The bound for $\tau_{i,j}^{-1}$ follows in precisely the same manner.
\end{proof}

The above Lemma shows that the $\tau_{i,1}$ are uniformly Cauchy in $C^{0,\alpha}$, and therefore we have a map $\tau_\infty$ so that $\tau_{i, 1} \to \tau_\infty$ uniformly.  Using the gradient bounds above, we deduce $\tau_\infty$ is bi-Lipschitz wrt the intrinsic distance in $T_\infty$, with constants
\begin{gather}
e^{-c_2M^{-1}\delta^2} |x - y| \leq |\tau_\infty(x) - \tau_\infty(y)| \leq e^{c_2 M^{-1}\delta^2} |x - y| \, .
\end{gather}
Note that this in particular implies that $\tau_\infty$ is a Lipschitz mapping from $\R^k$ to $\R^n$ with Lipschitz constant $e^{c_2 M^{-1}\delta^2}$.

Define $T_\infty = \tau_\infty(T_0)$, so $T_\infty$ is bi-Lipschitz to $T_0$.  We immediately have that
\begin{equation}\label{eqn:manifold-volume-control}
e^{-kc_2 M^{-1} \delta^2} \haus^k(T_0 \cap B_3) \leq \haus^k (T_\infty \cap B_3) \leq e^{kc_2 M^{-1} \delta^2} \haus^k(T_0 \cap B_3) \, .
\end{equation}

\vspace{.25cm}

\subsection{Packing estimate}
The construction of the balls and the $T_i$ ensure that for every $i$, the collection of good $r_i$ balls, and bad/stop balls at scales $r_i, r_{i-1}, \ldots, r_1$ form a Vitali covering not only in $\dR^n$, but also \emph{in the manifold $T_i$}.  This allows us to obtain the following packing estimates:

\begin{lemma}\label{lem:cutting-holes}
We have
\begin{gather}
\#\cur{\cG_i} r_i^k + \sum_{\ell \leq i} \#\cur{\cB_\ell \cup \cS_\ell} r_\ell^k  \leq 30^k e^{c_3 M^{-1} \delta^2},
\end{gather}
provided $\delta_1(n, \rho)$ is sufficiently small.  Here $c_3 = c_3(n, \rho)$.
\end{lemma}

\begin{proof}

We first demonstrate that each good/bad/stop $r_i/5$-ball intersects a definite amount of the $T_{i-1}$.  More precisely, let $x \in (\cB_i \cup \cG_i \cup \cS_i) \cap B_{r_{i-1}}(g)$ for some $g\in \cG_{i-1}$.  We have $\dist(x, L_{i-1, g}) < r_i/40$ by construction, and hence by graphicality we know
\begin{equation}\label{eqn:centers-close-to-T}
\dist(x, T_{i-1}) \leq \Lambda \delta_1 r_{i-1} + r_i/40 < r_i/30\, ,
\end{equation}
for $\delta_1$ sufficiently small (depending only on $n, \rho$).  Since $B_{r_{i-1}/5}(x) \subset B_{2r_{i-1}}(g)$, and we have graphical control over $T_{i-1}$ in this region, we deduce
\begin{gather}\label{e:cutting-holes:1}
\haus^k(T_{i-1} \cap B_{r_i/5}(x)) \geq \omega_k (r_i/10)^k ,
\end{gather}
again provided $\delta_1(n, \rho)$ is small.

Since the $T_i=T_\ell$ inside bad/stop balls of larger radii (Remark \ref{rem:hole-control}), we deduce from above that also any bad/stop $r_\ell/5$-ball cuts through a definite amount of $T_{i-1}$, for any $r_\ell \geq r_i$.  The collection of good $r_i$-balls, and the bad/stop balls of scale $r_i, r_{i-1}, \ldots, r_1$, form a Vitali collection (Remark \ref{rem:vitaliness}), and in particular the associated regions $T_{i-1} \cap B_{r_\ell/5}(x_\ell)$ form disjoint subsets of $T_{i-1}$.  Combining this with \eqref{e:cutting-holes:1} gives the volume estimate
\begin{align}
&\omega_k \#\{\cG_i\} (r_i/10)^k + \sum_{\ell=1}^i \omega_k \#\{\cB_\ell \cup \cS_\ell\} (r_i/10)^k \, ,\\
&\leq \haus^k(T_{i-1} \cap B_{r_i/5}(\cG_i)) + \sum_{\ell=1}^i \haus^k(T_{\ell-1} \cap B_{r_\ell/5}(\cB_\ell \cup \cS_\ell)) \, ,\\
&= \haus^k(T_{i-1} \cap B_{r_i/5}(\cG_i)) + \sum_{\ell=1}^i \haus^k(T_{i-1} \cap B_{r_\ell/5}(\cB_\ell \cup \cS_\ell))\, , \\
&\leq \haus^k(T_{i-1} \cap B_3)\, , \\
&\leq e^{kc_2M^{-1} \delta^2} \omega_k 30^k
\end{align}
The last inequality follows from \eqref{eqn:manifold-volume-control}.  This proves the Lemma.
\end{proof}

\vspace{.25cm}

\subsection{Bounding excess and Finishing the Proof}
In this section, we bound the measure of the excess set in our construction, and finish the proof of Theorem \ref{thm:good-tree}


\begin{lemma}\label{lem:excess-bound}
Define
\begin{gather}
 Z_i = B_2 \cap \bigcup_{\substack{ y \in \cC \\ r_y \leq r_i}} B_{r_y}(y) .
\end{gather}
We have for each $i$
\begin{gather}\label{eqn:individual-excess-bound}
\sum_{g\in \cG_i} \mu \ton{E_{ig}} \leq c(k,\rho) e^{c_3 M^{-1} \delta^2} \int_{T_\infty \cap B_{3r_i}(Z_i)} \beta(z, 10r_i)^2 d\haus^k(z) .
\end{gather}

Therefore,
\begin{gather}\label{eqn:total-excess-bound}
\sum_i \mu \ton{E_{i}} \leq c(k,\rho) e^{2c_3 M^{-1} \delta^2} \delta^2 .
\end{gather}
\end{lemma}

\begin{proof}
For each $g\in \cG_i$, by definition of the $L^2$-best plane we have
\begin{gather}
\mu \ton{E_{ig}} \leq c(k,\rho) r_i^k \beta(g, 8r_i)^2 .
\end{gather}

As in equation \eqref{e:cutting-holes:1}, we have $d(g, T_i) \leq r_i/30$ and $\haus^k(T_i \cap B_{r_i/5}(g)) \geq \omega_k (r_i/10)^k$, provided $\delta_1(n, \rho)$ is sufficiently small.  Therefore we deduce that
\begin{gather}\label{eqn:individual-excess-integral}
\mu \ton{E_{ig}} \leq c \int_{T_i \cap B_{r_i/5}(g)} \beta(z, 9r_i)^2 d\haus^k(z).
\end{gather}

Since the $r_i/5$-balls are disjoint, we can control the sum over $\cG_i$ by an integral over $T_\infty$.  Using ``radius control'' and the fact that any good ball contains some mass, we must have
\begin{gather}
B_{r_i}(g) \cap Z_i \neq \emptyset \quad\forall g\in \cG_i.
\end{gather}
We therefore calculate
\begin{align}
\sum_{g \in \cG_i} \mu \ton{E_{ig}}
&\leq c \int_{T_i \cap B_{2r_i}(Z_i)} \beta(z, 9r_i)^2 d\haus^k(z)\, , \label{eqn:bound-excess-1} \\
&\leq c e^{c_3M^{-1}\delta^2} \int_{T_\infty \cap B_{5r_i/2}(Z_i)} \beta(\tau^{-1}_{\infty,i}(z), 9r_i)^2 d\haus^k \, ,\label{eqn:bound-excess-2} \\
&\leq c(k,\rho) e^{c_3M^{-1} \delta^2} \int_{T_\infty \cap B_{3r_i}(Z_i)} \beta(z, 10r_i)^2 d\haus^k \, . \label{eqn:bound-excess-3}
\end{align}
In inequalities \eqref{eqn:bound-excess-2}, \eqref{eqn:bound-excess-3} we used that
\begin{gather}
|\tau^{-1}_{\infty, i} - Id| \leq c(n,\rho) \delta_1 r_i \leq r_i/2,
\end{gather}
ensuring $\delta_1(n,\rho)$ is sufficiently small.  In \eqref{eqn:bound-excess-2} we also used
\begin{gather}
|J\tau_{\infty, i}^{-1}| \leq e^{c_3M^{-1} \delta^2}.
\end{gather}
This establishes relation \eqref{eqn:individual-excess-bound}.

We prove relation \eqref{eqn:total-excess-bound}.  Suppose $z \in B_{3r_j}(Z_j)$.  Then we can find a $y \in \cC$ with $r_y \leq r_j$, and $|z - y| < 4r_j$, and for which \eqref{eq_sum_beta_bounds} holds.  We deduce 
\begin{gather}
\sum_{i \geq j} \beta(z, 10r_i)^2 \leq c(k) \sum_{i \geq j} \beta(y, 16r_i)^2 \leq c(k)\delta^2.
\end{gather}
Since $j$ and $z$ were arbitrary, we have
\begin{gather}
\sum_i 1_{B_{3r_i}(Z_i)}(z) \beta(z, 10r_i)^2 \leq c(k) \delta^2 \quad \forall z \in B_2.
\end{gather}

We now use Fubini's theorem and estimate \eqref{eqn:individual-excess-bound} to deduce 
\begin{align}
\sum_i \mu \ton{E_i}
&\leq  \sum_i \sum_{g \in \cG_i} \mu \ton{E_{ig}} \, , \\
&\leq c(k,\rho) e^{c_3 M^{-1} \delta^2} \sum_i \int_{T_\infty \cap B_{3r_i}(Z_i)} \beta(z, 10r_i)^2 d\haus^k \, ,\\
&= c  e^{c_3 M^{-1} \delta^2}\int_{T_\infty} \sum_i 1_{B_{3r_i}(Z_i)}(z) \beta(z, 10r_i)^2 d\haus^k \, ,\\
&\leq c  e^{c_3 M^{-1} \delta^2}\delta^2 \haus^k(T_\infty \cap B_5) \, , \\
&\leq c(k,\rho) \delta^2 e^{2c_3M^{-1} \delta^2} 50^k .
\end{align}
This completes the proof of Lemma \ref{lem:excess-bound}.
\end{proof}

To finish the proof of Theorem \ref{thm:good-tree} now ensure that $\delta_1$ small enough so that $30^k e^{c_3 \delta_1^2} \leq 50^k$.  The inclusions \eqref{eqn:T-ball-inclusions} follow from ``covering control'' and ``graphical control,'' ensuring $\delta_1$ is sufficiently small.  To see the bound \eqref{eq_measure_est}, observe that by ``covering control'' and Remark \ref{rem:good-mass} we have the inclusion
\begin{gather}
B_1 \setminus \left( \bigcup_{g\in\cG_i } B_{r_i}(g) \cup \bigcup_{\substack{b\in \cB_\ell \\ \ell\leq i}} B_{r_\ell}(b) \cup \bigcup_{\substack{x \in \cC(\cT) \\ r_x \geq r_i}} B_{4r_x}(x) \right) \subset  \bigcup_{\substack{s \in \cS_l \\ \mu(B_{r_l}(s)) \leq M r_l^k \\ l \leq i }} B_{r_l}(s) \cup \bigcup_{l \leq i-1} E_l
\end{gather}
Combining the above with Lemma \ref{lem:cutting-holes} and Lemma \ref{lem:excess-bound} gives \eqref{eq_measure_est}.  This completes the proof of Theorem \ref{thm:good-tree}.  Corollary \ref{cor:good-tree} is simply a restatement of Theorem \ref{thm:good-tree}.

\vspace{1cm}

\section{Bad tree}\label{section:bad-tree}

We define precisely the bad tree construction at $B_1(0)$.  As explained in the proof outline of Section \ref{ss:proof_outline}, the bad tree is essentially similar to the good tree construction, except with good and bad balls swapped, and best $(k-1)$-planes instead of best $k$-planes.  Since we are interested in $k$-control this simplifies many of the technical details.  The resulting bad tree decomposes $B_1(0)$ into a part of controlled measure, a family of good balls with packing estimates, and a subset with $\haus^k$-measure $0$.

Suppose $B_1(0)$ is a bad ball for a non-negative Borel-regular measure $\mu$, with parameters $\rho$ and $m = m_0(n,\rho)M$ as described in Section \ref{section:construction}.  We take $\cC$ a covering pair for $\mu$, with the property that
\begin{gather}\label{eqn:covering-bad-hyp}
B_2(0) \cap B_{r_y}(y) \neq\emptyset \implies r_y \leq 1\quad \forall y \in \cC_+.
\end{gather}
As before, in principle this is the assumption $r_y \leq 1$, but for technical reasons it is preferable to state it this way.

We define inductively families of bad balls $\cur{B_{r_i}(b)}_{b\in \cB_i}$, good balls $\cur{B_{r_i}(g)}_{g\in \cG_i}$, and stop balls $\cur{B_{r_i}(s)}_{s\in \cS_i}$.  At our initial scale $r_0 = 1$, $B_1(0)$ is the only ball, and it is bad by assumption.  So $\cB_0 = \{0\}$, and $\cG_0 = \cS_0 = \emptyset$.

In each bad ball $B_{r_i}(b)$, we have an associated $(k-1)$-plane $W^{k-1}_{ib}$ from Lemma \ref{lem:ball-dichotomy}.  Just for clarity, $\mu$ may be quite badly approximated by $W^{k-1}_{ib}$ in this context, but this will not be relevant since we have a dimension drop.  We define the bad ball excess set to be 
\begin{gather}
E_{ib} = B_{r_i}(b) \setminus B_{2r_{i+1}}(W_{ib}),
\end{gather}
and the total excess
\begin{gather}
E_i = \bigcup_{b \in \cB_i} E_{ib}.
\end{gather}
The excess sets will have controlled measure.

Define inductively the remainder set to be the good and stop balls at all previous scales:
\begin{gather}
R_i = \bigcup_{\ell=0}^i \ton{ B_{r_\ell}(\cS_\ell) \cup B_{r_\ell}(\cG_\ell) }
\end{gather}

Suppose we have defined the good/bad/stop balls down through scale $r_{i-1}$.  We define the $r_i$-balls as follows.  Let $J_i$ form a maximal $2r_i/5$-net in
\begin{gather}\label{eqn:bad-cover-step}
B_1 \cap \left[ \bigcup_{b \in \cB_{i-1}} (B_{r_{i-1}}(b) \cap B_{5r_i}(W_{i-1,b}) ) \right] \setminus R_{i-1}.
\end{gather}
so the balls $\{B_{r_i}(z)\}_{z \in J_i}$ cover \eqref{eqn:bad-cover-step}, and the balls $\{B_{r_i/5}(z)\}_{z \in J_i}$ are disjoint.

Now define
\begin{align}
\cS_i &= \{z \in J_i : \text{$B_{r_i}(z)$ is a stop ball} \}, \\
\cG_i &= \{z \in J_i : \text{$B_{r_i}(z)$ is not stop, but is good} \}, \\
\cB_i &= \{z \in J_i : \text{$B_{r_i}(z)$ is not stop, but is bad}\}.
\end{align}
Evidently $J_i = \cS_i \cup \cG_i \cup \cB_i$.  This completes the bad tree construction.

\begin{definition}\label{def:bad-tree}
The construction defined above is said to be the \emph{bad tree rooted at $B_1(0)$}, and may be written as $\cT = \cT(B_1)$.  Given such a tree $\cT$ we refer to the collection of all good/bad balls by
\begin{gather}
\cG(\cT) := \cup_i \cG_i, \quad \cB(\cT) := \cup_i \cB_i.
\end{gather}
As before we define the associated radius functions $r_g$, $r_b$ for $\cG(\cT)$ and $\cB(\cT)$.

For every stop ball $B_{r_i}(s)$ we have either $\mu \ton{B_{r_i}(s)} \leq M r_i^k$, or $s \in B_{r_y + 2r_i}(y_s)$ for some choice of $y_s \in \cC_+$ with $r_i < r_{y_s} \leq r_{i-1}$.  Define
\begin{gather}
\cC_+(\cT) = \bigcup_{i=0}^\infty \bigcup_{\substack{s \in \cS_i \\ \mu \ton{B_{r_i}(s)} > M r_i^k }} y_s
\end{gather}
to be a choice of $y$'s arising in this way (so, at most one $y$ per stop ball).  Define also
\begin{gather}
\cC_0(\cT) = \cap_i \overline{B_{r_i}(\cB_i)}\, .
\end{gather}
In a bad tree, we refer to the collection of all good balls as the \emph{tree leaves}.
\end{definition}

\begin{remark}\label{rem:bad-radius-control}
As in the good ball construction, we have
\begin{gather}\label{eqn:scale-control-bad}
y \in \cC \text{ and } B_{2r_i}(\cG_i \cup \cB_i) \cap B_{r_y}(y) \neq\emptyset \implies r_y \leq r_i .
\end{gather}

In particular, Corollary \ref{lem:good-containment} and Remark \ref{rem:good-containment} are still valid.
\end{remark}

\begin{remark}\label{rem:bad-mass}
Remark \ref{rem:good-mass} continues to hold for bad trees. In other words, we have the inclusion
\begin{gather}
\bigcup_{i=0}^\infty \bigcup_{\substack{s \in \cS_i \\ \mu \ton{B_{r_i}(s)} > M r_i^k}} B_{r_i}(s) \subset \bigcup_{y \in \cC_+(\cT)} B_{4r_y}(y).
\end{gather}
And therefore we also have the lower bound
\begin{gather}
\mu \ton{B_{4r_y}(y)} > \frac{M}{\rho^k} r_y^k \quad \forall y \in \cC_+(\cT).
\end{gather}
\end{remark}

The following is our main theorem in this section:\\
\begin{theorem}[Bad Tree Structure]\label{thm:bad-tree}
Provided $c_1(n)\rho \leq 1/2$, where $c_1(n)$ as in Theorem \ref{thm:packing-control}, then for any $i \geq 0$, we have the packing estimate
\begin{gather}\label{eqn:bad-packing}
\sum_{\ell = 1}^i \#\cur{\cG_\ell \cup \cB_\ell \cup \cS_\ell} r_\ell^k \leq 2 c_1(n) \rho,
\end{gather}
and the measure estimate
\begin{gather}\label{eqn:bad-total-measure}
\mu \left[ B_1 \setminus \left( \bigcup_{b\in \cB_i } B_{r_i}(b) \cup \bigcup_{\substack{g\in \cG_\ell \\ \ell\leq i}} B_{r_\ell}(g) \cup \bigcup_{\substack{x \in \cC_+(\cT) \\ r_x \geq r_i}} B_{4r_x}(x) \right) \right] \leq 3M.
\end{gather}
Recall that we have chosen $m = m_0(n,\rho)M$ as in Theorem \ref{thm:packing-control}.
\end{theorem}
\begin{remark}
Note that in comparison to the Good Tree Structure theorem we have both smallness on the estimates, and in \eqref{eqn:bad-packing} we have that the sum of all balls on all scales have a small packing estimate, not just the balls in the final covering.  This is due to our ability to approximate the support by a subspace of dimension strictly less than $k$.
\end{remark}

In the language of Definition \ref{def:bad-tree}, Theorem \ref{thm:bad-tree} says the following.
\begin{corollary}\label{cor:bad-tree}
Suppose the hypotheses of Theorem \ref{thm:bad-tree}.  Then we have
\begin{itemize}
\item[A)] Tree-leaf/good-ball packing, with \emph{small} constant:
\begin{gather}
\sum_{g \in \cG(\cT)} r_g^k \leq 2 c_1 \rho .
\end{gather}

\item[B)] Original ball packing:
\begin{gather}
\sum_{x \in \cC_+(\cT)} r_x^k \leq c(n, \rho), \quad \text{and} \quad r^{k-n}|B_r(\cC(\cT))| \leq c(n, \rho) \quad \forall 0 < r \leq 1.
\end{gather}

\item[C)] Upper measure estimates:
\begin{gather}
\mu \left[ B_1 \setminus \left( \bigcup_{g \in \cG(\cT)} B_{r_g}(g) \cup B_{4r_x}(\cC(\cT)) \right) \right] \leq 3M, 
\end{gather}

\item[D)] Lower measure estimates:
\begin{gather}
\mu \ton{B_r(y)} > \frac{M}{c(n,\rho)} r^k \quad \forall y \in \cC(\cT), \,\, \forall 4r_y < r \leq 1.
\end{gather}

\item[E)] Fine-scale packing structure: $\cC_0(\cT)$ is closed, with $\haus^k(\cC_0(\cT)) = 0$.
\end{itemize}
\end{corollary}
\vspace{.25cm}

\begin{proof}[Proof of Theorem \ref{thm:bad-tree} and Corollary \ref{cor:bad-tree}]
By Theorem \ref{thm:packing-control} applied to each scale $r_\ell$, $\ell = 0, \ldots, i-1$, we have
\begin{gather}\label{eqn:single-scale-bad-packing}
\sum_{\cG_i \cup \cB_i \cup \cS_i} r_i^k \leq c_1(n) \rho \sum_{\cB_{i-1}} r_{i-1}^k \leq (c_1\rho)^2 \sum_{\cB_{i-2}} r_{i-2}^k \leq \ldots \leq (c_1 \rho)^i, 
\end{gather}
and therefore we obtain the estimate
\begin{gather}
\sum_{1 \leq \ell \leq i} \#\cur{\cG_\ell \cup \cB_\ell \cup \cS_\ell} r_\ell^k  \leq \sum_{\ell=1}^i (c_1(n) \rho)^\ell \leq 2c_1(n) \rho.
\end{gather}
The last inequality follows since $c_1 \rho \leq 1/2$.  This proves the first packing estimate of \eqref{eqn:bad-packing}

Equation \eqref{eqn:bad-total-measure} follows from the packing estimate \eqref{eqn:bad-packing}, and our choice of $m$: in any bad ball $B_{r_i}(b)$ we have 
\begin{gather}
\mu (E_{i b}) \leq M r_i^k \, ,
\end{gather}
and by construction we have for any $i$:
\begin{gather}
B_1 \subset B_{r_i}(\cB_i) \cup E_0 \cup \bigcup_{\ell = 1}^i \left( B_{r_\ell}(\cG_\ell \cup \cS_\ell) \cup E_\ell \right) .
\end{gather}

Corollary \ref{cor:bad-tree} is essentially a restatement of Theorem \ref{thm:bad-tree}.  We emphasize that the reason we have small packing is because there are no good balls at scale $0$.  The estimate $\haus^k(\cC_0(\cT)) = 0$ follows from \eqref{eqn:bad-packing} (or \eqref{eqn:single-scale-bad-packing}).
\end{proof}

\vspace{.5cm}

\section{Finishing the proof}\label{section:finishing}

We wish to obtain the decomposition of Theorem \ref{thm:core-estimate}, which consists of packing estimates inside the original balls ($\cC_+$) or sets ($\cC_0$), and measure estimates away from these.  In any single good or bad tree, this decomposition is achieved away from its leaves, which unfortunately are unavoidable.  As described in the proof outline in Section \ref{ss:proof_outline}, we implement the following construction: for argument's sake suppose $B_1$ is good, so we build a good tree downwards from $B_1$; at any (necessarily bad) leaf of this tree, we build a bad tree; at any (necessarily good) leaf in any of these bad trees, we build a good tree; repeat, etc...

In this way we can continue refining our decomposition, progressing down in scale either infinitely far, thus hitting some part of $\cC_0$, or until we have no leaves in any tree.  By Theorems \ref{thm:good-tree} and \ref{thm:bad-tree} applied to each tree, it is essentially enough to show we have packing on \emph{all} leaves among \emph{all} trees, which is the content of Theorem \ref{thm:tree-packing}.

We make the above idea rigorous in the following construction.  Let us fix $\rho(n) < 1/20$ so that
\begin{gather}\label{eq_rho}
 2 c_1(n) \rho 50^k \leq 1/2\, , 
\end{gather}
where $c_1$ is as in Theorem \ref{thm:packing-control}. Let us also fix
\begin{gather}
 m = m_0(n,\rho) M
\end{gather}
as in Theorem \ref{thm:packing-control}, and then $\delta_0 = \delta_1(n, \rho) 16^{-k-2}$ as in Theorem \ref{thm:good-tree}.

Let us inductively define for each $t = 0, 1, 2, \ldots$ a family of balls $\{B_{r_f}(f)\}_{f \in \cF_t}$, where $r_f$ is a radius function which varies with $f \in \cF_t$.  We will see in practice that each ball $B_{r_f}(f)$ is a tree leaf arising from some tree construction at a previous stage in the induction.  For each $t$, these leaves $\cF_t$ will be either all good balls or all bad balls, but none of them will be stop balls.  Moreover, for $t \geq 1$ each leaf of $\cF_t$ will be a leaf in some good or bad tree.  Let us caution the reader that stage index $t$ is only an \emph{upper bound} on the scale of the leaves in $\cF_t$, e.g. if $f\in \cF_t$ then $r_f\leq r_t$.

At stage $t = 0$, $B_1(0)$ is our only leaf.  So $\cF_0 = \{0\}$, and $r_{f = 0} = 1$.  Trivially all the leaves of $\cF_0$ are good or bad balls, since we can assume $B_1(0)$ is not a stop ball.

Suppose we have defined leaves down to stage $t-1$.  The leaves $\cF_{t-1}$ are (by inductive hypothesis) not stop, and either all good or all bad.  Let us assume first they are all good.  Then for each $f \in \cF_{t-1}$, build a good tree $\cT_{G, f} := \cT(B_{r_f}(f))$, with parameters $\rho$ and $m$ fixed as above.  Then we set $\cF_t$ to be the collection of bad leaves of these good trees:
\begin{gather}
\cF_t = \bigcup_{f \in \cF_{t-1}} \cB(\cT_{G, f}) ,
\end{gather}
and we let the radius function $r_f$ be the associated bad-ball radius function in the trees:
\begin{gather}
r_f|_{\cF_t} = r_b.
\end{gather}

Similarly, let us assume all the leaves in $\cF_{t-1}$ are bad.  Then for each $f \in \cF_{t-1}$, we build a bad tree $\cT_{B, f} := \cT(B_{r_f}(f))$, with parameters $\rho$, $m$, and define
\begin{gather}
\cF_t = \bigcup_{f \in \cF_{t-1}} \cG(\cT_{B, f}), \quad r_f|_{\cF_t} = r_g.
\end{gather}

In either case, we have that the $\cF_t$ are either all good or all bad (but not stop), and when $t \geq 1$ they arise as leaves in some good or bad tree.

We just need to check the conditions of the good/bad tree constructions are satisfied: the conditions \eqref{eqn:covering-good-hyp}, \eqref{eqn:covering-bad-hyp} on the covering pair $\cC$ are satisfied either by assumption (when $t = 1$, we are building a tree at $B_1$), or by Corollary \ref{rem:good-containment} and Remark \ref{rem:bad-radius-control}, since when $t \geq 1$ every leaf sits on some tree.  We are therefore justified in building each good/bad tree.

This completes the tree chaining construction.  Our key Theorem is the following leaf-packing estimate:
\begin{theorem}[Leaf Packing]\label{thm:tree-packing}
We have the packing bound
\begin{gather}\label{eqn:tree-packing}
\sum_{t = 0}^\infty \sum_{f \in \cF_t} r_f^k \leq c(k).
\end{gather}

And if $f \in \cF_t$ then $r_f \leq \rho^t$.
\end{theorem}

\begin{proof}
Suppose, for some fixed $t \geq 1$, all the $\cF_t$ are good.  Then each $f \in \cF_t$ is a leaf of a bad tree rooted at $\cF_{t-1}$.  Therefore, by our choice of $m$, we can apply Theorem \ref{thm:bad-tree} (or more precisely, Corollary \ref{cor:bad-tree} part A) to each bad tree in $\cF_{t-1}$ to obtain
\begin{gather}\label{eqn:good-leaf-packing}
\sum_{f \in \cF_t} r_f^k \leq 2c_1\rho \sum_{f \in \cF_{t-1}} r_f^k.
\end{gather}

Suppose now all the $\cF_t$ are bad.  Then each leaf in $\cF_t$ lies in some good tree rooted at $\cF_{t-1}$.  By our choice of $\delta_0$, and because $\mu$ is supported in $B_1$, we can apply Theorem \ref{thm:good-tree} (more precisely, Corollary \ref{cor:good-tree} part A) to obtain
\begin{gather}\label{eqn:bad-leaf-packing}
\sum_{f \in \cF_t} r_f^k \leq 50^k \sum_{f \in \cF_{t-1}} r_f^k.
\end{gather}

Since by construction the leaves change type from one stage to the next, we can alternatively iterate estimates \eqref{eqn:good-leaf-packing} and \eqref{eqn:bad-leaf-packing} to obtain a packing bound on the leaves at a given stage:
\begin{align}
\sum_{f \in \cF_t} r_f^k &\leq (2c_1 \rho 50^k) \sum_{f \in \cF_{t-2}} r_f^k \leq (2c_1 \rho 50^k)^2 \sum_{f \in \cF_{t-4}} r_f^k\leq \cdots \notag\\
&\leq (2c_1 \rho 50^k)^{t/2}\sum_{f \in \cF_{0}} r_f^k \leq c(k) (2c_1 \rho 50^k)^{t/2} \leq c(k) 2^{-t/2}.
\end{align}
where in the last inequality we used our choice of $\rho$.  The packing estimate \eqref{eqn:tree-packing} is now immediate.

The last assertion is a basic consequence of our tree constructions.  In any tree rooted at $B_r(x)$, any leaf has radius $\leq \rho r$.  Therefore
\begin{gather}
\max_{f \in \cF_t} r_f \leq \rho \max_{f \in \cF_{t-1}} r_f \leq \rho^t.
\end{gather}
\end{proof}

\subsection{Measure and packing decomposition}

We now define the subset $\cC' \subset \cC$, and demonstrate $\cC'$ satisfies the required packing and measure bounds of Theorem \ref{thm:core-estimate}.  Unfortunately, we cannot use the $\cC(\cT)$s by themselves to build our decomposition $\cC'$.  The problem is the following:

In any single tree $\cT$, we have measure control only away from $\cC(\cT)$ and the leaves.  If in some region $A \subset B_1^n(0)$ we must alternate the good-/bad-tree construction infinitely-many times, then in \emph{every} tree this region $A$ lies inside the leaves.  In other words, no single tree will see $A$ in its measure estimate, and we expect strict inclusion
\begin{gather}
\overline{\cup_{\cT} \cC_0(\cT)} \subsetneq \cC'_0.
\end{gather}
To capture regions like $A$, we must consider global collections of balls.  (Of course, any $A$ like this will have $\haus^k$-measure $0$, but could still have very large $\mu$-measure.)

To handle this issue, we consider the collection $\cQ_i$ of \emph{all} good and bad $r_i$-balls, contained in \emph{any} tree.  This ``horizontal slice'' of the tree-structures will see all the regions with big mass, and makes the global Minkowski bounds of $\cC'$ easier to handle.  Proving the correct measure and packing bounds for $\cQ_i$ is straightforward consequence of the good-/bad-tree structure Theorems \ref{thm:good-tree}, \ref{thm:bad-tree}, and the leaf packing Theorem \ref{thm:tree-packing}.

Precisely, $\cQ_i$ is defined as follows:
\begin{align}
\cQ_i 
&= \left\{ g \in \cG(\cT(B_{r_f}(f))) : f \in \bigcup_{t=0}^i \cF_t, \text{ and } r_g = r_i \right\} \\
&\quad \cup \left\{ b \in \cB(\cT(B_{r_f}(f))) : f \in \bigcup_{t=0}^i \cF_t , \text{ and } r_b = r_i \right\}.
\end{align}
From our tree constructions, it is clear that for each fixed $A$, we have
\begin{gather}
\cQ_A \subset B_{r_i}(\cQ_i) \quad \forall i \leq A.
\end{gather}

We now define the subcollection $\cC'$.  We set $\cC_+'$ to be all big-mass stop balls among all trees (recall Definitions \ref{def:good-tree} and \ref{def:bad-tree}):
\begin{gather}
\cC'_+ = \bigcup_{t = 0}^\infty \bigcup_{f \in \cF_t} \cC_+(\cT(B_{r_f}(f))).
\end{gather}
We let $\cC'_0$ be the region contained in every $r_i$-collection of good/bad balls:
\begin{gather}
\cC'_0 = \bigcap_{i=0}^\infty \overline{B_{r_i}(\cQ_i)}.
\end{gather}

Packing in individual trees, combined with the leaf packing \eqref{eqn:tree-packing}, gives us the following
\begin{lemma}\label{lem:all-good-packing}
For each $i$ we have the packing estimate
\begin{gather}\label{eqn:Q-packing}
\#\{\cQ_i\} r_i^k \leq c(k),
\end{gather}
and measure estimate
\begin{gather}
\mu \left[ B_1 \setminus \left( B_{4r_x}(\cC'_+) \cup B_{r_i}(\cQ_i) \right) \right] \leq c(n) M
\end{gather}
\end{lemma}

\begin{proof}
We first prove the packing estimate.  Each good/bad $r_i$-ball lives in some good/bad tree, rooted in $\cF_t$ with $t \leq i$.  We can therefore apply good-/bad-tree packing estimates \eqref{eq_pack_est} and \eqref{eqn:bad-packing} to each such tree, and then use leaf-packing \eqref{eqn:tree-packing} to deduce
\begin{gather}
\#\{\cQ_i\} r_i^k \leq c(k) \sum_{t=0}^i \sum_{f \in \cF_t} r_f^k \leq c(k).
\end{gather}

We prove the measure estimate.  By construction, the leaves $\cF_{t+1}$ include the leaves of any tree rooted at $\cF_t$.  Therefore, given any leaf $f \in \cF_t$, we can apply the good/bad-tree measure estimates \eqref{eq_measure_est} and \eqref{eqn:bad-total-measure} to deduce
\begin{gather}
\mu \left[ B_{r_f}(f) \setminus \left( B_{4r_x}(\cC'_+) \cup B_{r_i}(\cQ_i) \cup \bigcup_{f' \in \cF_{t+1}} B_{r_{f'}}(f') \right) \right] \leq c(n)(M + \delta_0^2M) r_f^k.
\end{gather}

Since $\cQ_i$ contains every leaf of scale $r_i$, and $r_f \leq r_i$ when $f \in \cF_i$, we deduce
\begin{align}
&\mu \left[ B_1 \setminus \left( B_{4r_x}(\cC') \cup B_{r_i}(\cQ_i) \right)  \right] \\
&\leq \sum_{t=0}^{i-1} \sum_{f \in \cF_t} \mu \left[ B_{r_f}(f) \setminus \left( B_{4r_x}(\cC'_+) \cup B_{r_i}(\cQ_i) \cup \bigcup_{f' \in \cF_{t+1}} B_{r_{f'}}(f') \right) \right] \\
&\leq c(n) (M + \delta_0^2 M) \sum_{t=0}^{i-1} \sum_{f \in \cF_t} r_f^k \\
&\leq c(n) M.
\end{align}
\end{proof}

Taking $i\to\infty$, from our definition of $\cC'$ we have that Lemma \ref{lem:all-good-packing} implies Theorem \ref{thm:core-estimate} part B).  We now show $\cC'$ admits the correct closure, packing, and non-collapsing properties.  We require first a Lemma on the comparability between balls in $\cC'_+$, and balls in $\cQ_i$.
\begin{lemma}\label{lem:good-bad-containment}
For any $y \in \cC'_+$, with $r_y \leq r_{i-1}$, we have
\begin{gather}\label{eqn:inside-big-good-bad}
B_{r_y}(y) \cap B_{2r_{i-1}}(\cQ_{i-1}) \neq \emptyset.
\end{gather}

Conversely, if $y \in \cC_+$ and $r_y > r_i$, then
\begin{gather}\label{eqn:avoids-small-good-bad}
B_{r_y}(y) \cap B_{2r_i}(\cQ_i) = \emptyset.
\end{gather}

For any $y \in \cC'$, we have
\begin{gather}\label{eqn:lem-orig-disjoint}
B_{100 r_y}(y) \cap B_{r_x}(x) = \emptyset \quad \forall x \in \cC \text{ such that } r_x > c(n) r_y.
\end{gather}
\end{lemma}

\begin{proof}
Given $y \in \cC'_+$, then $y \in \cC_+(\cT)$ for some good or bad tree.  By construction, there is a stop ball $B_{r_i}(s)$ in this tree with $r_i < r_y \leq r_{i-1}$ and $s \in B_{r_y + 2r_i}(y)$.  But then $s \in B_{r_{i-1}}(q)$ for some good or bad ball $B_{r_{i-1}}(q)$, since trees are never rooted in stop balls.  We deduce
\begin{gather}
|y - q| \leq |y - s| + |s - q| < (r_y + 2r_i) + r_{i-1} < r_y + 2r_{i-1}.
\end{gather}

Conversely, if $y \in \cC_+$, then using ``radius control'' in any tree (see inductive hypothesis \eqref{eqn:scale-control} and Remark \ref{rem:bad-radius-control}), we deduce
\begin{gather}
B_{r_y}(y) \cap B_{2r_i}(\cQ_i) \neq\emptyset \implies r_y \leq r_i.
\end{gather}

The last statement is a direct consequence of the first two, recalling that $\rho = \rho(n) < 1/10$.
\end{proof}

\begin{lemma}\label{lem:original-packing}
We have that $\cC'_0 \subset \cC_0$, $\cC'_+ \subset \cC_+$, both $\cC'_0$ and $\cC'$ are closed, and $\cC'$ satisfies the packing estimate
\begin{gather}\label{eqn:lem-orig-packing}
\sum_{x \in \cC'_+} r_x^k \leq c(n), \quad\text{and}\quad |B_r(\cC')| \leq c(n) r^{n-k} \quad \forall 0 < r \leq 1.
\end{gather}
\end{lemma}

\begin{proof}
Let us first prove the packing estimates \eqref{eqn:lem-orig-packing}.  Applying Corollaries \ref{cor:good-tree}, \ref{cor:bad-tree} parts B) to each good/bad tree, and recalling that $\rho = \rho(n)$, we have
\begin{align}
\sum_{x \in \cC'_+} r_x^k \leq c(n) \sum_{t=0}^\infty \sum_{f \in \cF_t} r_f^k \leq c(n).
\end{align}

Take any $r \in (0, 1]$, and choose $i$ so that $r_i < r \leq r_{i-1}$.  From Lemma \ref{lem:good-bad-containment}, we know
\begin{gather}
\bigcup_{y \in \cC'_+ : r_y \leq r_i} B_r(y) \subset B_{4r_{i-1}}(\cQ_{i-1}).
\end{gather}
Therefore
\begin{align}
r^{k-n}|B_r(\cC')|
&\leq r^{k-n} |B_r(\cC'_+)| + r^{k-n}|B_r(\cC'_0)| \\
&\leq \omega_n \sum_{y \in \cC'_+ : r_y > r_i} r^k + 2 r^{k-n} |B_{4r_{i-1}}(\cQ_{i-1})| \\
&\leq \omega_n \rho^{-k} \sum_{y \in \cC'_+ : r_y > r_i} r_y^k + 2 r_i^{k-n} \#\{\cQ_{i-1}\}\omega_n (4r_{i-1})^n \\
&\leq \omega_n \rho^{-k} c(n) + \omega_n (4/\rho)^n c(n).
\end{align}

We prove the other statements.  By Lemma \ref{lem:good-bad-containment} we have that $\cC'_0 \subset \spt\mu \setminus B_{r_x}(\cC_+)$, and therefore $\cC'_0 \subset \cC_0$.  The inclusion $\cC'_+ \subset \cC_+$ is by construction.

$\cC'_0$ is closed by construction.  To prove $\cC'$ is closed, it will suffice to show that any limit point of $\cC'_+$ must lie in $\cC'_0$.  Let $\{y_i\}_i \subset \cC'_+$ be a sequence converging to some $y$.  By packing estimate \eqref{eqn:lem-orig-packing}, necessarily $r_{y_i} \to 0$.  After relabeling we can assume $r_{y_i} < r_i$.

We deduce that, for any fixed $A$, we have
\begin{gather}
y_i \in B_{5r_{i-1}}(\cQ_{i-1}) \subset B_{r_A + o(1)}(\cQ_A) \quad \text{as } i \to \infty.
\end{gather}
This proves that $y \in \cC'_0$.
\end{proof}
\vspace{.2cm}

\begin{lemma}\label{lem:orig-lower-mass}
We have the lower bound
\begin{gather}\label{eqn:lem-orig-mass}
\mu \ton{B_r(y)} > \frac{M}{c(n)} r^k \quad \forall y \in \cC', \,\, \forall 4r_y \leq r \leq 1.
\end{gather}
\end{lemma}

\begin{proof}
Take $y \in \cC'_+$, and choose $A$ so that $r_{A+1} \leq r_y < r_A$.  By Lemma \ref{lem:good-bad-containment} we have for every $i \leq A-1$ the existence of a $q \in \cQ_i$ so that
\begin{gather}
B_{4r_i}(y) \supset B_{r_i}(q).
\end{gather}
Using the definition of a good/bad ball we therefore have
\begin{gather}
\mu \ton{B_{4r_i}(y)} > M r_i^k \quad \forall i \leq A-1.
\end{gather}
Since $\mu \ton{B_{4r_y}(y)} > \frac{M}{\rho^k} r_y^k$ (Remarks \ref{rem:good-mass}, \ref{rem:bad-mass}), this proves \eqref{eqn:lem-orig-mass} with $c = c(n,\rho) = c(n)$.

Given $y \in \cC'_0$, then by construction the above reasoning works for any $i$.
\end{proof}

\subsection{Structure of \texorpdfstring{$\cC'_0$}{C0}} \label{subsection:structure-of-C_0}

We have established Theorem \ref{thm:core-estimate} parts A)-C), and that $\cC'_0$ is closed.  To show upper-Ahlfors-regularity and rectifiability of $\cC'_0$ (conclusion D) we demonstrate that the measure $\eta = \haus^k \llcorner \cC'_0$ satisfies
\begin{equation}\label{eqn:set-D-by-mu}
\sum_{\alpha \in \dZ \,\, : \,\, 2^\alpha \leq 2} \beta^k_{\eta, 2, 0}(z, 2^\alpha)^2 \leq c(n) \delta_0^2\quad \forall z \in B_1 .
\end{equation}
We then can apply our own Theorem \ref{thm:core-estimate} parts A), B) to the measure $\eta$ to deduce upper-Ahlfors-regularity, and the Proposition \ref{prop:rect_with_bounds}, which will obviously be proven independently, to deduce rectifiability.

To this end, we first apply Lemma \ref{lem:density-to-ineq} with $A_1 = \cC'_0\cap U$, $A_2 = U$ to deduce
\begin{gather}
\eta(U) \leq \frac{c(k)}{M} \mu(U) \quad \forall \text{ open } U.
\end{gather}
Now by Fubini we deduce that for any $B_r(x)$, 
\begin{align}
\beta_{\eta, 2, 0}^k(x, r)^2 
&\leq r^{-k-2} \int_{B_{r}(x)} d(z, V^k_\mu(x, 2r))^2 d\eta \\
&= r^{-k-2} \int_0^r \eta\{ z \in B_r(x) : d(z, V_\mu(x, 2r))^2 > s\} ds \\
&\leq \frac{c(k)}{M} r^{-k-2} \int_0^r \mu\{z \in B_r(x) : d(z, V_\mu(x, 2r))^2 > s\} ds \\
&\leq \frac{c(n)}{M} \beta^k_{\mu, 2, \bar\eps_\beta}(x, 2r)^2\, ,
\end{align}
provided $\mu(B_{2r}(x)) > \bar\eps_\beta (2r)^k$.  But using C), if $\eta(B_{r}(x)) > 0$, then we must have some $z \in \cC'_0 \cap B_{r}(x)$, and therefore
\begin{gather}
\mu(B_{2r}(x)) \geq \mu(B_{r}(z)) \geq \frac{M}{c_{key}} \omega_k r^k \geq \bar\eps_\beta (2r)^k.
\end{gather}
We can sum over $\alpha$ to find
\begin{gather}
\sum_{\alpha \in \dZ \,\, : \,\, 2^\alpha \leq 2} \beta^k_{\eta, 2, 0}(x, 2^\alpha)^2 \leq \frac{c(n)}{M} \sum_{\alpha \in \dZ \,\, : \,\, 2^\alpha \leq 4} \beta^2_{\mu, 2, \bar\eps_\beta}(x, 2^\alpha)^2 \leq \frac{c(n)}{M} c(k) \delta^2_0 M.
\end{gather}
In the last inequality we used our assumption \eqref{eq_sum_beta_bounds}, and that $\mu$ is supported in $B_1$.  This proves the required estimate \eqref{eqn:set-D-by-mu}.

We use \eqref{eq_sum_beta_bounds} to prove upper-Ahlfors-regularity and rectifiability.  First, observe that
\begin{gather}\label{eqn:C_0-density}
2^{-k} \leq \Theta^{*, k}(\eta, x) \leq 1 \quad \text{for $\eta$-a.e. $x$},
\end{gather}
since $\haus^k(\cC'_0) < \infty$.

To prove upper-Ahlfors-regularity we essentially reprove Corollary \ref{cor:main_upperahlfors_1} part U, using Theorem \ref{thm:core-estimate} in place of Theorem \ref{thm:main}.  Fix any $B_r(x)$.  Define the covering pair $(\cS, s_z)$ by setting
\begin{gather}
\cS_+ = \{ z \in B_r(x) \cap \cC_0' : \Theta^{*, k}(\eta, z) \leq 1 \} , \quad \cS_0 = B_r(x) \setminus \cS_+.
\end{gather}
By \eqref{eqn:C_0-density} $\eta(\cS_0) = 0$, and for each $z \in \cS_+$, we can choose an $s_z < r$ for which $\eta(B_{4s_z}(z)) \leq 2\omega_k (4s_z)^k$.

Apply Theorem \ref{thm:core-estimate} parts A), B) to $\eta \llcorner B_r(x)$ at scale $B_r(x)$, and the covering pair $(\cS, s_z)$.  Then there is a subset $\cS'$ so that
\begin{align}
\eta(B_r(x)) \leq c(n) r^k + \sum_{z \in \cS'_+} \eta(B_{4r_z}(z)) \leq c(n) r^k + c(k) \sum_{z \in \cS'_+} s_z^k \leq c(n) r^k.
\end{align}

We deduce
\begin{gather}\label{eqn:C_0-upper}
\eta(B_r(x)) \equiv \haus^k(\cC'_0 \cap B_r(x)) \leq c(n) r^k \quad \forall x\in B_1, 0 < r < 1.
\end{gather}

Finally, we apply Proposition \ref{prop:rect_with_bounds}, which will be proved independently, to deduce $\cC'_0$ is rectifiable.  In fact Lemma \ref{lemma_rect_manifold} shows $\cC'_0$ is a kind of ``uniformly rectifiable,'' in the sense that whenever $\haus^k(\cC'_0 \cap B_r(x)) > 2^{-k-1}\omega_k r^k$, there is some bi-Lipschitz $k$-manifold $T$ so that
\begin{gather}
\haus^k(\cC'_0 \cap T) \geq \frac{1}{2} \haus^k(\cC'_0 \cap B_r(x)).
\end{gather}

This finishes the proof of Theorem \ref{thm:core-estimate}.

\vspace{1cm}

\section{Packing/Measure estimates and density}\label{sec:density}

In general we cannot do better than a combination of measure and packing estimate.  At the one extreme, $\mu$ could have infinite mass concentrated across a $k$-plane.  At the other, $\mu$ may have no $k$-dimensional structure whatsoever, as when $\mu = \leb^n$.

However in the presence of density bounds, we can deduce purely Hausdorff or purely measure estimates.  Lower bounds on upper density give global Hausdorff bounds, and upper bounds on lower density give global measure bounds.  Density enforces some kind of $k$-dimensional structure locally.

We focus this section on proving Theorems \ref{thm:main_density}, \ref{thm:main_upperahlfors}, and their Corollaries.  Theorem \ref{thm:main_density} is a more-or-less direct consequence of Theorem \ref{thm:main}.  However, Theorem \ref{thm:main_upperahlfors} is more involved.  As the proof is vastly simpler, and more instructive, let us give first the special case of Corollary \ref{cor:main_density_1}.

\begin{proof}[Proof of Corollary \ref{cor:main_density_1}]
Part L).  Apply Theorem \ref{thm:main} with $\eps = \bar\eps_\beta/c(k)$ to find a closed set $K_0$ so that
\begin{gather}
\haus^k(K_0) \leq c(n), \quad \mu(B_1 \setminus K_0) \leq c(n)(\eps + M) + \Gamma.
\end{gather}

Define
\begin{gather}
K_1 = \{z \in B_1 \setminus K_0 : \Theta^{*, k}(\mu, x) \geq a \} ,
\end{gather}
and then use Lemma \ref{lem:density-to-ineq} with $A_1 = K_1$, $A_2 = B_1 \setminus K_0$ to deduce
\begin{gather}
a \haus^k(K_1) \leq \mu(B_1 \setminus K_0) \leq c(n)(\eps + M) + \Gamma.
\end{gather}
Now set $\cC'_0 = K_0 \cup K_1$.

Part U).  Let $\cS$ be the covering pair defined by
\begin{gather}
\cS_+ = \{z \in B_1 : \Theta_*^k(\mu, z) \leq b \}, \quad \cS_0 = \cC \setminus \cS_+.
\end{gather}
For $x \in \cS_+$, choose $s_x \leq 1$ so that $\mu(B_{s_x}(x)) \leq 2\omega_k b s_x^k$.  By construction, $\mu(\cS_0) = 0$, and $\cS$ is closed.

Apply Theorem \ref{thm:main} with $\eps = \bar\eps_\beta/c(k)$ to the covering pair $(\cS, s_x)$, to obtain a $\cS'$.  We then have
\begin{align}
\mu(B_1) 
&\leq c(n)(M + \eps) + \Gamma + \mu(B_{s_x}(\cS'_+)) \\
&\leq c(n)(M + \eps) + \Gamma + c(k) b \sum_{x \in \cS'_+} s_x^k \\
&\leq c(n)(M + \eps + b) + \Gamma.
\end{align}
\end{proof}


There are a couple complications to proving Theorem \ref{thm:main_upperahlfors} in general.  First, in any given ball $B_r(x)$, condition \eqref{eqn:upperahlfors-hyp} only gives information on points $y$ with $r_y \leq 2r$.  We need to find a good ``minimal'' subset $\cU \subset \cC$, so that in \emph{any} ball centered in $\cU$ we have $\beta$-bounds down to a uniform, comparable scale.  This is the content of Lemma \ref{lem:uniform-covering}.

Second, to prove part L), we cannot use the approach of Corollary \ref{cor:main_density_1} as proven above, since this can only give a packing estimate at a single scale.  Nor can we apply Theorem \ref{thm:main_density} at every scale, since this will give us a \emph{different} packing for each scale.  We must choose a \emph{single} collection of balls, and demonstrate packing bounds at all scales.

Our strategy to prove part L), in general terms, is to find some discrete approximation $\mu_{dis}$ of $\mu$ on the regions of big mass, prove $\beta$-bounds on $\mu_{dis}$ in terms of the $\beta$-bounds of $\mu$, and then use Theorem \ref{thm:main} to bound $\mu_{dis}$.  The main headache is that we can only get and use meaningful information on the $\beta$-numbers of $\mu_{dis}$ if: 1) the balls associated to $\mu_{dis}$ are centered at $\mu$-centers-of-mass; 2) the balls have some disjointness control; and 3) the balls have lower $\mu$-mass bounds.

These issues necessitate the use of an \emph{intermediary} packing measure $\nu$, which controls the original packing measure $\mu_{dis}$, but which admits the right properties 1)-3).  The following Subsection deals with this process.


\subsection{Discretizing a measure}

We use the following specialization of packing measure as an intermediary.
\begin{definition}
We say $\nu$ is \emph{($k$-dimensional) disjoint packing measure} if there is some collection of disjoint balls $\{B_{r_p}(x_p)\}_p$ so that $\nu$ has the form
\begin{gather}
\nu = \sum_p r_p^k \delta_{x_p} .
\end{gather}
Here, as always, $\delta_{x_p}$ is the Dirac delta at $x_p$.
\end{definition}

A key utility of a disjoint packing measure for us is the following Theorem, what is essentially a baby version of discrete Reifenberg Theorem \ref{thm:main_discrete}.
\begin{theorem}\label{thm:disjoint-discrete}
Let $\nu$ be a $k$-dimensional disjoint packing measure supported in $B_1$, and suppose
\begin{gather}
\nu \left\{ z \in B_1 : \int_0^1 \beta^k_{\nu, 2, \bar\eps_\beta}(z, r)^2 \frac{dr}{r} > M \right\} \leq \Gamma\, .
\end{gather}

Then provided $\bar\eps_\beta \leq c(k) \max\{M, \eps\}$, we have
\begin{gather}
\nu(B_1) \leq c(n)(1 + \eps + M) + \Gamma\, .
\end{gather}
\end{theorem}

\begin{proof}
By definition we can find a covering pair $(\cS, r_x)$ so that $\cS_0 = \emptyset$, and
\begin{gather}
\nu = \sum_{x \in \cS_+} r_x^k \delta_x\, .
\end{gather}
Trivially $\cS$ is a covering pair for $\nu$.

Apply Theorem \ref{thm:main} to obtain a $\cS' \subset \cS$.  Using Theorem \ref{thm:main} parts A), B), and the disjointness hypothesis, we have
\begin{align}
\nu(B_1)
&\leq c(n)(\eps + M) + \Gamma + \nu(B_{r_x}(\cS'))\, , \\
&= c(n)(\eps + M) + \Gamma + \sum_{x \in \cS'} r_x^k\, , \\
&\leq c(n)(1 + \eps + M) + \Gamma\, .
\end{align}
Note we see above that in the context of a packing measure, the packing control on $\cS'$ automatically gives rise to a measure estimate on $\cS'$.
\end{proof}

The notation $\cU$ introduced below is purely a technical issue.  Since we a priori only know $\beta$-number information down to scale $r_x$, in a random ball $B_{r_x}(x)$ we may not have $\beta$-number bounds down to a uniform, comparable scale.  It turns out one can always find a good subset $\cU(\cC) \subset \cC$ with the properties we desire.
\begin{lemma}\label{lem:uniform-covering}
Given a covering pair $(\cC, r_x)$ with $r_x \leq 1$, there is a subset $\cU(\cC, r_x) \subset \cC$ which admits the following properties: 

\begin{enumerate}
\item[A)] if $x \in \cU$, and $r \geq r_x$, then $y \in \cC \cap B_r(x) \implies r_y \leq 2r $.

\item[B)] $B_{5r_x}(\cU) \supset B_{r_x}(\cC)$, and $B_{5r_x}(\cU_+) \supset B_{r_x}(\cC_+)$.
\end{enumerate}

\end{lemma}

\begin{remark}
Property A) trivially holds for any $y \not\in B_{r_x}(\cC_+)$, and $r > 0$.
\end{remark}

\begin{proof}
Let us write $r_\alpha = 2^{-\alpha}$.  For each integer $\alpha \geq 0$, define
\begin{gather}
W_\alpha = \{ z \in \cC : r_{\alpha + 1} < r_z \leq r_\alpha \}.
\end{gather}
By assumption $\cup_\alpha W_\alpha = \cC_+$.

Let $K_0 = W_0$, and for each integer $\alpha \geq 1$, let
\begin{gather}
K_\alpha = W_\alpha \setminus B_{r_x}(\cup_{\beta \leq \alpha-1} W_\beta ).
\end{gather}

Now let
\begin{gather}
\cU_+ = \cup_\alpha K_\alpha, \quad \cU_0 = \cC_0 \setminus B_{r_x}(\cC_+).
\end{gather}

We claim that if $x \in \cU$, and $r \geq r_x$, then
\begin{gather}
y \in \cC \cap B_r(x) \implies r_y \leq 2 r.
\end{gather}
If $x \in \cU_0$, then this follows trivially, since $x \not \in B_{r_y}(y)$ when $r_y > 0$.  Otherwise, $x \in K_\alpha$, and if $r_y > 2r_x$ then $y \in W_\beta$ for some $\beta < \alpha$.  But then, by construction $x \not\in B_{r_y}(y)$, and so $r \geq r_y$.  This proves the claim.

We now claim that $B_{5 r_x}(\cU) \supset B_{r_x}(\cC)$.  Let us prove first that
\begin{gather}\label{eqn:W-inside-K}
W_\alpha \subset B_{4r_x}(\cup_{\beta \leq \alpha} K_\beta).
\end{gather}
Take $x_1 \in W_\alpha$, and write $\alpha_1 = \alpha$.  If $x_1 \in K_{\alpha_1}$ we are done.  Otherwise, by construction, there is an $x_2 \in W_{\alpha_2}$, with $\alpha_2 < \alpha_1$, so that $x_1 \in B_{r_{\alpha_2}}(x_2)$.  If $x_2 \not\in K_{\alpha_2}$, then we can find an $x_3$ as we did $x_2$.  We can thus build a sequence $x_1, x_2, \ldots$ satisfying
\begin{gather}
x_i \in W_{\alpha_i} \setminus K_{\alpha_i}, \quad x_{i-1} \in B_{r_{\alpha_i}}(x_i), \quad \alpha_i < \alpha_{i-1}.
\end{gather}

Since $W_0 = K_0$, this must terminate at some $x_N \in K_{\alpha_N}$, with $x_{N-1} \in B_{\alpha_N}(x_N)$.  We therefore calculate
\begin{align}
|x_1 - x_N| &< r_{\alpha_2} + \ldots + r_{\alpha_N} \\
&\leq 2 r_{\alpha_N} \\
&\leq 4 r_{x_N}.
\end{align}
In other words, $x_1 \in B_{4r_x}(K_{\alpha_N})$.  This proves \eqref{eqn:W-inside-K}, and in particular
\begin{gather}
B_{r_x}(W_\alpha) \subset B_{5r_x}(\cup_{\beta \leq \alpha} K_\beta).
\end{gather}

We deduce $B_{r_x}(\cC_+) \subset B_{5r_x}(\cU_+)$.  We then have
\begin{align}
\cC_0 \subset \cU_0 \cup B_{r_x}(\cC_+) \subset \cU_0 \cup B_{5r_x}(\cU_+).
\end{align}
This proves the second claim.
\end{proof}
\vspace{.25cm}

We can now state the main Lemma.  It says that if you have the right disjointness to your packing measure, then $\beta$-bounds of $\mu$ controls the packing measure.  Due to the center of mass issues we must use an intermediary $\nu$.  Recall the packing measure $\mu_{\cC'}$ (Defintion \ref{def:mu_C}) of a covering set is defined as 
\begin{gather*}
\mu_{\cC'} := \haus^k \llcorner \cC'_0 + \sum_{x \in \cC'_+} r_x^k \delta_x.
\end{gather*}

\begin{lemma}\label{lem:intermediary-packing}
Let $\mu$ be a non-negative Borel-regular measure in $B_1$, and let $(\cC, r_x)$ be a covering pair for $\mu$ with $\cC = \cC_+$ and $r_x \leq 1$.  Suppose $\cC' := \cC'_+ \subset \cU(\cC)$ is such that the balls $\{B_{2r_x}(x)\}_{x \in \cC'_+}$ are disjoint, and
\begin{gather}\label{eqn:intermediary-hyp}
\mu(B_{r_x}(x)) \geq a r_x^k \quad \forall x \in \cC' = \cC'_+.
\end{gather}

Then there is $k$-dimensional disjoint packing measure $\nu$, supported in $\overline{B_1}$, so that
\begin{enumerate}
\item[A)] for any $x \in \cC'$, $r \geq r_x$, we have $\mu_{\cC'}(B_r(x)) \leq \nu(B_{2r}(x))$.

\item[B)] for any $x \in B_1$, $0 < r < 1$, $\eps > 0$, we have
\begin{align}
&\nu \left( z \in B_r(x) : \int_0^{2r} \beta^k_{\nu,2,c(k)\eps/a}(z, s)^2 \frac{ds}{s} >  \frac{c(k)M}{a} \right) \\
&\quad \leq \frac{1}{a} \mu\left( z \in B_{4r}(x) : \int_{8r_z}^{32r} \beta^k_{\mu, 2, \eps}(z, s)^2 \frac{ds}{s} > M \right).
\end{align}
\end{enumerate}
\end{lemma}

\begin{proof}
Given $y \in \cC'$, let $Y \in \overline{B_{r_y}(y) \cap B_1}$ be the generalized $\mu$-center of mass of $B_{r_y}(y)$.  Define
\begin{gather}
\nu = \sum_{y \in \cC'} r_y^k \delta_{Y}.
\end{gather}
By assumption the balls $\{B_{r_y}(Y)\}_{y \in \cC'}$ are disjoint.

Let us prove conclusion A).  Given any $y \in B_r(x) \cap \cC'$, by disjointness and our choice of $r$ we know $r > r_y$.  Since the generalized $\mu$-center of mass $Y \in \overline{B_{r_y}(y)}$, we deduce
\begin{align}
\mu_{\cC'}(B_r(x)) = \sum_{y \in B_r(x) \cap \cC'} r_y^k  \leq \sum_{Y \in B_{2r}(x) \cap \spt\nu} r_y^k  = \nu(B_{2r}(x)).
\end{align}

We now work towards conclusion B).  We claim that for each $x \in \spt\nu$, and $r \geq r_x$, we have
\begin{gather}\label{eqn:intermediary-ineq}
\nu(B_r(x)) \leq \frac{1}{a} \mu(B_{4r}(x)) \quad\text{and} \quad \beta^k_{\nu, 2, 4^k \eps/a}(x, r)^2 \leq \frac{c(k)}{a} \beta^k_{\mu, 2, \eps}(x, 4r)^2.
\end{gather}

Let us prove \eqref{eqn:intermediary-ineq}.  If $Y \in B_r(x) \cap \spt\nu$, then by disjointness $r_y < r$.  We deduce from hypothesis \eqref{eqn:intermediary-hyp}
\begin{align}
\nu(B_r(x)) = \sum_{Y \in B_r(x) \cap \spt\nu} r_y^k \leq \frac{1}{a} \sum_{y \in B_{2r}(x) \cap \cC'} \mu(B_{r_y}(y)) \leq \frac{1}{a} \mu(B_{4r}(x)).
\end{align}

Let $V = V_\mu^k(x, 4r)$.  The following calculation (specifically, inequality \eqref{eqn:intermediary-com-beta}) is the key place we use that $Y$ is the generalized $\mu$-center of mass.  This inequality \eqref{eqn:intermediary-com-beta} follows by Proposition \ref{prop:gen-COM}, and of course we only need to sum over $y$ with $\mu(B_{r_y}(y)) < \infty$.  We calculate 
\begin{align}
r^{k+2} \beta_\nu^k(x, r)^2
&\leq \sum_{Y \in B_r(x) \cap \spt\nu} r_y^k d(Y, V)^2\, , \\
&\leq \sum_{y \in B_{2r}(x) \cap \cC'} \frac{r_y^k}{\mu(B_{r_y}(y))} \int_{B_{r_y}(y)} d(z, V)^2 d\mu(z)\, , \label{eqn:intermediary-com-beta} \\
&\leq \frac{1}{a} \int_{B_{4r}(x)} d(z, V)^2 d\mu(z) \, ,\\
&= \frac{4^{k+2}}{a} r^{k+2} \beta_{\mu,2,\eps}^k(x, 4r)^2\, ,
\end{align}
provided $\mu(B_{4r}(x)) > \eps (4r)^k$, which certainly holds if $\nu(B_r(x)) > 4^k \eps/a r^k$.  This proves \eqref{eqn:intermediary-ineq}.

Equation \eqref{eqn:intermediary-ineq} is a pointwise $\beta$-number bound at a single scale.  Towards conclusion B), we prove a pointwise inequality on the summed-$\beta$-numbers.  Let us fix an $x \in B_1$, and $r > 0$ arbitrary.

First, if $B_r(x) \cap \spt\nu$ is a singleton, then by definition $\beta_\nu(x, r) = 0$.  Therefore, by disjointness, we can assume $r > r_y$ for every $Y \in B_r(x) \cap \spt\nu$.  Here of course $Y$ is the generalized $\mu$-center of mass for $B_{r_y}(y)$, where $y \in \cC'$.

Otherwise, given $z \in B_{r_y}(y)$ with $y \in \cC'$, we have by \eqref{eqn:intermediary-ineq} and our choice of $\cC' \subset \cU$:
\begin{align}
\int_{8r_z}^{32r} \beta^k_{\mu, 2, \eps}(z, s)^2 ds/s
&\geq \int_{4r_y}^{32r} \beta^k_{\mu, 2, \eps}(z, s)^2 ds/s\, ,  \\
&\geq \frac{a}{c(k)} \int_{r_y}^{8r} \beta^k_{\nu, 2, 4^k\eps/a}(z, s)^2 ds/s\, , \\
&\geq \frac{a}{c(k)} \int_{0}^{2r} \beta^k_{\nu,2,c(k)\eps/a}(Y, s)^2 ds/s .
\end{align}
In particular, suppressing the $\eps$ notation, if $\int_0^{2r} \beta_\nu(Y, s)^2 ds/s > \frac{c(k) M}{a}$, then
\begin{gather}
\mu \left( z \in B_{r_y}(y) : \int_{8r_z}^{32r} \beta_\mu(z, s)^2 ds/s > M \right) \geq \mu(B_{r_y}(y)) \geq  a r_y^k.
\end{gather}

We therefore calculate, 
\begin{align}
&\nu \left( z \in B_r(x) : \int_0^{2r} \beta_\nu(z, s)^2 ds/s > \frac{c(k)M}{a} \right) \\
&= \sum \left\{ r_y^k :  Y \in B_r(x) \cap \spt\nu \,\,\text{ and } \int_0^{2r} \beta_\nu(Y, s)^2 ds/s > \frac{c(k)M}{a}  \right\}\\
&\leq \frac{1}{a} \sum_{y \in B_{2r}(x)} \mu \left( z \in B_{r_y}(y) : \int_{8r_z}^{32r} \beta_\mu(z, s)^2 ds/s > M \right) \\
&\leq \frac{1}{a} \mu \left( z \in B_{4r}(x) : \int_{8r_z}^{32r} \beta_\mu(z, s)^2 ds/s > M \right).
\end{align}
This completes the proof of Lemma \ref{lem:intermediary-packing}.
\end{proof}

The main Theorem of this subsection is the following.  It follows easily from Lemma \ref{lem:intermediary-packing}.
\begin{theorem}\label{thm:discretized-packing}
Let $\mu$ be a non-negative Borel-regular measure in $B_1$, and let $(\cC, r_x)$ be a covering pair for $\mu$ with $\cC = \cC_+$ and $r_x \leq 1$.  Suppose $\cC' := \cC'_+ \subset \cU(\cC)$ is such that the balls $\{B_{2r_x}(x)\}_{x \in \cC'}$ are disjoint, and
\begin{gather}
\mu(B_{r_x}(x)) \geq a r_x^k \quad \forall x \in \cC' = \cC'_+.
\end{gather}

Take $x \in \cC'$ and $r \geq r_x$.  If
\begin{gather}\label{eqn:disc-packing-hyp}
\mu\left\{ z \in B_{16r}(x) : \int_{8r_z}^{64r} \beta^k_{\mu, 2, \bar\eps_\beta}(z, r) dr/r > M \right\} \leq \Gamma r^k,
\end{gather}
and $\bar\eps_\beta \leq c(k) \max\{M, \eps\}$, then
\begin{gather}
\mu_{\cC'}(B_r(x)) \equiv \sum_{y \in \cC' \cap B_r(x)} r_y^k \leq \left( c(n) + \frac{c(n)(\eps + M) + \Gamma}{a} \right) r^k.
\end{gather}
\end{theorem}

\begin{proof}
Let $\nu$ be the disjoint discrete packing measure of Lemma \ref{lem:intermediary-packing}.  Whenever \eqref{eqn:disc-packing-hyp} holds, then by Lemma \ref{lem:intermediary-packing} we have
\begin{gather}
\nu\left( z \in B_r(x) : \int_0^{2r} \beta_{\nu, 2, c(k)\bar\eps_\beta/a}(z, r)^2 \frac{ds}{s} > \frac{c(k)M}{a} \right) \leq \frac{\Gamma}{a} r^k.
\end{gather}

Using Theorem \ref{thm:disjoint-discrete} at scale $B_r(x)$, provided $\bar\eps_\beta \leq c(k) \max\{M, \eps\}$, we have
\begin{gather}
\frac{\nu(B_r(x))}{r^k} \leq c(n) \left(1 + \frac{\eps}{a} + \frac{c(k)M}{a}\right) + \frac{\Gamma}{a}.
\end{gather}
\end{proof}

\begin{remark}
Let us remind the reader that \eqref{eqn:disc-packing-hyp} is a \emph{weaker} assumption than the corresponding estimate without the factor of $8$.
\end{remark}

\vspace{.25cm}

\subsection{Proofs of Theorems}

We are now ready to prove the Theorems involving upper and lower density.  Corollaries \ref{cor:main_density_1}, \ref{cor:main_density_2}, \ref{cor:main_upperahlfors_1}, \ref{cor:main_upperahlfors_2} are immediate from Theorems \ref{thm:main_density}, \ref{thm:main_upperahlfors}, and we will not explicitly prove them.

\begin{proof}[Proof of Theorems \ref{thm:main_density}, \ref{thm:main_upperahlfors}, part L)]
Let us first handle the $\cC'_+$.  Define the covering pair $(\cS, s_x)$ by setting
\begin{gather}
\cS_+ = \{ x \in \cC_+ : \Theta^{*, k}(\mu, \cC, x) \geq a \} , \quad \cS_0 = \cC \setminus \cS_+.
\end{gather}
Choose $s_x \in [r_x/50, r_x/25]$ so that $\mu(B_{s_x}(x)) \geq \frac{\omega_k a}{2} s_x^k$.  Clearly $\cS$ is a covering pair for $\mu$, and by assumption $\mu(\cC_+ \setminus \cS_+) = 0$.

Define $\cC'_+$ to be (the centers of) a Vitali subcover of
\begin{gather}
\{ B_{5s_x}(x) : x \in \cU(\cS)_+ \}.
\end{gather}
In other words, the balls $\{B_{5s_x}(x)\}_{x \in \cC'_+}$ are disjoint, while
\begin{gather}
B_{r_x}(\cC'_+) \supset B_{25 s_x}(\cC'_+) \supset B_{5s_x}(\cU(\cS)_+) \supset \cS_+
\end{gather}
In particular, $\mu(\cC_+ \setminus B_{r_x}(\cC'_+)) = 0$.  Clearly $\cC_+'$ is independent of $\eps$.

We demonstrate that $\cC'_+$ admits packing bounds at scale $B_1$ (in the case of Theorem \ref{thm:main_density}), or at all scales (in the case Theorem \ref{thm:main_upperahlfors}).  First observe that $\cC'_+ \subset \cU(\cS)$ and $\mu$ as defined satisfy all hypotheses of Theorem \ref{thm:discretized-packing} except \eqref{eqn:disc-packing-hyp}.  Fix also $\eps = \bar\eps_\beta/c(k)$ as required by Theorem \ref{thm:discretized-packing}. 

Since $r_x \leq 50 s_x$, and $\mu$ is supported in $B_1$ the hypothesis of Theorem \ref{thm:main_density} implies \eqref{eqn:disc-packing-hyp} holds at scale $B_1$, except with $c(k)M$ in place of $M$, and $c(k)\Gamma$ in place of $\Gamma$.  We can therefore apply Theorem \ref{thm:discretized-packing}, recalling our choice $\eps = \bar\eps_\beta/c(k)$, to deduce
\begin{gather}
\sum_{x \in \cC'_+} r_x^k \leq 50^k \sum_{x \in \cC'_+} s_x^k \leq c(n) + \frac{c(n)(\bar\eps_\beta + M) + c(k)\Gamma}{a}.
\end{gather}

On the other hand, using again $r_x \leq 50s_x$, the hypothesis of Theorem \ref{thm:main_upperahlfors} implies \eqref{eqn:disc-packing-hyp} holds at any scale $B_r(x)$, except with $\Gamma = M$.  In this case we can apply Theorem \ref{thm:discretized-packing} at any $B_r(x)$ to deduce
\begin{gather}
\sum_{x \in B_r(z) \cap \cC'_+} r_x^k \leq \left( c(n) + \frac{c(n)(\bar\eps_\beta + M)}{a}\right) r^k.
\end{gather}
This proves the required packing estimate for $\cC'_+$.

We now handle $\cC'_0$.  Fix $\delta > 0$, and define the covering pair
\begin{gather}
\cS_+ = \{ x \in \cC_0 : \Theta^{*, k}(\mu, \cC, x) \geq a \} , \quad \cS_0 = \cS \setminus \cS_+.
\end{gather}
Choose $s_x < \delta$ so that $\mu(B_{s_x}(x)) \geq \frac{\omega_k a}{2} s_x^k$.

Define $\cC'_0 = \cS_+ \cap \cC_0$, which is clearly independent of $\delta$ and $\eps$.  By assumption $\mu(\cC_0 \setminus \cS_+) = \mu(\cC_0 \setminus \cC'_0) = 0$.  By the same reasoning as above, we can find a Vitali subcover $\cS'$ of
\begin{gather}
\{B_{5s_x}(x) : x \in \cU(\cS) \},
\end{gather}
so that $\{B_{5s_x}(x)\}_{x \in \cS'_+}$ are disjoint, and
\begin{gather}
B_{25s_x}(\cS'_+) \supset \cS_+
\end{gather}

By the same argument as before, taking $\eps = \bar\eps_\beta/c(k)$, the hypothesis of Theorem \ref{thm:main_density} together with Theorem \ref{thm:discretized-packing} implies a packing estimate
\begin{gather}
\sum_{x \in \cS'} s_x^k \leq c(n) + \frac{c(n)(\eps + M) + c(k)\Gamma}{a},
\end{gather}
which in turn gives
\begin{gather}
\haus^k_\delta(\cC'_0 \cap B_1) \leq c(n) + \frac{c(n)(\bar\eps_\beta + M) + c(k)\Gamma}{a}.
\end{gather}

While the hypothesis of Theorem \ref{thm:main_upperahlfors} implies that, for any $z \in B_1$ and $0 < r \leq 1$, 
\begin{gather}
\sum_{x \in \cS' \cap B_r(z)} s_x^k \leq \left( c(n) + \frac{c(n)(\eps + M)}{a}\right) r^k ,
\end{gather}
and therefore
\begin{gather}
\haus^k_\delta(\cC'_0 \cap B_r(x)) \leq \left( c(n) + \frac{c(n)(\bar\eps_\beta + M)}{a} \right) r^k.
\end{gather}

Since $\delta$ is arbitrary, this proves the required Hausdorff estimate on $\cC'_0$.
\end{proof}

\vspace{.25cm}

\begin{proof}[Proof of Theorems \ref{thm:main_density}, \ref{thm:main_upperahlfors}, part U)]
Let us first prove Theorem \ref{thm:main_density} part U).   Fix $\eps = \bar\eps_\beta/c(k)$ as required by Theorem \ref{thm:main}.

Let $(\cS, s_x)$ be a covering pair defined by
\begin{gather}
\cS_+ = \{x \in \cC : \Theta_*^k(\mu, \cC, x) \leq b \}, \quad \cS_0 = \cC \setminus \cS_+,
\end{gather}
and choosing $s_x \in (0, 1)$ so that $\mu(B_{s_x}(x)) \leq 2b \omega_k s_x^k$.  By assumption $\mu(B_1 \setminus \cS_+) = \mu(\cS_0) = 0$, and $\cS$ is closed.

Apply Theorem \ref{thm:main} and obtain a $\cS' \subset \cS$.  Trivially we have $\mu(\cS'_0) = 0$.  We have
\begin{align}
\mu(B_1)
&\leq c(n)(\eps + M) + \Gamma + \sum_{x \in \cS'_+} \mu(B_{s_x}(x)) \\
&\leq c(n)(\eps + M) + \Gamma + b \sum_{x \in \cS'_+} s_x^k \\
&\leq c(n)(\bar\eps_\beta + M + b) + \Gamma. 
\end{align}
This proves Theorem \ref{thm:main_density} part U).

We now use this to prove Theorem \ref{thm:main_upperahlfors}.  We first observe that, trivially, $(\cC, r_x/5)$ is a covering pair for $\mu$, and satisfies
\begin{gather}
\mu \left\{ z \in B_r(x) : \int_{5 (r_z/5)}^{2r} \beta^k_{\mu,2, \bar\eps_\beta}(z, s)^2 \frac{ds}{s} > M \right\} \leq M r^k \quad \forall x \in \cS, r \geq r_x.
\end{gather}

Let $\cC' = \cU(\cC, r_x/5)$ as in Lemma \ref{lem:uniform-covering}.  Then from our choice of $r_x/5$, we have $B_{r_x}(\cC') \supset \cC$.  Now take $x \in \cC'$ and $r \geq r_x$, or take $x \not\in B_{r_x/5}(\cC)$ and $r > 0$.  In either case, we have by construction that
\begin{gather}
y \in B_r(x) \cap \cC \implies r_y \leq 10r.
\end{gather}

Let $\mu' = \mu \llcorner B_r(x) \leq \mu$.  Then $\cC \cap B_r(x)$ is a covering pair for $\mu'$, which (by monotonicity of $\beta$) satisfies the requirements of Theorem \ref{thm:main_density} part U) at scale $B_{10r}(x)$.  We deduce that
\begin{gather}
\mu(B_r(x)) = \mu'(B_{10r}(x)) \leq c(n)(\bar\eps_\beta + M + b)(10r)^k,
\end{gather}
which evidently finishes the proof.
\end{proof}

\vspace{1cm}

\section{Rectifiability}\label{sec:rect}

In this section, we prove the rectifiability properties promised in the main theorems, and in particular Theorem \ref{thm:main_rect}.  All these properties will be obtained as corollaries of the following Proposition, which is the main statement of this section.  The main purpose of the following Proposition, in comparison with other results in the literature, is that we sharply weaken the density assumptions required on the measure:

\begin{proposition}\label{prop:rect_with_bounds}
 Let $\mu$ be a non-negative Borel-regular measure such that for $\mu$-a.e.:
 \begin{gather}
  \Theta_*^k (\mu, x)\leq b<\infty\, , \quad \Theta^{*,k}(\mu, x)\geq a>0\, .
 \end{gather}
Suppose also that for $\mu$-a.e.
 \begin{gather}
 \int_0^2 \beta^k_{\mu, 2, \bar\eps_\beta}(x, r)^2 \frac{dr}{r} \leq M\, ,
 \end{gather}
where $\bar \eps_\beta\leq \min\{a, b\}/c(n, b, M)$. Then $\mu$ is $k$-rectifiable, and in particular we have that the density $\Theta^{k}(\mu, x)$ exists with $a\leq \Theta^k(\mu, x)\leq b$ $\mu$-a.e.
\end{proposition}
\vspace{5mm}

Before proving this Theorem, we briefly discuss the main ideas behind the proofs in this section.\\

First, to prove rectifiability we need control over the $\beta$-numbers all the way down to radius $0$, not just some positive $r_x$.  Second, we need to assume some kind of upper and lower density bounds on $\mu$, see for instance Example \ref{ex:sharp}.  The lower bound prevents any higher dimensional pieces from appearing, and the upper bound prevents any lower dimensional pieces from appearing.  Note, however, we will only assume upper bounds on the \textit{lower} density of the measure and lower bounds on {\it upper} density, which is the weakest density assumption one might hope for. 

With these two assumptions we can apply Theorem \ref{thm:main_upperahlfors} and obtain full-blown upper Ahlfors' regularity for the measure $\mu$ under consideration. This apriori bound immediately implies that $\mu\ll\haus^k$, and more importantly can be used as in Theorem \ref{thm:packing-control} part C) to guarantee that a bad ball carries very small measure.

With this information, by repeating the construction of a good tree rooted at any $\B r x$, we can build a Lipschitz manifold $T_\infty$ such that $\mu (\B r x \setminus T_\infty)\leq c \epsilon r^k$.

By itself, this is not enough to prove rectifiability, as it is clear from Example \ref{ex_density}. The problem is that $\mu \B r x$ can be much smaller than the part of the measure covered by $T_\infty$. However, under lower bounds on the density, we can guarantee that $\B r x$ (for suitably chosen $x$ and $r$) has at least some mass: $\mu(\B r x)\geq a r^k$. If $a>0$ and $\epsilon$ is sufficiently small, this implies that at least half of the measure of $\mu\llcorner \B r x$ is covered by our Lipschitz manifold $T_\infty$. With a simple inductive argument, this implies the rectifiability of $\mu$.

\vspace{2.5mm}

We start by showing that, under uniform upper Ahlfors bounds, smallness of the $\beta$ numbers implies that we can cover a big chunk of the support of $\mu$ with a Lipschitz manifold.  

\begin{lemma}\label{lemma_rect_manifold}
For any $\eps > 0$, there is a $\delta_2(n, \Gamma, \eps)$ so that the following holds: let $\mu$ be a non-negative Borel-regular measure supported in $B_1$, having uniform upper bounds
\begin{gather}
\mu \ton{B_r(x)} \leq \Gamma r^k \quad \forall x \in B_1, \forall 0 < r \leq 1\, ,
\end{gather}
and satisfying
\begin{gather}\label{eq_small_bound}
\int_0^2 \beta^k_{\mu,2,\bar\eps_\beta}(z, r)^2 \frac{dr}{r} \leq \delta^2 \quad \text{for $\mu$-a.e. $z \in B_1$}\, ,
\end{gather}
with $\delta \leq \delta_2$, and $\bar\eps_\beta \leq \eps/c(n, \Gamma)$.

Then there is $e^{c(n,\eps,\Gamma)\delta^2}$-Lipschitz mapping $\tau : \R^k \to \R^n$, so that
\begin{gather}
\mu(B_1 \setminus \tau(B_2^k)) \leq c_4(n)\eps + c_5(n, \eps, \Gamma)\delta^2\, .
\end{gather}
\end{lemma}

\begin{proof}
Let us remark that the proof of the following will use heavily the results and notation of Section \ref{section:good-tree}.\\

First, by monotonicity of $\beta$ (Remark \ref{rem:pointwise-integral}), and the fact $\mu$ is supported in $B_1$, we have
\begin{gather}
\sum_{\alpha \in \dZ \, \, : \, \, 2^\alpha \leq 2} \beta^k_{\mu, 2, 2^k \bar\eps_\beta}(z, 2^\alpha)^2  \leq c(k) \delta^2\, .
\end{gather}

Choose a scale $\rho(n,\eps, \Gamma) < 1/20$ as in Theorem \ref{thm:packing-control} so that any bad ball $B_r(x)$ with $m = m_0(n,\rho)\eps$ as in Theorem \ref{thm:packing-control} part A), must have $\mu(B_r(x)) \leq \eps r^k$.

Define the covering pair $\cC = \cC_0 = B_1(0)$, with $r_x \equiv 0$.  With our choice of $\rho$ above, any ball $B_r(x)$ must be either stop or good (as per Definition \ref{def:balls}), but cannot be bad.  

We can clearly assume without loss of generality that $B_1$ is a good ball, and build a good tree $\cT = \cT(B_1)$ rooted in $B_1$ following the construction of Section \ref{section:good-tree}.

Let $T_\infty=\tau_\infty(\R^k)$ be the Lipschitz manifold of Theorem \ref{thm:good-tree}.  By the above reasoning, $\cB(\cT) = \emptyset$, and by construction, $\cC_0(\cT) \subset T_\infty$.  We deduce
\begin{gather}
\mu(B_1 \setminus T_\infty) \leq \mu(B_1 \setminus C_0(\cT)) \leq c(n) \eps + c(n,\rho) \delta^2\, .
\end{gather}
This proves the Lemma.
\end{proof}

In order to prove rectifiability, we need smallness of \eqref{eq_small_bound} \emph{relative to} the lower density bound.  However, rectifiability is a local property, and so we can afford to drop in scale until \eqref{eq_small_bound} is verified.  This is the content of the following Lemma.
\begin{lemma}\label{lem:scale-of-D}
Take $\tau, \chi > 0$.  Suppose $\mu$ is a non-negative Borel-regular measure supported in $B_1$, with upper mass bounds
\begin{gather}
\mu(B_r(x)) \leq \Gamma r^k \quad \forall x \in B_1, 0 < r < 1\, ,
\end{gather}
and satisfying
\begin{gather}
\int_{B_1} \int_0^2 \beta^k_{\mu, 2, \bar\eps_\beta}(z, r)^2 \frac{dr}{r} d\mu(z) < \infty\, .
\end{gather}

Then for $\mu$-a.e. $x \in B_1$, there is a scale $R_x > 0$ so that
\begin{equation}\label{eqn:mostly-small-D}
\mu\left\{ z \in B_r(x) : \int_0^{2 r} \beta_{\mu, 2, \bar\eps_\beta}(z, s)^2 \frac{ds}{s} > \tau \right\} \leq \chi r^k \quad \forall 0 < r < R_x\, .
\end{equation}
\end{lemma}

\begin{proof}
Let $F$ be the set of points $x$ which fail \eqref{eqn:mostly-small-D}.  Fix an $0 < R < 1/4$.  For $\mu$-a.e. $x \in F$, we can find a scale $0 < s_x < R$ for which
\begin{gather}
s_x^k < \frac{1}{\chi} \mu \left\{ z \in B_{s_x}(x) : \int_0^{2 s_x} \beta^k_{\mu, 2, \bar\eps_\beta}(z, r)^2 \frac{dr}{r} > \tau \right\} \, .
\end{gather}
Choose a Vitali subcovering $\{B_{5s_i}(x_i)\}_i$ of $\{B_{5s_x}(x)\}_{x \in F}$, so that the $s_i$-balls are disjoint.  Then we calculate
\begin{align}
\mu(F)
&\leq \sum_i \mu(B_{5s_i}(x_i))\notag \\
&\leq c(k) \Gamma \sum_i s_i^k \notag\\
&\leq c(k) \frac{\Gamma}{\chi} \sum_i \frac{1}{\tau} \int_{B_{s_i}(x_i)} \int_0^{2s_i} \beta^k_{\mu, 2, \bar\eps_\beta}(z, r)^2 \frac{dr}{r} d\mu(z) \\
&\leq c(k) \frac{\Gamma}{\chi \tau} \int_{B_1} \int_0^{2R} \beta^k_{\mu, 2, \bar\eps_\beta}(z, r)^2 \frac{dr}{r} d\mu(z) \, .\notag
\end{align}
By the dominated convergence theorem the RHS tends to $0$ as $R \to 0$.  This shows $\mu(F) = 0$.
\end{proof}

We are now ready to prove our rectifiable Proposition.
\begin{proof}[Proof of Proposition \ref{prop:rect_with_bounds}]
Begin by applying Theorem \ref{thm:main_upperahlfors} part U) to deduce for our measure the estimate\footnote{In some cases (as in Section \ref{subsection:structure-of-C_0}) $\mu$ may already be upper-Ahlfors-regular, in the sense that $\mu(B_r(x)) \leq \Gamma(n, b, M) r^k \quad \forall x \in B_1, \ 0 < r \leq 1 $, and one does not need these first few lines of the argument.}
\begin{gather}\label{eq_upper_Ahl_measure}
\mu(B_r(x)) \leq c(n) (b + M) r^k =: \Gamma r^k \quad \forall x \in B_1, \ 0 < r \leq 1\, .
\end{gather}
This implies in particular that
\begin{gather}
\int_{B_1} \int_0^2 \beta^k_{\mu, 2, \bar\eps_\beta}(z, r)^2 \frac{dr}{r} d\mu(z) < \infty\, ,
\end{gather}
and that $\mu$ is absolutely continuous wrt $\haus^k$, in symbols $\mu \ll \haus^k$. 

Choose
\begin{gather}
\eps = \frac{\omega_k a}{c_4(n) 16 \cdot 10^k}, \quad \delta^2 = \min \left\{ \frac{\omega_k a}{c_5(n, \eps, \Gamma) 16 \cdot 10^k}, \delta_2(n, b, M, \eps)^2 \right\}, \quad \chi = \frac{\omega_k a}{8 \cdot 10^k}, \quad \tau = \delta^2\, ,
\end{gather}
so that
\begin{gather}
\chi + c_4 \eps + c_5 \delta^2 \leq \frac{\omega_k a}{4 \cdot 10^k}\, .
\end{gather}

We claim the following: let $C$ be any closed set.  Then there is a finite collection $T_1, \ldots, T_N$ of Lipschitz $k$-manifolds so that
\begin{gather}\label{eq_claim_rect}
\mu(B_1 \setminus (C \cup T_1 \cup \cdots \cup T_N)) \leq \frac{1}{2} \mu(B_1 \setminus C)\, .
\end{gather}
To be specific, each $T_i$ is the image of a bi-Lipschitz mapping $\tau_i : B_3^k \to \R^n$.

Given this claim, rectifiability of $\mu$ follows by an straightforward iteration.  By taking $C = \emptyset$, then $C = T_1 \cup \cdots \cup T_N$, etc. we can inductively construct a nested sequence of closed $k$-rectifiable sets $\cK_1 \subset \cK_2 \subset \ldots$ so that
\begin{gather}
\mu(B_1 \setminus \cK_i) \leq 2^{-i} \mu(B_1) \leq 2^{-i} \Gamma.
\end{gather}
Letting $\cK = \cup_i \cK_i$, then $\mu(B_1 \setminus \cK) = 0$.  We have already established $\mu \ll \haus^k$, and therefore $\mu$ is $k$-rectifiable.

Now we turn to the claim \eqref{eq_claim_rect}. By assumption and Lemma \ref{lem:scale-of-D}, we can find a measurable set $A \subset B_1 \setminus C$ with $\mu(B_1 \setminus (A \cup C))=0$ such that for all $x\in A$, there is a radius $s_x$ so that
\begin{enumerate}
\item[A)] $B_{s_x}(x) \subset B_1 \setminus C$,

\item[B)] $\mu(B_{s_x/5}(x)) \geq \frac{\omega_k a}{2} (s_x/5)^k$,

\item[C)] and
\begin{gather}
\mu \left\{ z \in B_{s_x}(x) : \int_0^{2s_x} \beta^k_{\mu, 2, \bar\eps_\beta}(z, r)^2 \frac{dr}{r} > \delta^2 \right\} \leq \chi s_x^k\, .
\end{gather}
\end{enumerate}

In particular, by Lemma \ref{lemma_rect_manifold} and our choice of constants, there is for every such $x$ a Lipschitz $k$-manifold $T_x$ so that
\begin{gather}
\mu(B_{s_x}(x)\setminus T_x) \leq \frac{\omega_k a}{4 \cdot 10^k} s_x^k\, .
\end{gather}

Let $\{B_{s_i}(x_i)\}$ be a Vitali subcover of $\cup_{x\in A} \B {r_x}{x}$, so that the $B_{s_i/5}$-balls are disjoint, and write $T_i = T_{x_i}$.  We then calculate
\begin{align}
\mu\ton{B_1 \setminus (C \cup \bigcup_i T_i)}\, 
&\leq \sum_i \mu(B_{s_i}(x_i) \setminus T_i)\, , \\
&\leq \sum_i \frac{\omega_k a}{4 \cdot 10^k} s_i^k\, , \\
&\leq \frac{1}{4} \sum_i \frac{\omega_k a}{2} (s_i/5)^k\, , \\
&\leq \frac{1}{4} \mu(B_1 \setminus C)\, .
\end{align}
This proves the Claim, taking $N$ sufficiently large.

\end{proof}

Now we are in a position to obtain Theorem \ref{thm:main_rect} as a simple corollary of the previous result. 

\begin{proof}[Proof of Theorem \ref{thm:main_rect}]
By assumption there is some $x \in \spt\mu$, with $\beta^k_{\mu, 2, 0}(x, 2) < \infty$.  Therefore if $V = V(x, 2)$, and
\begin{gather}
U_j = \{ z \in B_1 : \dist(z, V) > 1/j\}\, , 
\end{gather}
then $\mu \llcorner U_j$ is Radon, and satisfies
\begin{gather}\label{eqn:theta-is-the-same-2}
\Theta^{*, k}(\mu \llcorner U_j, x) = \Theta^{*,k}(\mu, x) \quad \forall x \in U_j\, .
\end{gather}

For each integer $i$, $j$, define the Borel set
\begin{gather}
K_{ij} = \cur{z \in U_j : \int_0^2 \beta^k_{\mu, 2, 0}(z, r)^2 dr/r \leq i, \quad \text{and} \quad \Theta^{*, k}(\mu, z) > 1/i }\, .
\end{gather}
By assumption and \eqref{eqn:theta-is-the-same-2}, we have $\mu(U_j \setminus \cup_i K_{ij}) = 0$ for each $j$.

Using \eqref{eqn:theta-is-the-same-2} and Lemma \ref{lem:density-to-ineq}, we have
\begin{gather}
\haus^k \llcorner K_{ij} \leq i \cdot \mu \llcorner K_{ij}\, ,
\end{gather}
and therefore by monotonicity of $\beta$ we know
\begin{gather}
\int_0^2 \beta^k_{\haus^k \llcorner K_{ij}, 2, 0}(z, r)^2 dr/r \leq i^2 \quad \forall z \in K_{ij}\, .
\end{gather}

Since $\mu \llcorner U_j$ is a finite measure, so is $\haus^k \llcorner K_{ij}$, and hence we have the standard density estimates
\begin{gather}
2^{-k} \leq \Theta^{*, k}(\haus^k \llcorner K_{ij}, x) \leq 1 \quad \text{for $\haus^k$-a.e. $x \in K_{ij}$}\, .
\end{gather}
Using Proposition \ref{prop:rect_with_bounds}, we deduce $K_{ij}$ is rectifiable.

Clearly 
\begin{gather}
\cK_0 = V \cup \bigcup_{ij} K_{ij}\, 
\end{gather}
is still a rectifiable set, and $\mu(\B 1 0 \setminus \cK_0)=0$. This concludes the first part of the proof.

\vspace{2.5 mm} Under the additional assumption that $\Theta_*^k(\mu, x) < \infty$ for $\mu$-a.e. $x$, define the Borel sets
\begin{gather}
W_i = \cur{z \in B_1 : \int_0^2 \beta_{\mu, 2, 0}(z, r) dr/r \leq i \quad \text{and} \quad \Theta_*^k(\mu, z) \leq i }\, .
\end{gather}
By assumption, $\mu(B_1 \setminus \cup_i W_i) = 0$.

For any $i$ we trivially have that
\begin{gather}
\int_0^2 \beta_{\mu \llcorner W_i, 2,0}(z, r)^2 dr/r \leq i, \quad \text{and} \quad \Theta_*^k(\mu \llcorner W_i, z) \leq i \quad \forall z \in W_i\, .
\end{gather}
Therefore, by Theorem \ref{thm:main_upperahlfors} part U), 
\begin{gather}
\mu(W_i \cap B_r(x)) \leq c(i, n) r^k \quad \forall x \in B_1, 0 < r \leq 1\, .
\end{gather}
Therefore $\mu \llcorner W_i \ll \haus^k$ for each $i$, and hence $\mu \ll \haus^k$ also.
\end{proof}

\vspace{1cm}

\section{Remaining Theorems}

In this last section, we gather all the proofs of the secondary results that were stated in the introduction and in Section \ref{section:applications} ``Applications''.  The proofs are more-or-less direct corollaries of the results already proven.

\subsection{Proof of the statements in the introduction}
Theorem \ref{thm:teaser1} follows from Theorem \ref{thm:main} by choosing the trivial covering pair $\cC=\cC_0= \B 1 0$, and $r_x \equiv 0$. In a similar way, Corollary \ref{cor:teaser1} is a direct consequence of Corollary \ref{cor:main_density_1}.

%
%

\begin{proof}[Proof of Theorem \ref{thm:teaser2}]
Consider the measure $\mu_{x, r}(A) := r^{-k} \mu(x + r A)$.  Depending on whether $\mu$ satisfies the condition A), B), or C), then $\mu_{x,r}$ satisfies (respectively):
\begin{enumerate}
\item[A)] 
\begin{gather}
\int_0^2 \beta_{\mu_{x,r}, 2, 0}(z, r)^2 \frac{dr}{r} \leq M \quad \text{for $\mu_{x,r}$-a.e. $z \in B_1$}\, ,
\end{gather}

\item[B)] or
\begin{gather}
\int_{B_1} \int_0^2 \beta_{\mu_{x, r}, 2, 0}(z, r)^2 \frac{dr}{r} d\mu_{x,r}(z) \leq M^2\, ,
\end{gather}

\item[C)] or
\begin{gather}
\int_{B_1} \int_0^2 \beta_{\mu_{x, r}, 2, 0}(z, r)^2 \frac{dr}{r} d\mu_{x, r}(z) \leq M \mu_{x, r}(B_1)\, .
\end{gather}
\end{enumerate}

In cases A) or B), apply Theorem \ref{thm:main_density} 
with $\Gamma = M$, to deduce
\begin{gather}
\mu_{x, r}(B_1) \leq c(n)(M + b + \eps) + M \quad \forall \eps > 0\, .
\end{gather}
In case C), apply Theorem \ref{thm:main_density} with $\Gamma = \frac{1}{2} \mu_{x, r}(B_1)$, to deduce
\begin{gather}
\mu_{x, r}(B_1) \leq c(n)(2M + b + \eps) + \frac{1}{2} \mu_{x,r}(B_1) \quad \forall \eps > 0\, .
\end{gather}
Therefore, in all cases, we have $\mu(B_r(x)) \leq c(n)(M + b) r^k$.

The rectifiability is a direct consequence of Theorem \ref{thm:main_rect}.
\end{proof}

\begin{proof}[Proof of Corollary \ref{cor:teaser2}]
We first note that if $\haus^k(S) < \infty$ (as is automatic in case C), then
\begin{gather}
2^{-k} \leq \Theta^{*, k}(\mu, x) \leq 1 \quad \text{for $\haus^k$-a.e. $x$},
\end{gather}
and can therefore apply Theorem \ref{thm:teaser2} directly.

Let us consider cases A) and B).  In either case, there is an $x \in \spt\mu \cap B_1$ so that
\begin{gather}
\beta^k_{\mu, 2}(x, 2) < \infty.
\end{gather}
Trivially, $S \cap V^k(x, 2)$ is $k$-rectifiable, and
\begin{gather}
\haus^k(V^k(x, 2) \cap S \cap B_r(x)) \leq c(n) r^k \quad \forall x \in B_1, 0 < r \leq 1.
\end{gather}

On the other hand, let us define for each $i$
\begin{gather}
S_i = \{ z \in B_1 : d(z, V^k(x, 2)) > 1/i \} .
\end{gather}
Then $\haus^k(S_i) < \infty$, and by monotonicity of $\beta$ the measure $\haus^k \llcorner S_i$ satisfies either condition A) or B).  We deduce from Theorem \ref{thm:teaser2} that $S_i$ is $k$-rectifiable, and admits Hausdorff bounds
\begin{gather}
\haus^k(S_i \cap B_r(x)) \leq c(n)(1 + M) r^k \quad \forall x \in B_1, 0 < r \leq 1.
\end{gather}
Taking $i \to \infty$, we deduce the required estimate.
\end{proof}

We conclude this section with the proof of Corollary \ref{cor:teaser3}, which follows immediately from our measure bounds and the result of \cite{girela}.
\begin{proof}[Proof of Corollary \ref{cor:teaser3}]
By Theorem \ref{thm:teaser2}, we obtain that $\mu$ is upper Ahlfors regular. In particular, we have the uniform bounds $\mu(B_r(x))/r^k \leq c(n)(M + b)$ for all $x$ and $0 < r \leq 1$.  Therefore condition \eqref{eq:girela-hyp} is satisfied, and the conclusion follows by \cite{girela}.
\end{proof}

\subsection{Applications}
Most of the work to obtain the following Theorems is in decomposing the measure in various ways via Theorem \ref{thm:main}.

\begin{proof}[Proof of Theorem \ref{thm:main-thm-L1}]
Apply Theorem \ref{thm:main} at scale $B_1$, with the trivial covering pair $\cC \equiv \cC_0 = B_1$ (so $r_x \equiv 0$), and $\Gamma = M$, to obtain a closed, $k$-rectifiable set $\cK_0$ with
\begin{gather}
\haus^k(\cK_0 \cap B_r(x)) \leq c(n)r^k \quad \forall x \in B_1, 0 < r \leq 1, \quad\text{and} \quad \mu(B_1 \setminus \cK_0) \leq c(n)(M + \eps)\, .
\end{gather}
Since $\cK_0$ is closed, we know
\begin{gather}\label{eqn:theta-is-the-same}
\Theta^{*, k}(\mu \llcorner (B_1 \setminus \cK_0), x) = \Theta^{*, k}(\mu, x) \quad \text{for every $x \in B_1 \setminus \cK_0$}\, .
\end{gather}

Consider the Borel sets
\begin{gather}
U_+ = \{ x \in B_1 \setminus \cK_0 : \Theta^{*, k}(\mu, x) > 0 \} \quad \text{and}\quad \cK_\infty = \{x \in B_1 \setminus \cK_0 : \Theta^{*, k}(\mu, x) = \infty \}\, .
\end{gather}
Since $\mu \llcorner (B_1 \setminus \cK_0)<\infty$, we know by \eqref{eqn:theta-is-the-same} and Lemma \ref{lem:density-to-ineq} that $\haus^k(\cK_\infty) = 0$, and by Radon-Nikodym theorem 
\begin{gather}
0 < \Theta^{*, k}(\mu \llcorner (U_+ \setminus (\cK_0 \cup \cK_\infty)), x) < \infty \quad \text{for $\mu$-a.e. $x \in U_+ \setminus (\cK_0 \cup \cK_\infty)$}\, .
\end{gather}

Define
\begin{gather}
\mu_\ell = \mu \llcorner (U_+ \setminus (K_0 \cup K_\infty))\, .
\end{gather}
By monotonicity of $\beta$, 
\begin{gather}
\int_0^2 \beta_{\mu_\ell, 2, 0}^k(z, r)^2 dr/r < \infty \quad \mu_\ell-a.e.  \,z\, ,
\end{gather}
and therefore by Theorem \ref{thm:main_rect} $\mu_\ell$ is rectifiable.

Set $\cK_h = \cK_0 \cup \cK_\infty$, and $\mu_h = \mu \llcorner \cK_h$.  Since $\cK_0$ is rectifiable, and $\haus^k(\cK_\infty) = 0$, then $\cK_h$ is rectifiable also.  The volume bounds of part A) follow since $\haus^k(\cK_\infty) = 0$.  Using Theorem \ref{thm:main} part C) with Lemma \ref{lem:density-to-ineq}, we have
\begin{gather}
\haus^k(\cK_h \cap U) = \haus^k(\cK_0 \cap U) \leq c(k) \eps^{-1} \mu(U) \quad \forall \text{ open $U$}\, .
\end{gather}

Last, set $U_0$ to be the points $x \in B_1 \setminus \cK_0$ for which $\Theta^{*, k}(\mu, x) = 0$.  By definition,
\begin{gather}
B_1 \subset (\cK_0 \cup \cK_\infty) \cup U_+ \cup U_0\, ,
\end{gather}
thus we can set $\mu_0 = \mu \llcorner U_0$, and this completes the decomposition.  Since $\mu_0 \leq \mu \llcorner (B_1 \setminus \cK_0)$, the mass bound of Theorem \ref{thm:main-thm-L1} part C) is immediate.  The density assertion in C) is by construction.
\end{proof}

In a similar way, to prove Theorem \ref{cor:main-thm-L1-scales}, we decompose $\mu$ as in Theorem \ref{thm:main-thm-L1} \emph{at every scale}.
\begin{proof}[Proof of Theorem \ref{cor:main-thm-L1-scales}]
Let $\cK_0$, $\cK_\infty$, $U_+$, and $U_0$ be defined as in the proof of Theorem \ref{thm:main-thm-L1}.  Thus, $\cK_0$ is closed, rectifiable, and satisfies
\begin{gather}
\haus^k(\cK_0 \cap B_r(x)) \leq c(n) r^k \quad \forall x \in B_1, 0 < r \leq 1 , \quad\text{and} \quad \mu(B_1 \setminus \cK_0) \leq c(n)(\eps + M)\, .
\end{gather}
Moreover, by Theorem \ref{thm:main} part C) and Lemma \ref{lem:density-to-ineq}, 
\begin{gather}
\haus^k \llcorner \cK_0 \leq c(k)\eps^{-1} \mu \quad\text{on open sets}\, .
\end{gather}
Recall that $\haus^k(\cK_\infty) = 0$.

Define the intermediate measure
\begin{gather}
\mu_1 = \mu \llcorner (B_1 \setminus (\cK_0 \cup \cK_\infty))\, .
\end{gather}
Then $\mu_1$ is Radon.

Let $\{B_{s_i}(x_i)\}_i$ be a countable collection of balls, so that centers and radii $\{(x_i, s_i)\}_i$ are dense in $B_1 \times (0, 1)$.  For each $i$, apply Theorem \ref{thm:main} to $\mu_1$ at scale $B_{s_i}(x_i)$, with covering pair $\cC = \cC_0 = B_{s_i}(x_i)$ and $\Gamma = M$.

This produces a sequence of closed, $k$-rectifiable sets $\cK_i$ such that
\begin{gather}\label{eqn:upper-ahlfors-by-cutting}
\mu_1(B_{s_i}(x_i) \setminus \cK_i) \leq c(n)(\eps + M) s_i^k \quad\text{and}\quad \haus^k \llcorner \cK_i \leq c(k)\eps^{-1} \mu_1 \llcorner \cK_i \, .
\end{gather}
This last statement follows from Theorem \ref{thm:main} part C), Lemma \ref{lem:density-to-ineq}, and finiteness of $\mu_1$.  Since each $\cK_i$ is Borel, we have
\begin{gather}
\haus^k \llcorner \cup_i \cK_i \leq c(k) \eps^{-1} \mu_1 \llcorner \cup_i \cK_i \quad\text{on Borel sets}\, .
\end{gather}
We deduce by monotonicity of $\beta$ that the measure $\eta = \haus^k \llcorner \cup_i \cK_i$ satisfies
\begin{gather}
\int_{B_r(x)} \int_0^2 \beta_{\eta, 2, 0}^k(z, r)^2 \frac{dr}{r} d\eta(z) \leq c(k)M^2/\eps^2 \quad \forall x \in B_1, 0 < r \leq 1\, .
\end{gather}

Since $\mu_1$ is finite, $\haus^k(\cup_i \cK_i) < \infty$, and so $\Theta^{*,k}(\eta, x) \leq 1$ at $\eta$-a.e. $x$.  Theorem \ref{thm:main_upperahlfors} part U) implies that
\begin{gather}
\eta(B_r(x)) \leq c(n)(1 + M/\eps) r^k \quad \forall x \in B_1, 0 < r \leq 1\, .
\end{gather}
Set 
\begin{gather}
\cK_h = \cK_0 \cup \cK_\infty \cup \bigcup_i \cK_i\, ,\quad \mu_\ell = \mu_1 \llcorner (U_+ \setminus \cup_i \cK_i), \quad \mu_0 = \mu_1 \llcorner (U_0 \setminus \cup_i \cK_i).
\end{gather}
With this definition, $\mu_\ell$ is $k$-rectifiable, and $\Theta^{*, k}(\mu, x) = 0$ at $\mu_0$-a.e. $x$, as in the proof of Theorem \ref{thm:main-thm-L1}.  Upper-Ahlfors-regularity of $\mu_\ell + \mu_0$ follows by \eqref{eqn:upper-ahlfors-by-cutting}.
\end{proof}

We conclude this section with the proofs of Theorems \ref{thm:main_discrete} and \ref{thm:main_hausdorff_massbounds}, which are immediate consequences of previous results.

\begin{proof}[Proof of Theorem \ref{thm:main_discrete}]
Define the covering pair $\cC = \cC_+ = \{x_i\}_i$, with $r_{x_i} = r_i$.  Then Theorem \ref{thm:main_discrete} is an immediate Corollary of Theorem \ref{thm:main_density}.
\end{proof}

\begin{proof}[Proof of Theorem \ref{thm:main_hausdorff_massbounds}]
By lower-semi-continuity of $\beta$ (Lemma \ref{thm:lsc}) there is no loss in assuming that $S$ is closed.  As in the Proof of Corollary \ref{cor:teaser2}, hypothesis \eqref{eq_hausdorff_integral_bounds} implies $S$ is $\sigma$-finite.  Therefore $\Theta^*(\haus^k \llcorner S, x) \leq 1$ at $\haus^k$-a.e. $x \in S$.  In order to conclude, simply choose the covering pair $\cC = \cC_0 = S$, and apply Theorem \ref{thm:main_density}.
\end{proof}

\bibliographystyle{aomalpha}
\bibliography{ENV_Reifenberg}

\end{document}